\documentclass[a4paper,11pt,fleqn]{article}
\usepackage[official]{eurosym}
\usepackage{amsmath}
\usepackage{amssymb}
\usepackage{theorem}
\usepackage[usenames,dvipsnames]{pstricks}
\usepackage{euscript}
\usepackage{graphicx}
\usepackage{charter}

\usepackage{comment}
 \usepackage{array,multirow}
 \usepackage{float}

\newcommand{\email}{}%[1]{\href{mailto:#1}{\nolinkurl{#1}}}
\definecolor{labelkey}{rgb}{0,0.08,0.45}
\definecolor{refkey}{rgb}{0,0.6,0.0}
\definecolor{Brown}{rgb}{0.45,0.0,0.05}
\definecolor{dgreen}{rgb}{0.00,0.49,0.00}
\definecolor{dblue}{rgb}{0,0.08,0.75}
\RequirePackage[colorlinks,hyperindex]{hyperref} %click pdf
\hypersetup{linktocpage=true,citecolor=dblue,linkcolor=dgreen}
\usepackage{cleveref}
\PassOptionsToPackage{normalem}{ulem}
\usepackage{ulem}

\usepackage{array}          % For better table formatting
\usepackage{booktabs}       % For nicer table lines (toprule, midrule, etc.)

\definecolor{LightGray}{RGB}{230, 234, 237}
\definecolor{DarkGray}{RGB}{205, 210, 216}
\definecolor{LightBlue}{RGB}{220,230,245}
\definecolor{DarkBlue}{RGB}{47,95,127}
\definecolor{LightGreen}{RGB}{220,245,230}
\usepackage{colortbl}       % For \rowcolors

\newcommand{\Rset}{\ensuremath{\mathbb{R}}}
\newcommand{\Nset}{\ensuremath{\mathbb{N}}}
\newcommand{\prox}{\ensuremath{\operatorname{prox}}}

\newcommand{\Id}{\ensuremath{\operatorname{Id}}}

\newcommand{\Argd}{\ensuremath{{\operatorname{Arg}}}}
\newcommand{\Argmind}[1]{\ensuremath{\underset{#1}{\operatorname{Argmin}}\;}}
\newcommand{\argmind}[1]{\ensuremath{\underset{#1}{\operatorname{argmin}}\;}}
\newcommand{\minimize}[1]{\ensuremath{\underset{#1}{\operatorname{minimize}}\;}}
\newcommand{\Idt}{\ensuremath{\operatorname{Id}}}
\newcommand{\Dics}{\ensuremath{D}}
\newcommand{\Dica}{\ensuremath{\Gamma}}
\newcommand{\NN}{\ensuremath{G}}
\newcommand{\Am}{\ensuremath{\mathcal{A}}}
\newcommand{\Sm}{\ensuremath{\mathcal{S}}}
\newcommand{\feedAF}{\ensuremath{T^{\Am}}}
\newcommand{\feedSF}{\ensuremath{T^{\Sm}}}

\newcommand{\Ltrain}{\ensuremath{L_{\text{train}}}}
\newcommand{\Lpnp}{\ensuremath{L_{\text{PnP}}}}
\newcommand{\Lden}{\ensuremath{L_{\text{den}}}}

\newcommand{\pgx}[1]{\ensuremath{u^\ddagger_{#1}}}

\newtheorem{theorem}{Theorem}[section]

\newtheorem{proposition}[theorem]{Proposition}
\theoremstyle{plain}{\theorembodyfont{\rmfamily}%
}
\theoremstyle{plain}{\theorembodyfont{\rmfamily}%
}
\theoremstyle{plain}{\theorembodyfont{\rmfamily}%
}
\theoremstyle{plain}{\theorembodyfont{\rmfamily}%
}
\theoremstyle{plain}{\theorembodyfont{\rmfamily}%
\newtheorem{example}[theorem]{Example}}
\theoremstyle{plain}{\theorembodyfont{\rmfamily}%
\newtheorem{remark}[theorem]{Remark}}
\theoremstyle{plain}{\theorembodyfont{\rmfamily}%
}
\theoremstyle{plain}{\theorembodyfont{\rmfamily}%
}
\theoremstyle{plain}{\theorembodyfont{\rmfamily}%
\newtheorem{model}[theorem]{Model}}

\numberwithin{equation}{section}
\begin{document}

\title{\sffamily Analysis and Synthesis Denoisers \\
for Forward-Backward Plug-and-Play Algorithms
\thanks{This work was funded by the Royal Society of Edinburgh and the EPSRC grant EP/X028860 and by grants from Digiteo France.}}
\author{
Matthieu Kowalski$^\bullet$, Beno\^{i}t Mal\'ezieux$^\diamond$, Thomas Moreau$^\diamond$, Audrey Repetti$^{\dagger\star}$ 
\\[5mm]
\small
\small $^\bullet$ Laboratoire Interdisciplinaire des Sciences du Num\'eriques, Inria, Universit\'e Paris-Saclay, CNRS, \\
\small Gif-sur-Yvette, France \\
\small $^\diamond$ Inria, Universit\'e Paris-Saclay, CEA, Palaiseau, France \\
\small $^\dagger$ School of Mathematics and Computer Sciences and School of Engineering and Physical Sciences, \\
\small Heriot-Watt University, Edinburgh, UK\\
\small $^\star$ Maxwell Institute for Mathematical Sciences, Edinburgh, UK \\[3mm]
\small \email{matthieu.kowalski@universite-paris-saclay.fr}, \email{benoit.malezieux@free.fr}, \\ 
\small \email{thomas.moreau@inria.fr}, \email{a.repetti@hw.ac.uk}
}
\date{}

\maketitle

\vskip 8mm

\begin{abstract}
    In this work we study the behavior of the forward-backward (FB) algorithm when the proximity operator is replaced by a sub-iterative procedure to approximate a Gaussian denoiser, in a Plug-and-Play (PnP) fashion. 
    Specifically, we consider both analysis and synthesis Gaussian denoisers within a dictionary framework, obtained by unrolling dual-FB iterations or FB iterations, respectively.
    We analyze the associated minimization problems as well as the asymptotic behavior of the resulting FB-PnP iterations. In particular, we show that the synthesis Gaussian denoising problem can be viewed as a proximity operator. For each case, analysis and synthesis, we show that the FB-PnP algorithms solve the same problem whether we use only one or an infinite number of sub-iteration to solve the denoising problem at each iteration. To this aim, we show that each "one sub-iteration" strategy within the FB-PnP can be interpreted as a primal-dual algorithm when a warm-restart strategy is used. 
    We further present similar results when using a Moreau-Yosida smoothing of the global problem, for an arbitrary number of sub-iterations.
    Finally, we provide numerical simulations to illustrate our theoretical results. In particular we first consider a toy compressive sensing example, as well as an image restoration problem in a deep dictionary framework.
\end{abstract}

{\bfseries Keywords.}
Plug-and-Play, forward-backward algorithm, unrolling, deep dictionary learning, inverse imaging problems

{\bfseries MSC.}
    90C59, %    Approximation methods and heuristics in mathematical programming
    65K10, %    Numerical optimization and variational techniques
    68T07, %    Artificial neural networks and deep learning
    68U10, %    	Computing methodologies for image processing
    94A08  % Image processing (compression, reconstruction, etc.) in information and communication theory

%\newpage
%\setcounter{page}{1}

\section{Introduction}\label{Sec:intro}
%
%% Intro inverse problem
Linear inverse problems play a pivotal role in various scientific disciplines, including
imaging~\cite{ribes2008linear}, neurosciences~\cite{GKH12}, and astrophysics~\cite{St16}.
In these scenarios, an unknown signal $\overline{x} \in \Rset^N$ is observed through a degraded linear system given by
\begin{equation}\label{pb:invpb}
    y = A \overline{x} + \varepsilon w,
\end{equation}
where $y \in \Rset^M$ represents the degraded observations, $A \colon \Rset^N \to \Rset^M$ models a linear measurement operator, $w \in \Rset^M$ is a realization of an \emph{i.i.d.} standard normal random variable, and $\varepsilon>0$. The inverse problem~\eqref{pb:invpb} aims to find an estimate $\widehat{x} \in \Rset^N$ of $\overline{x}$ from the degraded measurements $y$.
This problem is often challenging due to issues such as under-sampling and noise, rendering it ill-posed and/or ill-conditioned.

% Variational MAP formulation
To address these challenges, practitioners commonly employ \textit{prior} knowledge to guide the selection of a
plausible solution. A widely adopted approach involves a maximum \textit{a posteriori} (MAP) strategy, defining
$\widehat{x}$ as a minimizer of a penalized least squares objective:
\begin{equation}\label{pb:min_gen}
    \text{find }
    \widehat{x} \in \Argmind{x \in \Rset^N} \frac12 \| Ax - y \|^2 + p(x),
\end{equation}
where $p \colon \Rset^N \to (-\infty, +\infty]$ is a convex, lower semi-continuous, proper function, representing a penalization term that incorporates prior information on the target solution. The choice of the prior is crucial for both reconstruction performance and computational complexity. Notably, functions summarizing signal structures based on sparsity have been extensively studied in the literature~\cite{elad2010sparse, bach2012optimization, foucart2013invitation}. 
Proximal algorithms~\cite{combettes2011proximal, komodakis2015playing} are efficient to solve the resulting minimization problem~\eqref{pb:min_gen}. They are scalable and versatile, offering convergence guarantees and maintaining their status as the state-of-the-art for solving inverse problems for more than two decades. A celebrated proximal algorithm, extensively used in the literature for solving~\eqref{pb:min_gen}, is the forward-backward (FB) algorithm, also known as proximal-gradient or ISTA~\cite{Combettes2005}. This scheme alternates at each iteration between a gradient step on the differentiable least squares function and a proximal step on the non-smooth function $p$. This algorithm reads
\begin{equation}    \label{algo:FB}
    \begin{array}{l}
        x_0 \in \mathbb R^N \\
       \text{for } k = 0, 1, \ldots \\
       \left\lfloor
       \begin{array}{l}
        x_{k+1} = \prox_{\tau p} \big( x_k - \tau A^*(A x_k - y) \big),
       \end{array}
       \right.
    \end{array}
\end{equation}
where $\prox_{\tau p}$ is the so-called proximity operator of $p$ (see \eqref{def:prox} for its definition), and $A^*$ denotes the adjoint operator of $A$. In \eqref{algo:FB}, $\tau>0$ is a step-size chosen to ensure the convergence of sequence ${(x_k)}_{k\in \Nset}$ to a solution to~\eqref{pb:min_gen}.
In~\eqref{algo:FB}, one has to compute the proximity operator of $p$ (defined in~\eqref{def:prox}) at each iteration. This operator can have an explicit formula for many simple choices of functions $p$ (see e.g.~\cite{combettes2011proximal}). However this is not true in many cases of interest, and the proximity operator must be approximated numerically with sub-iterations.
These sub-iterations can become computationally expensive in practice as they need to accurately approximate $\prox_p$.
Recently, to improve the reconstruction quality and avoid these sub-iterations, proximal algorithms have also been paired with deep learning techniques for solving inverse imaging problems.
Examples include deep dictionary learning (DDL), Plug-and-Play (PnP) algorithms, and unfolded neural networks (see, e.g., \cite{SEM19, brifman2016turning, kamilov2023plug, martin2024pnp, GL10, JP2020} and references therein).
However, providing convergence guarantees for these methods is not straightforward.

A common prior's choice is based on the sparsity of the unknown signal in a certain basis. In this context, the prior $p$ takes the form of
\begin{equation}\label{eq:def_p}
    (\forall x \in \Rset^N)\quad
    p(x) = \lambda g(\Dica x),
\end{equation}
where $\lambda>0$ is a regularization parameter balancing data fidelity and regularization terms, $\Dica \colon \Rset^N \to \Rset^S$ is a linear operator modeling a sparsifying transformed domain, and $g\colon \Rset^S \to \Rset$ is a convex, lower semi-continuous, proper function. The function $g$ is then chosen for promoting sparsity in the transformed domain (e.g., $\ell^1$ norm).
Notable choices for the sparsifying operator $\Dica$ include the Fourier, DCT, or Wavelet transforms~\cite{Mallat2008} as well as the Total Variation (TV) regularization~\cite{rudin1992nonlinear, chambolle2010introduction}, depending on the properties of the target solution.
The structure of the sought signal, through the operator $\Dica$, can also be learned from a ground truth dataset in a supervised setting, leading to dictionary learning (DL) approaches~\cite{AEB06, MBPS09}. 
For such sparsity-based priors, there exists no close-form formula and one typically needs to resort to expansive sub-iterations as well as inexact proximal methods allowing for computational errors to approximate the proximity operator (see, e.g., \cite{Combettes2005}).
In this paper, we investigate the properties of such approaches, reducing the sub-iteration computational burden. We will in particular study convergence properties of the FB iterations~\eqref{algo:FB} when the proximity operator is (very) roughly approximated.

\subsection*{Contributions}
We adopt a PnP framework, where the proximity operator of $p$ in~\eqref{algo:FB} is replaced by a denoiser.
Building our PnP algorithm with FB iterations, we investigate the behavior of the following FB-PnP scheme
\begin{equation}    \label{algo:FB-PnP}
    \begin{array}{l}
        x_0 \in \mathbb R^N \\
       \text{for } k = 0, 1, \ldots \\
       \left\lfloor
       \begin{array}{l}
        x_{k+1} = G \big( x_k - \tau A^*(A x_k - y) \big),
       \end{array}
       \right.
    \end{array}
\end{equation}
where $\tau>0$ and $G\colon \mathbb R^N \to \mathbb R^N$ is a sparsity-based denoising operator.
Specifically, we consider $G$ to be either an analysis or a synthesis dictionary-based denoiser. 
Let $v \in \Rset^N$ be a noisy image.
On the one hand, the Analysis Denoiser (AD) is built as
\begin{equation}\label{pb:AF-intro}
    % (\forall v \in \mathbb R^N)\quad
    \text{find } \quad G_{\Dica} (v) \approx
    \prox_{g_\lambda \circ \Dica}(v) =
    \argmind{x \in \Rset^N} \frac12 \| x - v \|^2 +  g_\lambda(\Dica x),
\end{equation}
where $g_\lambda \equiv \lambda g$ and $\Dica \colon \mathbb R^N \to \mathbb R^S$. 
On the other hand, the Synthesis Denoiser (SD) is built as
\begin{equation}\label{pb:SF-intro}
    % (\forall v \in \mathbb R^N)\quad
    \text{find} \quad G_{\Dics}(v) 
    = \Dics z^\dagger_{\Dics}
    \;\text{ where }\; 
    z^\dagger_{\Dics} \approx \argmind{z \in \Rset^S} \frac12 \| \Dics z-v \|^2 +  g_\lambda(z),
\end{equation}
where $\Dics \colon \mathbb R^S \to \mathbb R^N$.
In general these two classes of denoisers do not have a closed-form formula (unless for specific cases, e.g., when $\Dica$ is (semi)-orthogonal~\cite{EMR07}), thus requiring sub-iterations to compute their output.
We then aim to investigate the behavior of ${(x_k)}_{k\in \mathbb N}$ when $G_{\Dica}$ and $G_{\Dics}$ are computed with FB-based iterations following an unrolling framework.

\smallskip

Our first contribution is to analyze the minimization problems associated with the proposed AD and SD in~\eqref{pb:AF-intro}-\eqref{pb:SF-intro} when $\Dica$ and $\Dics$ are fixed dictionaries. In particular we show that~\eqref{pb:SF-intro} corresponds to a proximity operator, and that the associated problem is equivalent to a synthesis formulation of problem~\eqref{pb:min_gen}-\eqref{eq:def_p} (see Section~\ref{Ssec:SF-equiv}).
Second, we investigate the behavior of algorithm~\eqref{algo:FB-PnP}, where $G_{\Dica}$ and $G_{\Dics}$ are obtained running $1$ sub-iteration of a FB-based scheme for solving~\eqref{pb:AF-intro}-\eqref{pb:SF-intro}, using a warm-restart strategy (see Sections~\ref{Sec:AD-PnP} and~\ref{Sec:SD-PnP}, respectively). Specifically, we show that the iterations using the AD within a FB-PnP with only $1$ sub-iteration of a dual FB algorithm \cite{combettes2010dualization} are equivalent to those of the Loris-Verhoeven primal-dual algorithm~\cite{loris2011generalization}. Similarly, we show that the iterations using the SD within a FB-PnP with only $1$ sub-iteration of a FB algorithm are equivalent to the FB iterations for solving the full synthesis formulation of problem~\eqref{pb:min_gen}-\eqref{eq:def_p}.
We further investigate the behavior of a similar algorithm to~\eqref{algo:FB-PnP} for solving a Moreau envelope smoothing of problem~\eqref{pb:min_gen}-\eqref{eq:def_p}. In this context, we show that using a warm-restart strategy yields an algorithm which converges toward a solution of the smooth approximated problem, if the operators $G_{\Dica}$ and $G_{\Dics}$ satisfy some sufficient decrease conditions (see Section~\ref{sec:bi-level-approximate}).
Finally, we illustrate our theoretical study on two simulation examples: a toy compressive sensing example, and an example on image restoration. In particular,
we show through simulations that the proposed framework is well suited to adopt a DDL approach for learning dictionaries $\Dica$ or~$\Dics$ (see Section~\ref{Sec:exp}). \\
A summary of the methods investigated in this work is provided in Table~\ref{tab:algo_sum}, and their associated theoretical conditions for convergence are summarized in Table~\ref{tab:algo_sum_cvg}.

\begin{table}
    \resizebox{\textwidth}{!}{
    \centering
    \renewcommand{\arraystretch}{2}
    \setlength{\tabcolsep}{3pt}
    % \small
    \rowcolors{2}{DarkGray}{LightGray}
    \arrayrulecolor{white}
    \begin{tabular}{>{\centering\arraybackslash} m{3.9cm}|>{\centering\arraybackslash} m{2.3cm}|>{\centering\arraybackslash} m{6.6cm}|>{\centering\arraybackslash} m{3.3cm}|>{\centering\arraybackslash} m{7.2cm}|>{\centering\arraybackslash} m{1.7cm}}
        %\toprule
        \rowcolor{DarkBlue}
        & \textcolor{white}{\textbf{Inverse problem }} 
        & \textcolor{white}{\textbf{Variational problem }}
        & \textcolor{white}{\textbf{Prox equivalence}}
        & \textcolor{white}{\textbf{Operator / Algorithm}}
        & \textcolor{white}{\textbf{Section}} 
        \\
        \midrule
        % -------------------------------
        % -------------------------------
        \textbf{Analysis Denoiser (AD)}
        &
            \[v = \overline{x} + \varepsilon_d w\]
        &
            \[
            \left\lbrace\hspace{-0.1cm}
            \begin{aligned}
                &\ x_{\Dica}^\dagger = v - \Dica^* u_{\Dica}^\dagger \\
                &\ u_{\Dica}^\dagger = \argmind{u \in \mathbb R^S} \frac12 \|\Dica^*u - v \|^2 + g^*_\lambda(u)
            \end{aligned}
            \right.
            \]
        &
            \[x_{\Dica}^\dagger = \prox_{g_\lambda \circ \Dica}(v)\]
        &
            \[ 
            \begin{aligned}
              \widetilde{\NN}_{L, \lambda, v}^{\Am} &  = \underset{\text{$L$ compositions}}{\underbrace{T_{\lambda, v}^{\Am} \circ \dots \circ T_{\lambda, v}^{\Am}}} \\
               T_{\lambda, v}^{\Am} \colon u \in \mathbb R^S \mapsto &  \prox_{ \sigma g_\lambda^*} \big( u - \sigma {\Dica}( \Dica^* u - v) \big)
            \end{aligned}
            \]
        &   \ref{Ssec:AD-struct}
        \\
        \hline
        % -------------------------------
        \textbf{AD within FB-PnP}
        &
            \[ y = A\overline{x} + \varepsilon w \]
        &
            \[ \widehat{x}_{\Am} \in \Argmind{x \in \mathbb R^N} \frac12 \|Ax - y \|^2 + g_\lambda(\Dica x) \]
        &
            --
        &
            \vspace{-0.7cm}
            \[
            \begin{array}{l}
               \text{for } k = 0, 1, \ldots \\
               \left\lfloor
               \begin{aligned}
                    &\ v_{k}  =  x_k - \tau A^*(A x_k - y) , \\
                    &\ u_{k+1}  = \widetilde{\NN}^{\Am}_{L, \tau \lambda, v_k}(u_k), \\
                    &\ x_{k+1}  = v_k - \Dica^* u_{k+1},
                \end{aligned}
                   \right.
            \end{array}
            \]  
        &   \ref{Ssec:FB-PnP-AD}
        \\
        \midrule
        % -------------------------------
        \textbf{Synthesis Denoiser (SD)}
        &
            \[ v = \overline{x} + \varepsilon_d w \]
        &
            \[ \left\lbrace\hspace{-0.1cm}
            \begin{aligned}
                &\ x_{\Dics}^\dagger = \Dics z_{\Dics}^\dagger \\
                &\ z_{\Dics}^\dagger = \argmind{z \in \mathbb R^S} \frac12 \|\Dics z - v \|^2 + g_\lambda(u)
            \end{aligned}\right.
            \]
        &
            \[x_{\Dics}^\dagger = \prox_{\Dics \rhd g_\lambda }(v) \]
        &
            \[ 
            \begin{aligned}
            \widetilde{\NN}_{L, \lambda, v}^{\Sm} & = \underset{\text{$L$ compositions}}{\underbrace{T_{\lambda, v}^{\Sm} \circ \dots \circ T_{\lambda, v}^{\Sm}}} \\
            T_{\lambda, v}^{\Sm} \colon u \in \mathbb R^S \mapsto & \prox_{ \zeta g_\lambda} \big( z - \zeta {\Dics^*}( \Dics z - v) \big)
            \end{aligned}
            \]
        &   \ref{Ssec:SF-equiv} \& \ref{Ssec:SD-struct}
        \\
        \hline
        % -------------------------------
        \textbf{SD within FB-PnP}
        &
            \[ y = A\overline{x} + \varepsilon w \]
        &
            \[ \widehat{x}_{\Sm} \in \Argmind{x \in \mathbb R^N} \frac12 \|Ax - y \|^2 + \Dics \rhd g_\lambda( x)\] 
            \linebreak
            or 
            \linebreak
            \[
            \left\lbrace\hspace{-0.1cm}
            \begin{aligned}
            &\ \widehat{x}_{\Sm} = \Dics \widehat{z}_{\Sm} \\
            &\ \widehat{z}_{\Sm} \in \Argmind{z \in \mathbb R^S} \frac12 \|A \Dics z - y \|^2 + g_\lambda( z)
            \end{aligned}\right.
            \]
        &
            --
        & 
            \vspace{-0.7cm}
            \[ 
            \begin{array}{l}
               \text{for } k = 0, 1, \ldots \\
               \left\lfloor
               \begin{aligned}
                    &\ v_{k} =  x_k - \tau A^*(A x_k - y) , \\
                    &\ z_{k+1} = \widetilde{\NN}^{\Sm}_{L, \tau \lambda, v_k}(z_k), \\
                    &\ x_{k+1} = \Dics z_{k+1},
               \end{aligned}
               \right.
            \end{array}
            \]
        &   \ref{Ssec:SD-FBPnP}
        \\
        \midrule
        % -------------------------------
        % -------------------------------
        \textbf{Unfolded denoiser within 
FB-PnP}
        &
            \[ y = A\overline{x} + \varepsilon w \]
        &
            \[ x^\ddagger = \argmind{x \in \mathbb{R}^{\widetilde{N}}} 
    \widetilde{f}(x) + {^\mu\widetilde{g}_\lambda}(x) \]
        &
            --
        &
            \vspace{-0.7cm}
            \[ 
            \begin{array}{l}
            \text{for } k = 0, 1, \ldots \\
            \left\lfloor
            \begin{aligned}
              &\ x_{k+1}  =  x_k - \tau \nabla \widetilde{f}( x_k) - \tau\mu (x_k - u_k) , \\
              &\ u_{k+1} = \widetilde{\NN}_{L, \lambda\mu^{-1}} (x_{k+1}, u_k),
            \end{aligned}
            \right.
            \end{array}
            \] 
            \linebreak
            $\widetilde{\NN}_{L, \lambda\mu^{-1}}(x_{k+1}, u_k) \approx
            \prox_{\widetilde{g}_{\frac{\lambda}{\mu}}}(x_{k+1})$
            \vspace{0.3cm}
        &   \ref{sec:bi-level-approximate}
        \\
        % -------------------------------
        % -------------------------------
        %\bottomrule
    \end{tabular}
    }
    \caption{\label{tab:algo_sum}
    Summary of the different problems and associated algorithms investigated in this work. 
    Notation are the same as those used throughout the article. 
    Conditions for convergence of the FB-PnP iterations are summarized in Table~\ref{tab:algo_sum_cvg}.
    }
\end{table}

\begin{table}[ht]
    \centering

    \rowcolors{2}{DarkGray}{LightGray}
    \arrayrulecolor{white}
    {\footnotesize 
    \begin{tabular}{>{\centering\arraybackslash}m{0.23\textwidth} | >{\centering\arraybackslash}m{0.43\textwidth} | 
                    >{\centering\arraybackslash}m{0.24\textwidth}}
 
    %\toprule

    \rowcolor{DarkBlue}
    \textcolor{white}{\rule{0pt}{13pt}\textbf{}} &
    \textcolor{white}{\textbf{Convergence conditions}} &
    \textcolor{white}{\textbf{Convergence}} 
    \\
    \midrule
    \textbf{
    Analysis denoiser (AD) within FB-PnP} &
    \phantom{pour centrer}
    \begin{itemize}
      \item $L \to \infty$ or $L = 1$ 
      \item $g$ convex, lsc, proper
      \item $0 < {\tau} < 2 \|A\|_S^{-2}$
      \item $0< {\sigma} < 2 \|\Gamma\|_S^{-2}$
    \end{itemize}
    \phantom{verticalement}
    & 
    Theorem~\ref{thm:cvgce-pnp-limit-AD} ($L\to\infty$) and Theorem~\ref{prop:analysis_warm_restart}$^{\ast}$ ($L=1$)
    \\
    \hline
    \textbf{
    Synthesis denoiser (SD) within FB-PnP}
    &
    \phantom{pour centrer}
    \begin{itemize}
        \item $L \to \infty$
        \item $g$ convex, lsc, proper
        \item $0< {\tau} < 2 \|A\|_S^{-2}$
        \item $0 < {\zeta} < 2 \|D\|_S^{-2}$
    \end{itemize}
    \phantom{verticalement}
    & 
    Theorem~\ref{thm:cvgce-pnp-limit-SD} 
    \\
    \hline
    \textbf{
    Synthesis denoiser (SD) within FB-PnP}
    &
    \phantom{pour centrer}
    \begin{itemize}
        \item $L = 1$
        \item $g$ convex, lsc, proper
        \item $0< \tau\zeta < 2 \|AD\|_S^{-2}$
    \end{itemize}
    \phantom{verticalement}
    & 
    Theorem~\ref{thm:cvg-SF-L1}$^{\ast}$
    \\
    \hline
    \textbf{
    Unfolded denoiser within FB-PnP}
    &
    \phantom{pour centrer}
    \begin{itemize}
      \item $L$ chosen such that \eqref{thm:bilev:ass} is satisfied
      \item $\tilde{g}$ convex, lsc, proper
      \item $\tilde{f}$ convex, Lipschitz-differentiable
      \item $\tau$ chosen according to Theorem~\ref{thm:bilevel}\ref{thm:bilevel:cond:ii}
    \end{itemize}
    \phantom{verticalement}
    & 
    Theorem~\ref{thm:bilevel}$^{\ast}$
    \\
    %\bottomrule
    \end{tabular}
    }
    \caption{\label{tab:algo_sum_cvg}
    Summary of the theoretical conditions necessary to ensure convergence of the algorithms investigated in this work, and summarized in Table~\ref{tab:algo_sum}. The symbol $^{\ast}$ is used to highlight new results.}
\end{table}

\smallskip

\subsection*{Outline}

The remainder of the paper is organized as follows. 
Section~\ref{Sec:AD-PnP} is dedicated to the analysis-based approach. We first present the denoiser based on the sparse coding
of the analysis coefficient using a given dictionary. 
Then, the unrolled version of this denoiser is embedded in a PnP framework to solve general linear inverse problems for which
we provide a convergence analysis. Section~\ref{Sec:SD-PnP} gives a similar study of the synthesis-based approach.
In Section~\ref{sec:bi-level-approximate}, we use a smooth regularized version of the original problem, to study the PnP algorithm with denoisers built with a finite number of iterations.
Finally, Section~\ref{Sec:exp} illustrates the convergence properties of the two
PnP approaches on a toy compressive sensing example, as well as to an image restoration problem within a DDL framework.

\subsection*{Notation}%\label{Ssec:Back:notation}
An element of $\Rset^N$ is denoted by ${x = {(x^{(n)})}_{1 \le n \le N}}$.
The standard Euclidean norm is denoted $\|\cdot\|$, the $\ell_p$ norm (for $p>0$) is denoted $\| \cdot \|_p$, and $\| \cdot \|_S$ denotes the spectral norm.
We denote $\Gamma_0(\Rset^N)$ the set of proper, lower semi-continuous convex functions from $\Rset^N$ to $(-\infty,
+\infty]$. 
We use $\mathbb{N}^* := \mathbb{N} \setminus \{0\}$ to denote the set of strictly positive integers, and $\Rset^* := \mathbb{R} \setminus \{0\}$ for the set of nonzero real numbers.
For a function $f\in \Gamma_0(\Rset^N)$, we write $\displaystyle \min_{x \in \Rset^N} f(x)$ to denote the minimum value of $f$, and $\argmind{x \in \Rset^N} f(x)$ for a unique minimizer, when it exists. When the minimizer is not unique, we use $\Argmind{x \in \Rset^N} f(x)$ to denote the set of minimizers.
Further, the notation $\minimize{x \in \Rset^N} f(x)$ refers to the task of finding an element in $\Argmind{x \in \Rset^N} f(x)$.

The proximity operator of a function $g\in \Gamma_0(\Rset^N)$ is defined as~\cite{moreau1965proximite}
\begin{equation}    \label{def:prox}
    (\forall v \in \mathbb{R}^N)\quad
    \prox_{ g}(v) = \argmind{x \in \mathbb{R}^N}  g(x) + \frac12 \| x - v \|^2\ .
\end{equation}
Let $\mathcal{C} \subset \Rset^N$ be a closed, non-empty, convex set. We denote by $\iota_\mathcal{C}$ the indicator
function of $C$ defined, for every $x \in \Rset^N$, as $\iota_\mathcal{C}(x)=0$ if $x \in \mathcal{C}$, and
$\iota_\mathcal{C}(x)=+\infty$ otherwise. 
The Fenchel-Legendre conjugate function of $g$ is denoted by $g^* \in \Gamma_0(\Rset^N)$
and is defined as, for every $v\in \Rset^N$, $\displaystyle g^*(v) = \sup_{x \in \Rset^N}  \langle v, x \rangle - g(x)
$. 
The Moreau's identity is given by ${\Idt \mkern-3mu  = \mkern-3mu\prox_g\mkern-3mu +\mkern-3mu \prox_{g^*}}$, where $\Id$ denotes the identity operator. 
The infimal post-composition of $g$ by an operator $U\colon \Rset^S \to \Rset^N$ is defined as 
\begin{equation}\label{def:inf-post-comp}
    \displaystyle U \rhd g \colon  \Rset^N  \to [-\infty, +\infty] 
    \colon  x  \mapsto \inf_{z \in \Rset^S, x = Uz} g(z) .
\end{equation} 
In particular, by definition of the $\inf$ operator, if, for $x \in \Rset^N$, the set $\{z \in \Rset^S | x = Uz \}$ is empty, then $U\rhd g = +\infty$.
An operator $G\colon \Rset^N \to \Rset^N$ is $\beta$-Lipschitz continuous with parameter $\beta>0$ if, for
every $(x,y) \in {(\Rset^N)}^2$, $\|G (x) - G (y) \| \le \beta \| x - y \|$. 

For further background on convex optimization, we refer the reader to~\cite{bauschke2011convex, rockafellar1997convex, rockafellar2009variational} and references therein.

% --------------------------------------------------------------------
% --------------------------------------------------------------------
\section{Analysis denoiser for FB-PnP algorithm}\label{Sec:AD-PnP}
In this section we focus on the analysis denoiser defined in~\eqref{pb:AF-intro}, for a fixed dictionary $\Dica$.
By definition, the proximity operator of $g_\lambda \circ \Dica$ can be interpreted as a MAP estimate for a Gaussian denoising problem,
where the least-squares function corresponds to the data-fidelity term, and $g_\lambda \circ \Dica$ is the penalization function.
During the last decade, this interpretation has been leveraged to develop new hybrid optimization algorithms, dubbed PnP methods, where the proximity operator can be replaced by more powerful denoisers~\cite{brifman2016turning, ryu2019plug} to solve inverse problems of the form of~\eqref{pb:invpb}.
For instance, the PnP formulation of the FB algorithm~\eqref{algo:FB}, called FB-PnP hereafter, is given by~\eqref{algo:FB-PnP}.
More generally, PnP algorithms have been proposed in the literature using a wide range of proximal algorithms, including HQS algorithm, Douglas-Rachford algorithm, or primal-dual algorithms.
Denoisers can be either hand-crafted (e.g. BM3D~\cite{dabov2009bm3d}), or learned (e.g. neural network denoisers~\cite{chan2016plug, romano2017little, rick2017one}). 
Although PnP methods have shown outstanding performances in many applications, they can be unstable in practice.
The main challenge is to guarantee that PnP iterations produce a converging sequence of estimates ${(x_k)}_{k \in \mathbb{N}}$, without sacrificing reconstruction performances. 
In particular, it is well known that PnP algorithms output converging sequences ${(x_k)}_{k \in \mathbb{N}}$ if the denoiser $\NN$ is firmly non expansive, as a consequence of fixed point theory~\cite{bauschke2011convex}.
Recently, they have been extensively studied, in particular when denoisers are neural networks~\cite{pesquet2021learning,
terris2021enhanced, hurault2021gradient, hurault2022proximal, kamilov2023plug, repetti2022dual, bredies2024learning}. 

\smallskip

In this section, we focus on FB-PnP defined in~\eqref{algo:FB-PnP}, where the denoiser is built with sub-iterations based on the FB scheme for approximating~\eqref{pb:AF-intro}, for a fixed dictionary $\Dica$. In the following, we call such a denoiser an analysis denoiser (AD).

We use the subscript $\Dica$ in $x^\dagger_\Dica$ and $u^\dagger_\Dica$ to denote outputs obtained from denoising tasks using an analysis operator
$\Dica$, in the primal and dual domains respectively. Later, in~Section~\ref{Ssec:FB-PnP-AD}, the subscript $\Am$ will be used to refer to the solution of the full inverse problem using an AD built from $\Dica$. This convention highlights the distinction between the denoising phase and the inverse problem phase.

\subsection{Analysis denoiser structure}\label{Ssec:AD-struct}
As mentioned in Section~\ref{Sec:intro}, unless $\Dica$ is (semi)-orthogonal, the AD problem~\eqref{pb:AF-intro} does not have an explicit solution. However, it can be solved by the FB algorithm when applied to the dual problem of~\eqref{pb:AF-intro} (in the sense of Fenchel-Rockafellar, see e.g.~\cite[Chap.~15]{bauschke2011convex}), as proposed in~\cite{combettes2010dualization, combettes2011proximity}.
This problem aims to
\begin{equation}\label{pb:AF-intro-dual}
    \text{find } u^\dagger_{\Dica} 
    {\in \Argmind{u \in \Rset^S}} 
        \frac12 \| \Dica^* u -v \|^2 + g_\lambda^*(u).
\end{equation}
The resulting dual-FB algorithm reads 
\begin{equation}    \label{algo:dualFB-AF}
\begin{array}{l}
\text{for } \ell = 0, 1, \ldots \\
    \left\lfloor
    \begin{array}{l}
        u_{\ell+1} =
        \prox_{ \sigma g_\lambda^*} \big( u_\ell - \sigma {\Dica}(   \Dica^* u_\ell - v) \big)
    \end{array}
    \right.
\end{array}
\end{equation}
where $u_0\in \Rset^S$ and $\sigma>0$ is a step-size. Then the following convergence result can be deduced
from~\cite[Thm.~3.7]{combettes2010dualization}.

\begin{theorem}\label{thm:cv:dualFB}
\sloppy
Assume that $0< \sigma < 2 \| \Dica \|_S^{-2}$.
Then ${(u_\ell)}_{\ell \in \Nset}$ converges to a solution $u^\dagger_{\Dica}$ to the dual problem \eqref{pb:AF-intro-dual} of~\eqref{pb:AF-intro}, and $x^\dagger_{\Dica} = v - \Dica^* u^\dagger_{\Dica}$ is a solution to~\eqref{pb:AF-intro}.
\end{theorem}

Algorithm~\eqref{algo:dualFB-AF} can be used as sub-iterations when included within the global FB algorithm~\eqref{algo:FB-PnP}, to approximate the computation of $\prox_{g_\lambda \circ \Dica}$.
However, for~\eqref{algo:FB-PnP} to still benefit from convergence guarantees, the computation of the proximity operator must be accurate, which can be difficult in practice as this is only achieved at asymptotic convergence of~\eqref{algo:dualFB-AF}. 
In this section, we investigate the behavior of~\eqref{algo:FB-PnP} with the AD, when only one iteration of algorithm~\eqref{algo:dualFB-AF} is computed. 
In the remainder, we use the superscript $\Am$ to identify operators associated with the analysis formulation.

\begin{model}[Analysis denoiser (AD)]\label{mod:AF}
    Let $(v, u_0) \in \Rset^N \times \Rset^S$, $\Dica \colon \Rset^N \to \Rset^S$ be a linear operator, and $\lambda>0$ be a regularization parameter. Let $L \in \Nset^*$.
    Let $\widetilde{\NN}^{\Am}_{L,\lambda, v} \colon \Rset^S \to \Rset^S$ be defined as
    \begin{equation}    \label{DLnet:ana-dualFB-dual}
    \widetilde{\NN}^{\Am}_{L,\lambda, v}(u_0) = \underset{\text{$L$ compositions}}{\underbrace{\feedAF_{\lambda, v} \circ \dots \circ \feedAF_{\lambda, v}}}(u_0)	,
    \end{equation}
    where
    \begin{equation} \label{DLnet:ana-dualFB-layer}
    \feedAF_{\lambda, v} \colon \Rset^S \to \Rset^S \colon u \mapsto %z
    {\prox_{ \sigma g_\lambda^*} \big( u - \sigma {\Dica}^*( \Dica u - v) \big)}
    \end{equation}
    with $\sigma>0$.
    The unrolled AD $\NN^{\Am}_{L,\lambda, v} \colon  \Rset^S \to \Rset^N$ is obtained as follows:
    \begin{equation}    \label{DLnet:ana-dualFB}
    \NN^{\Am}_{L,\lambda, v}( u_0) 
    = v - {\Dica^*} \widetilde{\NN}^{\Am}_{L,\lambda, v}(u_0).
    \end{equation}
\end{model}

A few comments can be made regarding Model~\ref{mod:AF}.
\begin{remark}\
\begin{enumerate}
\item
In~\cite{repetti2022dual, le2023pnn}, the authors have used the unrolled Model~\ref{mod:AF} to design denoising NNs, where the dictionary $\Dica$ is learned and is different at each iteration.  In their work, a slightly different structure of operator $\feedAF_{\lambda, v}$ was used, accounting for a simple convex constraint on the signal domain, and rescaling the dual variable.
\item
It can be noticed that, since the layers of Model~\ref{mod:AF}, defined in~\eqref{DLnet:ana-dualFB-layer}, correspond to FB iterations, it can be formulated as a feed-forward network, acting in the dual domain.  Precisely,  $ \feedAF_{\lambda, v}$ can be reformulated as
\begin{equation*}
(\forall u \in \Rset^S) \quad
 \feedAF_{\lambda, v}(u) =
 \prox_{ \sigma g_\lambda^*} \left( (\Id -  \sigma {\Dica \Dica^*}) u + \sigma {\Dica} v \right),
\end{equation*}
where $\Id -  \sigma {\Dica \Dica^*}$ is a linear operator and $\sigma {\Dica} v$ a bias.
In~\cite{le2023pnn}, the authors describe the layers of Model~\ref{mod:AF} using a primal-dual formulation, where each layer is composed of two feed-forward layers, one acting in the dual domain, and one acting in the primal domain.
\item
As a direct consequence of Theorem~\ref{thm:cv:dualFB}, it can be noticed that, for any $u_0\in \Rset^S$,  if  $0 < \sigma < 2 \|\Dica\|_S^{-2}$, then $\displaystyle \lim_{L \to +\infty} \NN^{\Am}_{L, \lambda, v}(u_0) = x_{\Dica}^\dagger$,
where $x_{\Dica}^\dagger$ is the solution to Problem~\eqref{pb:AF-intro}. % the problem to
% \begin{equation}\label{pb:AF-inner}
%     \text{find } \quad x^\dagger_{\Dica} =
%     \prox_{g_\lambda \circ \Dica}(v) =
%     \argmind{x \in \Rset^N} \frac12 \| x - v \|^2 + g_\lambda(\Dica x).
% \end{equation}
\end{enumerate}
\end{remark}

\subsection{FB-PnP algorithm with approximated AD}\label{Ssec:FB-PnP-AD}
In this section we aim to  
\begin{equation}\label{pb:inv-min-AF}
    \text{find } \widehat{x}_{\Am} \in \Argmind{x \in \Rset^N} \frac12 \| Ax - y \|^2 + g_\lambda(\Dica x).
\end{equation}
We study the asymptotic behavior of the FB-PnP algorithm~\eqref{algo:FB-PnP} with the AD given in Model~\ref{mod:AF}.
This algorithm boils down to
\begin{equation}    \label{algo:FB-pnp-AD}
    \begin{array}{l}
        x_0\in \Rset^N, u_0 \in \Rset^S,\\
       \text{for } k = 0, 1, \ldots \\
       \left\lfloor
       \begin{array}{l}
        {v}_{k} =  x_k - \tau A^*(A x_k - y) , \\
        u_{k+1} = \widetilde{\NN}^{\Am}_{L, \tau \lambda, v_k}(u_k), \\
        x_{k+1} = v_k - \Dica^* u_{k+1}, %\varrho^{\Am}( u_{k+1} ),
       \end{array}
       \right.
    \end{array}
\end{equation}
where $\tau>0$ is the stepsize of the FB algorithm and $L\in \Nset^*$ is the number of iterations of the AD in Model~\ref{mod:AF}.

The following convergence result directly follows from Theorem~\ref{thm:cv:dualFB} and~\cite[Thm.~3.4]{Combettes2005}.
\begin{theorem}\label{thm:cvgce-pnp-limit-AD}
    Let ${(x_k)}_{k\in \Nset}$ be generated by algorithm~\eqref{algo:FB-pnp-AD}. Assume that $0 < \tau < 2 \| A\|_S^{-2} $ and $0 < \sigma < 2 \|\Dica\|_S^{-2}$. 
    When $L\to \infty$ in Model~\ref{mod:AF} (i.e., when the proximity operator in~\eqref{pb:AF-intro} is computed accurately) 
    then ${(x_k)}_{k \in \mathbb{N}}$ converges to a solution to problem~\eqref{pb:inv-min-AF}.
\end{theorem}

Theorem~\ref{thm:cvgce-pnp-limit-AD} is purely theoretical as it only holds when $L \to \infty$, i.e.\ when $\widetilde{\NN}^{\Am}$ has a very large number of sub-iterations.
In practice, this is intractable and only a fixed finite number of iterations are used. For instance, in the context of DDL, typically $L$ will be smaller than $20$. 
We propose to investigate to what extent this impacts the convergence and the nature of the final solution.
Note that in algorithm~\eqref{algo:FB-pnp-AD}, the unrolled denoiser benefits from a warm restart using the output $u_k$ of the inner iterations $\widetilde{\NN}^\Am$ obtained at iteration $k$. We will take advantage of this warm restart to show the convergence of ${(x_k)}_{k\in \Nset}$ when $L=1$. 
In this particular case, at iteration $k\in \mathbb N$, the update of $u_k$ in algorithm~\eqref{algo:FB-pnp-AD} boils down to 
\begin{equation}\label{algo:FB-pnp-ana-1}
    u_{k+1} = \widetilde{\NN}^{\Am}_{1, \tau \lambda, v_k}(u_k) = \feedAF_{\tau \lambda, v_k}(u_k),
\end{equation}
where $\feedAF_{\tau \lambda, v_k}$ is defined in~\eqref{DLnet:ana-dualFB-layer}.

In the following result, we show that algorithm~\eqref{algo:FB-pnp-AD} with $L=1$ corresponds to the scaled primal-dual algorithm proposed by~\cite{loris2011generalization} (see Section~6 in this article, equation (40)). We can then deduce convergence guarantees for ${(x_k)}_{k\in \mathbb N}$, leveraging~\cite[Thm.~1]{loris2011generalization}.

\begin{theorem}\label{prop:analysis_warm_restart}
    Let ${(x_k)}_{k\in \Nset}$ and ${(u_k)}_{k\in \Nset}$ be sequences generated by algorithm~\eqref{algo:FB-pnp-AD} with $L=1$. Assume that $0< \tau < 2 \| A\|_S^{-2} $ and $0< \sigma < 2 \|\Dica\|_S^{-2}$.
    Then the following statements hold:
    \begin{enumerate}
        \item     
        ${(x_k)}_{k\in \Nset}$ converges to a solution to~\eqref{pb:inv-min-AF}.
        \item 
        ${(x_k, \tau^{-1} u_k)}_{k\in \Nset}$ converges to a solution to the saddle-point problem aiming to
        \begin{equation}\label{prop:analysis_warm_restart:SPP}
            \text{find } (\widehat{x}_{\Am}, \widehat{u}_{\Am}) \in
            \Argd \min_{x \in \Rset^N} \max_{u \in \Rset^S} \; \frac12 \| Ax - y \|^2 + \left\langle \Dica x , u \right\rangle - g_\lambda^*(u).
        \end{equation}
    \end{enumerate}
\end{theorem}
\begin{proof}
    Let ${(x_k)}_{k\in \Nset}$ and ${(u_k)}_{k\in \Nset}$ be generated by algorithm~\eqref{algo:FB-pnp-AD} with $L=1$.
    We have
    \begin{equation*}
        x_{k+1}
            = v_k - \Dica^* u_{k+1}
            = x_k - \tau A^*(A x_k - y) - \Dica^* u_{k+1}
    \end{equation*}
    and
    \begin{align*}
        u_{k+1}
        &= \prox_{\sigma g_{\lambda \tau}^*} \Big( u_k - \sigma \Dica \Dica^* u_k + \sigma \Dica v_k \Big) \\
        &= \prox_{\sigma g_{\lambda \tau}^*} \Big( u_k - \sigma \Dica \Dica^* u_k + \sigma \Dica \big( x_k - \tau A^*(A x_k - y) \big) \Big)\\
        &= \prox_{\sigma g_{\lambda \tau}^*} \Big( u_k + \sigma \Dica \Big( x_k - \tau A^* \big( A x_k - y \big) - \Dica^* u_k \Big) \Big).
    \end{align*}
    Hence, algorithm~\eqref{algo:FB-pnp-AD}-\eqref{algo:FB-pnp-ana-1} is equivalent to
    \begin{equation}    \label{thm:equiv:proof:eq1}
    \begin{array}{l}
       \text{for } k = 0, 1, \ldots \\
       \left\lfloor
       \begin{array}{l}
            u_{k+1}
                = \prox_{\sigma g_{\lambda \tau}^*} \Big( u_k + \sigma \Dica \Big( x_k - \tau A^* \big( A x_k - y \big) - \Dica^* u_k \Big) \Big)\\
            x_{k+1}
                = x_k - \tau A^*(A x_k - y) - \Dica^* u_{k+1}.
       \end{array}
       \right.
    \end{array}
    \end{equation}
    Further, noticing that $g_{\lambda\tau} = \tau g_\lambda$ (see definition of $g_\lambda$ after \eqref{pb:AF-intro}), and applying subsequently Prop.~24.8(v) and Prop.~13.23(i) from~\cite{bauschke2011convex}, we have
    \begin{equation*}
        (\forall u \in \Rset^S)\quad
        \prox_{\sigma g_{\lambda \tau}^*}(u)
        = \prox_{\sigma \tau g_\lambda^*}(u)
        % = \prox_{\sigma \tau g_\lambda^*(\tau^{-1} \cdot)} (u)
        = \tau \prox_{\sigma \tau^{-1} g_\lambda^*}(\tau^{-1} u).
    \end{equation*}
    Applying this result to~\eqref{thm:equiv:proof:eq1} leads to
    \begin{equation}    \label{thm:equiv:proof:eq2}
    \begin{array}{l}
       \text{for } k = 0, 1, \ldots \\
       \left\lfloor
       \begin{array}{l}
            u_{k+1}
                = \tau \prox_{\sigma \tau^{-1} g_\lambda^*} \bigg( \tau^{-1} \Big( u_k + \sigma \Dica \Big( x_k - \tau A^* \big( A x_k - y \big) - \Dica^* u_k \Big) \Big) \bigg)\\
            x_{k+1}
                = x_k - \tau A^*(A x_k - y) - \Dica^* u_{k+1}.
       \end{array}
       \right.
    \end{array}
    \end{equation}
    Hence, by setting, for every $k\in \Nset$, $\widetilde{u}_k = \tau^{-1} u_k$, we obtain
    \begin{equation}    \label{algo:LV}
    \begin{array}{l}
        % x_0\in \Rset^N, u_0 \in \Rset^S,\\
       \text{for } k = 0, 1, \ldots \\
       \left\lfloor
       \begin{array}{l}
        \widetilde{u}_{k+1}
            = \prox_{\frac{\sigma}{\tau} g_\lambda^* } \Big( \widetilde{u}_k + \frac{\sigma}{\tau} \Dica \big( x_k - \tau A^*(Ax_k - y) - \tau \Dica^* \widetilde{u}_k \big) \Big)\\
        x_{k+1}
            = x_k - \tau A^*(A x_k - y) - \tau \Dica^* \widetilde{u}_{k+1},
       \end{array}
       \right.
    \end{array}
    \end{equation}
    which is equivalent to the 
    % By noticing that algorithm~\eqref{algo:LV} corresponds to the 
    scaled Loris-Verhoeven primal-dual algorithm proposed by~\cite{loris2011generalization} (see Appendix~\ref{ASec:lvpda}). Hence \cite[Thm.~1]{loris2011generalization} can directly be applied to deduce the convergence results. %, using the scaled version of the algorithm given in Section~6 of the same article.
\end{proof}

Theorem~\ref{prop:analysis_warm_restart} shows that using a single sub-iteration $L=1$ in algorithm~\eqref{algo:FB-pnp-AD} is sufficient to ensure the convergence toward a solution to the same problem as the one solved when using $L \to \infty$, provided that we use a warm-starting scheme.
It also informs us on the structure of the recovered solutions as the solution of a min-max problem.
Let us give a practical implication of (ii) in the case where the $\ell_1$ norm is used as a sparsity-inducing regularization.

\begin{example}
In the particular case  $g_\lambda = \lambda \|\cdot\|_1$, then Theorem~\ref{prop:analysis_warm_restart} shows that
\begin{equation}\label{eq:AD_l1}
        \min_{x \in \mathbb R^N} \frac12 \| Ax-y\|^2 + \lambda \| \Dica x \|_1=
        \min_{x \in \mathbb R^N} \max_{x_0 \in \mathcal{F}_{\lambda}^{\Am} (\Dica)} \frac12 \| Ax-y\|^2 + \langle x, x_0 \rangle,
\end{equation}
where  $\mathcal{F}_{\lambda}^{\Am} (\Dica) = \big \{ x \ | \ \inf \{ \| z \|_\infty \ | \ \Dica^* z = x \} \le \lambda \big \}$. 
This is true for any choice of $\Dica$, including wavelet transforms or the finite difference operator (so that $g_\lambda\circ\Dica$ corresponds to the anisotropic total variation).
\end{example}

% --------------------------------------------------------------------
% --------------------------------------------------------------------
\section{Synthesis denoiser for FB-PnP algorithm}\label{Sec:SD-PnP}
We now focus on the synthesis denoiser (SD) defined in~\eqref{pb:SF-intro} for a fixed dictionary $\Dics$. The objective will be to investigate the behavior of the FB-PnP algorithm~\eqref{algo:FB-PnP}, where the operator $G$ is such an SD. 
Before investigating suitable unrolling techniques to build an SD within the FB-PnP algorithm, we will show that~\eqref{pb:SF-intro} is a proximity operator of a convex function. We will further investigate the link between the associated minimization problem, and the problem that would be obtained if a full synthesis formulation of~\eqref{pb:min_gen}-\eqref{eq:def_p} was considered.
We will then propose an unrolled FB algorithm for solving~\eqref{pb:SF-intro}, and investigate the behavior of algorithm~\eqref{algo:FB-PnP} with this unrolled~SD.

In this section, we use the subscript $\Dics$ in $x^\dagger_\Dics$ and $z^\dagger_\Dics$ to refer to outputs of
denoising tasks involving a fixed synthesis dictionary $\Dics$, from which the SD is built. 
Similarly to Section~\ref{Sec:AD-PnP}, we further use the subscript $\Sm$ to denote solutions associated with the synthesis formulation of the full inverse problem, where 
$\Dics$ is used to reconstruct the signal from sparse codes.
This notational distinction aims to separate the denoising phase from the inverse problem phase.

\subsection{Synthesis problem formulations}\label{Ssec:SF-equiv}
In this section we first define the synthesis problem of interest, and investigate its properties.

\paragraph{Synthesis denoiser as proximity operator}
Unlike for the AD studied in Section~\ref{Sec:AD-PnP}, it is not obvious that the SD corresponds to a proximity operator. 
The first result of this section will show that~\eqref{pb:SF-intro} indeed is a proximity operator of a convex function, making it a suitable candidate for the FB-PnP algorithm~\eqref{algo:FB-PnP}.
For a fixed synthesis dictionary $\Dics$, we introduce the notation $(x^\dagger_{\Dics}, z^\dagger_{\Dics})$, that is a solution to the problem that aims to
\begin{equation}\label{pb:SF-inner}
    \text{find} \quad x^\dagger_{\Dics} = \Dics z^\dagger_{\Dics}
    \quad \text{where} \quad
    z^\dagger_{\Dics} =
    \argmind{z \in \Rset^S} \frac12 \| \Dics z-v \|^2 + g_\lambda(z).
\end{equation}
Problems~\eqref{pb:AF-intro} and~\eqref{pb:SF-inner} are equivalent when the dictionary is invertible,  which is generally not the case.
Detailed comparisons between analysis and synthesis problem formulations can be found in the literature, e.g., in~\cite{EMR07}.

The following result highlights that the synthesis denoising problem~\eqref{pb:SF-inner} corresponds to a proximity operator. This fact has also been noted in~\cite[Sect.~5.3]{chambolle2016introduction} and \cite[Prop.~1]{briceno2024infimal} to reformulate the ADMM and the Douglas-Rachford iterations, respectively.

\begin{proposition}\label{prop:synthesis_as_prox_first}
    Let $\Dics \colon \Rset^S \to \Rset^N$, $v \in \Rset^N$, $\lambda>0$, and let $g \in \Gamma_0(\Rset^N)$. Then, for $z^\dagger_{\Dics}$ defined as in~\eqref{pb:SF-inner}, we have
    \begin{align}
        x^\dagger_{\Dics} = \Dics z^\dagger_{\Dics}
        &=  \prox_{{\left(  g_\lambda^* \circ \Dics^* \right)}^*} (v),     \label{prop:SF_prox:eq} \\
        &=  \prox_{\Dics \rhd g_\lambda}(v),  \label{prop:SF_prox:eq_post}
    \end{align}
    where the infimal post-composition operator $\rhd$ is defined in~\eqref{def:inf-post-comp}.
\end{proposition}

\begin{proof}
    According to Fenchel-Rockafellar duality (see, e.g.,~\cite{bauschke2011convex}), the dual problem associated with~\eqref{pb:SF-inner} aims to
    \begin{equation}
        \text{find } \quad
        u^\dagger_{\Dics} = \argmind{u\in \Rset^N} \frac{1}{2}\| u - v \|^2 + g_\lambda^*(\Dics^* u) = \prox_{g_\lambda^* \circ \Dics^*}(v)  . \label{prop:SF_prox:pr:eq1}
    \end{equation}
    In addition, according to~\cite[Thm.~19.1]{bauschke2011convex}, we have 
    \begin{equation}
        -u^\dagger_{\Dics} \in \partial_{\frac{1}{2}\| \cdot - v \|^2_2}(\Dics z^\dagger_{\Dics}) = \{\Dics z^\dagger_{\Dics} - v\} 
        \quad\Leftrightarrow \quad   
        \Dics z^\dagger_{\Dics} = v - u^\dagger_{\Dics}. \label{prop:SF_prox:pr:eq2}
    \end{equation}
    Combining~\eqref{prop:SF_prox:pr:eq1} and~\eqref{prop:SF_prox:pr:eq2} leads to $\Dics z^\dagger_{\Dics} = (\Id - \prox_{g_\lambda^* \circ \Dics^*})(v)$, which, combined with Moreau's identity gives~\eqref{prop:SF_prox:eq}.
    Further, by combining Prop.~12.36(ii) and Prop.~13.24(iv) from~\cite{bauschke2011convex}, we have ${(g_{\lambda}^* \circ \Dics^*)}^* =  \Dics \rhd g_\lambda$,
    hence the final result~\eqref{prop:SF_prox:eq_post}.
\end{proof}

Is it important to note that here, we don't consider the traditional output of the synthesis formulation $z^\dagger_{\Dics}$ but the denoised image $x^\dagger_{\Dics}$ that is recovered from it.
While $z^\dagger_{\Dics}$ might be non-unique and the application linking some input $v$ to $z^\dagger_{\Dics}$ is non-smooth, it is not the case for $x^\dagger_{\Dics}$, which results from the proximity operator.
We can thus deduce from~Proposition~\ref{prop:synthesis_as_prox_first} that the PnP-FB algorithm~\eqref{algo:FB-PnP} with $G$ being the SD aims to
\begin{equation}\label{pb:inv-min-SF}
    \text{find } \widehat{x}_{\Sm} \in \Argmind{x \in \Rset^N} \frac12 \| Ax - y \|^2 + \Dics \rhd g_\lambda( x).
\end{equation}

\smallskip
\paragraph{Synthesis problem equivalence}

Another standard strategy is to consider a full synthesis formulation of~\eqref{pb:min_gen}-\eqref{eq:def_p}. In this context, the objective is to
\begin{equation} \label{pb:min_gen-synth}
    \text{find } \widehat{z}_{\Sm} \in \Argmind{z \in \Rset^S} \frac{1}{2} \| A \Dics z -y \|^2 +  g_\lambda(z).
\end{equation}
Then, the signal estimate is \textit{synthesized} from the dictionary by computing $ \Dics \widehat{z}_{\Sm}$.
This problem can be solved by any sparse coding algorithm. For instance, when $g$ is proximable, algorithms of choice include FB algorithm, and its accelerated versions such as FISTA~\cite{BT09, chambolle2015convergence}.
In particular, the sparse code estimate $\widehat{z}_{\Sm}$ can be obtained using FB iterations, as follows
\begin{equation}    \label{algo:synthFB_simple}
    \begin{array}{l}
        % z_0 \in \Rset^S,\\
       \text{for } k = 0, 1, \ldots \\
       \left\lfloor
       \begin{array}{l}
        v_{k} =  z_t - \gamma \Dics^* A^*(A \Dics z_t - y)) , \\
        z_{k+1} = \prox_{\gamma g_\lambda}(v_k) ,
       \end{array}
       \right.
    \end{array}
\end{equation}
where $z_0 \in \Rset^S$, and $\gamma>0$ is the step-size chosen to enable convergence of ${(z_k)}_{k \in \mathbb N}$ to a solution to~\eqref{pb:min_gen-synth}~\cite[Thm. 3.4]{Combettes2005}.

The result below highlights the links between the full synthesis problem~\eqref{pb:min_gen-synth} and the analysis problem with SD~\eqref{pb:inv-min-SF}. 
Similar results can be found in~\cite[Prop.~1]{briceno2024infimal} to deduce an alternative formulation for the Douglas-Rachford algorithm.
\begin{theorem}\label{thm:equivalence_synthesis}
    Let $f = \frac{1}{2} \|A \cdot - y\|^2$. Then we have
    \begin{equation}    \label{thm:equivalence_synthesis:min}
        \min_{x\in \mathbb R^N} f( x) + \Dics \rhd g_\lambda ( x)
        = \min_{z \in \mathbb R^S} f(\Dics z) + g_\lambda(z).
    \end{equation}
    In addition, 
    \begin{enumerate}
        \item if $\widehat{z}_{\Sm} \in \Rset^S$ is a solution to~\eqref{pb:min_gen-synth}, then $\widehat{x}_{\Sm}= \Dics \widehat{z}_{\Sm} \in \Rset^N$ is a solution to~\eqref{pb:inv-min-SF},
        \item if $\widehat{x}_{\Sm}\in \Rset^N$ is a solution to~\eqref{pb:inv-min-SF}, then there exists a solution $\widehat{z}_{\Sm} \in \Rset^S$ to~\eqref{pb:min_gen-synth} such that $\widehat{x}_{\Sm}= \Dics \widehat{z}_{\Sm}$.
    \end{enumerate}
    
\end{theorem}

\begin{proof}
    Using twice Fenchel-Rockafellar strong duality (see, e.g., Def.~15.19 and Prop.~15.21 in~\cite{bauschke2011convex}), we have
    \begin{align}
        \min_{z \in \Rset^S} f(\Dics z) +  g_\lambda(z)
        &= - \min_{v \in \Rset^N} f^*(v) + g_\lambda^*(-\Dics^* v)
        =  - \min_{v \in \Rset^N} f^*(-v) + g_\lambda^* \circ \Dics^*( v) \nonumber \\
        &= \min_{x \in \Rset^N} f(x) + {(g_\lambda^* \circ \Dics^*)}^*(x)
        \label{eq:minima_equality} \\
        &= \min_{x \in \Rset^N} f(x) + \Dics \rhd g_\lambda (x) ,\label{eq:minima_equality2}
    \end{align}
    where in~\eqref{eq:minima_equality}, the Fenchel-Rockafellar strong duality is applied to $f^*$ and $(g_\lambda^* \circ \Dics^*)$, with $f^{**}=f$ as $f\in \Gamma_0(\Rset^N)$, and the last equality is obtained using \cite[Prop.~13.24]{bauschke2011convex}.
    This shows~\eqref{thm:equivalence_synthesis:min}, i.e., that the values of the two problems minima are the same.

    \smallskip
    \noindent
    (i) Let $\widehat{z}_{\Sm}$ be a solution to~\eqref{pb:min_gen-synth}. We now want to show that $\widehat{x}_{\Sm}= \Dics \widehat{z}_{\Sm}$ is solution to~\eqref{pb:inv-min-SF}.
    Using the definitions of the infimal post-composition and conjugate functions, we have
    \begin{align*}
        f( \Dics \widehat{z}_{\Sm} ) + \Dics \rhd g_\lambda(\Dics \widehat{z}_{\Sm})
        &=f{( \Dics \widehat{z}_{\Sm} ) + ( g_\lambda^* \circ \Dics^*)}^*(\Dics \widehat{z}_{\Sm}) \nonumber
        = f( \Dics \widehat{z}_{\Sm} ) + \sup_{x\in \Rset^N} \langle x, \Dics \widehat{z}_{\Sm} \rangle - g_\lambda^*(\Dics^* x) \\
        &= f( \Dics \widehat{z}_{\Sm} ) + \sup_{x\in \Rset^N} \langle \Dics^* x, \widehat{z}_{\Sm} \rangle -  g_\lambda^*(\Dics^* x).
    \end{align*}
    Since, for every $x \in \mathbb R^N$,  $\Dics^*x \in \mathbb R^S$, we obtain
    \begin{align}
        f( \Dics \widehat{z}_{\Sm} ) + \Dics \rhd g_\lambda(\Dics \widehat{z}_{\Sm})
        & \le f( \Dics \widehat{z}_{\Sm} ) + \sup_{v\in \Rset^S} \langle v, \widehat{z}_{\Sm} \rangle -  g_\lambda^*(v) %\nonumber \\
         = f( \Dics \widehat{z}_{\Sm} ) +  g_\lambda(\widehat{z}_{\Sm}) \label{eq:unconstrained_conjugate}
    \end{align}
    where the last equality is obtained using the definition of conjugate functions (and $g_\lambda^{**}=g_\lambda$ as $g_\lambda\in \Gamma_0$).
    Further, according to~\eqref{eq:minima_equality}, since $\widehat{z}_{\Sm}$ is a solution to~\eqref{pb:min_gen-synth}, we have
    \begin{align}
         f( \Dics \widehat{z}_{\Sm} ) +  g_\lambda(\widehat{z}_{\Sm})
        &   = \min_{x \in \mathbb R^N} f(x) +  \Dics \rhd g_\lambda(x).
        \label{eq:optimal_pnp}
    \end{align}
    % Let $\widehat{x}_{\Sm}$ be a solution to~\eqref{pb:inv-min-SF}. 
    Combining~\eqref{eq:unconstrained_conjugate} and~\eqref{eq:optimal_pnp}, we obtain $f( \Dics \widehat{z}_{\Sm} ) + \Dics \rhd g_\lambda(\Dics \widehat{z}_{\Sm}) \le  \min_{x \in \mathbb R^N} f(x) +  \Dics \rhd g_\lambda(x).$ \\
    Hence 
    $\Dics \widehat{z}_{\Sm} \in \Argmind{x\in \Rset^N} f(x) +  \Dics \rhd g_\lambda(x)$, so
    $\widehat{x}_{\Sm} = \Dics \widehat{z}_{\Sm}$ is solution to~\eqref{pb:inv-min-SF}.

    \smallskip
    \noindent
    (ii) Let $\widehat{x}_{\Sm}$ be a solution to~\eqref{pb:inv-min-SF}. We will show that there exist some
    $\widehat{z}_{\Sm}$ such that $\widehat{x}_{\Sm} = \Dics \widehat{z}_{\Sm}$ and $\widehat{z}_{\Sm}$ solution
    to~\eqref{pb:min_gen-synth}.
    Assume that the component of $\widehat{x}_{\Sm}$ contained in ${\mathrm{im}(\Dics)}^\bot = \ker \Dics^*$ is not $0$, then $\sup_{v \in \Rset^N} \langle v, \widehat{x}_{\Sm} \rangle -  g_\lambda^*( \Dics^* v)
        \ge \sup_{v \in \ker \Dics^*} \langle v, \widehat{x}_{\Sm} \rangle -  g_\lambda^*(0)
        \ge +\infty$.
    % \begin{equation*}
    %     \sup_{v \in \Rset^N} \langle v, \widehat{x}_{\Sm} \rangle -  g_\lambda^*( \Dics^* v)
    %     \ge \sup_{v \in \ker \Dics^*} \langle v, \widehat{x}_{\Sm} \rangle -  g_\lambda^*(0)
    %     \ge +\infty.
    % \end{equation*}
    As the value of~\eqref{eq:minima_equality} is finite, this would violate the optimality of $\widehat{x}_{\Sm}$. Thus $\widehat{x}_{\Sm}$ is contained in the image of $\Dics$, and there indeed exists $\widehat{z}_{\Sm}\in \mathbb R^S$ such that $\widehat{x}_{\Sm} = \Dics \widehat{z}_{\Sm}$. 
    By definition of the infimal post-composition, we have
    \begin{equation}
        f( \widehat{x}_{\Sm} ) + \Dics \rhd g_\lambda(\widehat{x}_{\Sm})
        = f( \widehat{x}_{\Sm} ) + \inf_{z \in \mathbb R^S, \widehat{x}_{\Sm} = \Dics z }  g_\lambda(z) 
        = \min_{z\in \mathbb R^S, \widehat{x}_{\Sm} = \Dics z} f (\Dics z ) + g_\lambda(z)  .
        \label{eq:synth:optim_constrined_Xstar}
    \end{equation}
    Then, there exists $\widehat{z}_{\Sm}$ such that $\widehat{x}_{\Sm} = \Dics \widehat{z}_{\Sm}$, solution to~\eqref{eq:synth:optim_constrined_Xstar} %, i.e.,
    and we have
    \begin{equation*}
        f( \Dics \widehat{z}_{\Sm} ) + g_\lambda(\widehat{z}_{\Sm})
            = f( \widehat{x}_{\Sm} ) + \Dics \rhd g_\lambda(\widehat{x}_{\Sm})  .
    \end{equation*}
    Hence, thanks to~\eqref{eq:minima_equality2}, $\widehat{z}_{\Sm}$ is a minimizer of~\eqref{pb:min_gen-synth} and this concludes the proof. 
\end{proof}

Theorem~\ref{thm:equivalence_synthesis} shows that the minimum values of Problems~\eqref{pb:inv-min-SF} and~\eqref{pb:min_gen-synth} are the same.
It also highlights that any solution to~\eqref{pb:inv-min-SF} can be projected through $\Dics$ to find a solution to~\eqref{pb:min_gen-synth}. Reciprocally, any solution to~\eqref{pb:min_gen-synth} can be generated by some synthesis coefficients of $\Dics$ which are also solutions to~\eqref{pb:inv-min-SF}. 

\begin{example}
    In the particular case when $g_\lambda = \lambda \|\cdot\|_1$, then Theorem~\ref{thm:equivalence_synthesis} ensures that
    \begin{align}
        \min_{z \in \mathbb R^S} \frac{1}{2} \|y - A \Dics z \|_2^2 + \lambda \| z\|_1
        & = \min_{x \in \mathbb R^N} \frac{1}{2} \| y - A x \|_2^2 + \iota_{\mathcal{F}_\lambda^S(\Dics)}^*(x) \nonumber\\
        & = \min_{x \in \mathbb R^N}  \max_{x_0 \in \mathcal{F}_\lambda^S (\Dics)} \frac12 \| Ax-y \|^2 + \langle x , x_0 \rangle,\label{eq:SD_l1}
    \end{align}
    where $\mathcal{F}_\lambda^S(\Dics) = \{x \ | \ \|\Dics^* x \|_\infty \le \lambda \}$.
\end{example}

One can note the similarity between the analysis and synthesis problems formulations~\eqref{eq:SD_l1} and~\eqref{eq:AD_l1}. Indeed the analysis and the synthesis Lasso formulations for solving the inverse problem~\eqref{pb:invpb} share the same min-max optimization structure, with different convex constraints on the max variable.

\subsection{Synthesis denoiser structure}\label{Ssec:SD-struct}
As explained in Section~\ref{Sec:intro}, and similarly to the AD problem, the SD problem~\eqref{pb:SF-intro} does not have an explicit solution. However, it can be solved using FB iterations when $g$ is proximable~\cite{tseng2000modified, DDM04, Combettes2005, combettes2011proximal}. 
The resulting iterations read 
\begin{equation}    \label{algo:FB-SF}
\begin{array}{l}
\text{for } \ell = 0, 1, \ldots \\
    \left\lfloor
    \begin{array}{l}
        z_{\ell+1} = \prox_{ \zeta g_\lambda} \big( z_\ell - \zeta \Dics^*(\Dics z_\ell - v ) \big) ,
    \end{array}
    \right.
\end{array}
\end{equation}
where $z_0 \in \Rset^S$, and $\zeta>0$ is a step-size. 
Then the following convergence result can be deduced from~\cite[Thm.~3.4]{Combettes2005}.
\begin{theorem}\label{thm:cv:FB-SF}
\sloppy
Assume that ${0 < \zeta < 2 \| \Dics \|^{-2}_S}$.
Then the sequence ${(z_\ell)}_{\ell \in \Nset}$ converges to $z_{\Dics}^\dagger$, and $x^\dagger_{\Dics} = \Dics
z_{\Dics}^\dagger$ is a solution to~\eqref{pb:SF-inner}. 
\end{theorem}
Algorithm~\eqref{algo:FB-SF} can be used as sub-iterations when included within the global FB algorithm~\eqref{algo:FB-PnP}, to approximate the computation of $\prox_{\Dics \rhd g_\lambda}$. 
Similarly to Section~\ref{Sec:AD-PnP}, in this section we investigate the behavior of~\eqref{algo:FB-PnP} with the SD, when only one iteration of algorithm~\eqref{algo:FB-SF} is computed. To this aim, we define an unfolded iterative scheme approximating the proximal operator of $p = \Dics \rhd g_\lambda$, i.e., an SD as defined in~\eqref{pb:SF-intro}.
In the remainder, we use the superscript $\Sm$ to identify operators associated with the synthesis formulation.

\begin{model}[Synthesis denoiser (SD)]\label{mod:SF}
    Let $(v, z_0) \in \Rset^N \times \Rset^S$, $\Dics \colon \Rset^S \to \Rset^N$ be a linear operator, and $\lambda>0$ be a regularization parameter. Let $L \in \Nset^*$
    and $\widetilde{\NN}^{\Sm}_{L,\lambda, v} \colon  \Rset^S \to \Rset^S$ be defined as
    \begin{equation}\label{DLnet:synth-FB-dual}
    \widetilde{\NN}^{\Sm}_{L,\lambda, v}(z_0) = \underset{\text{$L$ compositions}}{\underbrace{\feedSF_{\lambda, v} \circ \dots \circ \feedSF_{\lambda, v}}}(z_0)	,
    \end{equation}
    where
    \begin{equation}\label{DLnet:synth-FB-layer}
    \feedSF_{\lambda, v} \colon \Rset^S \to \Rset^S 
    \colon z \mapsto \prox_{ \zeta g_\lambda} \big( z - \zeta {\Dics}^*( \Dics z - v) \big)
    \end{equation}
    with $\zeta>0$.
    The unrolled SD $\NN^{\Sm}_{L,\lambda, v} \colon  \Rset^S \to \Rset^N$ is obtained as follows:
    \begin{equation}\label{DLnet:synth-FB}
    \NN^{\Sm}_{L,\lambda, v}( z_0) = \Dics \widetilde{\NN}^{\Sm}_{L,\lambda, v}(z_0).
    \end{equation}
\end{model}

A few comments can be made on Model~\ref{mod:SF}.
\begin{remark}\label{rmk:unroll-SF-struct}\
\begin{enumerate}
\item
The layers of Model~\ref{mod:SF} consisting of FB iterations, they correspond to feed-forward layers, and $\feedSF_{\lambda, v}$ can be reformulated as
    \begin{equation*}
    (\forall z \in \Rset^S)\quad
    \feedSF_{\lambda, v} ( z) = \prox_{ \zeta g_\lambda} \Big( \big( \Id - \zeta \Dics^*\Dics \big) z + \zeta {\Dics}^* v \Big),
    \end{equation*}
where $\Id - \zeta \Dics^*\Dics$ is a linear operator and $\zeta {\Dics}^* v$ a bias.
This reformulation highlights the residual structure of each layer, and makes explicit the operator and bias involved in the synthesis denoiser.

\item\label{rmk:unroll-SF-struct:ii}
The feed-forward layers $\feedAF_{\lambda, v}$ and $\feedSF_{\lambda, v}$ share the same structure, except that the activation function in $\feedAF_{\lambda, v}$ corresponds to the proximity operator of $g_\lambda^*$, while for $\feedSF_{\lambda, v}$ it corresponds to the proximity operator of $g_\lambda$.

\item
A direct consequence of Theorem~\ref{thm:cv:FB-SF} is that, for any $z_0\in \Rset^S$,  if $0<  \zeta < 2 \|\Dics\|_S^{-2}$, then $\displaystyle \lim_{L \to +\infty} \NN^{\Sm}_{L, \lambda, v}(z_0) = x_{\Dics}^\dagger$,
where $x_{\Dics}^\dagger$ is defined in~\eqref{pb:SF-inner}.
\end{enumerate}
\end{remark}

\subsection{FB-PnP algorithm with approximated SD}
\label{Ssec:SD-FBPnP}

In this section we aim to solve~\eqref{pb:inv-min-SF}.
We thus study the asymptotic behavior of the FB-PnP algorithm~\eqref{algo:FB-PnP} with the SD given in Model~\ref{mod:SF}. In this context, this algorithm boils down to
\begin{equation}    \label{algo:FB-pnp-SD}
    \begin{array}{l}
        x_0\in \Rset^N, z_0 \in \Rset^S,\\
       \text{for } k = 0, 1, \ldots \\
       \left\lfloor
       \begin{array}{l}
        {v}_{k} =  x_k - \tau A^*(A x_k - y) , \\
        z_{k+1} = \widetilde{\NN}^{\Sm}_{L, \tau \lambda, v_k}(z_k), \\
        x_{k+1} = \Dics z_{k+1}, %\varrho^{\Sm}( z_{k+1} ),
       \end{array}
       \right.
    \end{array}
\end{equation}
where $\tau>0$ is the stepsize of the FB algorithm and $L\in \Nset^*$ is the number of iterations of the AD in Model~\ref{mod:AF}.
The following convergence result is a direct application of results from~\cite{Combettes2005}.
\begin{theorem}\label{thm:cvgce-pnp-limit-SD}
    Let ${(x_k)}_{k\in \Nset}$ be generated by algorithm~\eqref{algo:FB-pnp-SD}. Assume that $0 < \tau < 2 \| A\|_S^{-2} $ and $0 < \zeta < 2 \|\Dics\|_S^{-2}$.
    When $L\to \infty$ in Model~\ref{mod:SF} (i.e., when problem~\eqref{pb:SF-inner} is solved accurately)
    then ${(x_k)}_{k \in \mathbb{N}}$ converges to a solution to problem~\eqref{pb:inv-min-SF}.
\end{theorem}

Similarly to Theorem~\ref{thm:cvgce-pnp-limit-AD}, Theorem~\ref{thm:cvgce-pnp-limit-SD} only holds when $L \to \infty$, i.e.\ when $\widetilde{\NN}^{\Sm}$ has a very large number of sub-iterations (i.e., layers). As for the AD case, this is intractable, and only a fixed (low) number of iterations are considered in practice. 
In the remainder of this section, we focus on the case when only $L=1$ sub-iteration is computed in Model~\ref{mod:SF}, and investigate the behavior of ${(x_k)}_{k\in \mathbb N}$ in this context, as well as the nature of the provided solution. 
In this particular case, at iteration $k\in \mathbb N$, the update of $z_k$ in algorithm~\eqref{algo:FB-pnp-SD} boils down to 
\begin{equation}\label{algo:FB-pnp-synthesis-1}
    z_{k+1} = \widetilde{\NN}^{\Sm}_{1, \tau \lambda, v_k}(z_k) = \feedSF_{\tau \lambda, v_k}(z_k),
\end{equation}
where $\feedSF_{\tau\lambda, v_k}$ is defined in~\eqref{DLnet:synth-FB-layer}.

We will now show that algorithm~\eqref{algo:FB-pnp-SD}-\eqref{algo:FB-pnp-synthesis-1} is equivalent to the FB iterations in~\eqref{algo:synthFB_simple} for solving~\eqref{pb:min_gen-synth}. This will enable us to deduce asymptotic convergence guarantees of algorithm~\eqref{algo:FB-pnp-SD} when $L=1$.

\begin{theorem}\label{thm:cvg-SF-L1}
    Let ${(z_k)}_{k\in \mathbb N}$ and ${(x_k)}_{k\in \mathbb N}$ be sequences generated by algorithm~\eqref{algo:FB-pnp-SD} with $L=1$. Assume that $0< \tau \zeta < 2 \| A \Dics \|_S^{-2}$, then 
    \begin{enumerate}
        \item ${(z_k)}_{k\in \mathbb N}$ converges to a solution $\widehat{z}_S$ to problem~\eqref{pb:min_gen-synth}.
        \item ${(x_k)}_{k \in \mathbb N}$ converges to a solution $\widehat{x}_S$ to problem~\eqref{pb:inv-min-SF}.
    \end{enumerate}
\end{theorem}
\begin{proof}
    Expending algorithm~\eqref{algo:FB-pnp-SD}-\eqref{algo:FB-pnp-synthesis-1}, we obtain, for every $k \in \mathbb N$,
    \begin{align}
        v_k &=  x_k - \tau A^*(Ax_k - y)    \label{thm:cvg-SF-L1:pr1}\\
        \widetilde{z}_k &=  z_k - \zeta D^*(D z_k - v_k)    \label{thm:cvg-SF-L1:pr2} \\
        z_{k+1} &=  \prox_{\tau \zeta g_\lambda} (\widetilde{z}_k)    \label{thm:cvg-SF-L1:pr3} \\
        x_{k+1} &=  D z_{k+1},     \label{thm:cvg-SF-L1:pr4}
    \end{align}
    where \eqref{thm:cvg-SF-L1:pr3} is obtained by definition of $\feedSF_{\tau \lambda, v_k}(z_k) = \prox_{\zeta g_{\tau\lambda}}(\widetilde{z}_k)$, and noticing that $g_{\lambda\tau} = \tau g_{\lambda}$.
    Combining~\eqref{thm:cvg-SF-L1:pr1} and~\eqref{thm:cvg-SF-L1:pr2}, and then using~\eqref{thm:cvg-SF-L1:pr4}, we obtain
    \begin{align*}
        \widetilde{z}_k 
        &=  z_k - \zeta D^*\Big( D z_k - \big( x_k - \tau A^*(Ax_k - y) \big) \Big) \\
        &=  z_k - \zeta D^*\Big( D z_k - \big( Dz_k - \tau A^*(A Dz_k - y) \big) \Big) \\
        &=  z_k - \tau \zeta D^*A^*(A Dz_k - y)  .
    \end{align*}
    Hence, algorithm~\eqref{algo:FB-pnp-SD}-\eqref{algo:FB-pnp-synthesis-1} is equivalent to
    \begin{equation}    \label{thm:cvg-SF-L1:pr5}
    \begin{array}{l}
        % x_0\in \Rset^N, z_0 \in \Rset^S,\\
       \text{for } k = 0, 1, \ldots \\
       \left\lfloor
       \begin{array}{l}
        \widetilde{z}_k  =  z_k - \tau \zeta D^*A^*(A Dz_k - y) , \\
        z_{k+1} = \prox_{\tau \zeta g_\lambda} (\widetilde{z}_k) , \\
        x_{k+1} = \Dics z_{k+1} .
       \end{array}
       \right.
    \end{array}
    \end{equation}
    Algorithm~\eqref{thm:cvg-SF-L1:pr5} corresponds to the FB iterations given in~\eqref{algo:synthFB_simple} with step-size $\gamma=\tau \zeta$. Then, according to~\cite[Thm.~3.4]{Combettes2005}, ${(z_k)}_{k \in \mathbb N}$ converges to a solution $\widehat{z}_S$ to~\eqref{pb:min_gen-synth} if $0< \tau \zeta < 2 \| A \Dics \|_S^2$. 
    Further, according to Theorem~\ref{thm:equivalence_synthesis}(i), $\widehat{x}_S = D \widehat{z}_S$ is a solution to~\eqref{pb:inv-min-SF}.
\end{proof}

To summarize, in Theorem~\ref{thm:cvg-SF-L1}, as in the case of AD, we show that we can recover the solution of the full problem~\eqref{pb:inv-min-SF} using only $L=1$ sub-iterations, provided that we use a warm-restarting strategy.

% --------------------------------------------------------------------
% --------------------------------------------------------------------
\section{Approximated AD and SD for smooth minimization}\label{sec:bi-level-approximate}
Algorithms~\eqref{algo:FB-pnp-AD}-\eqref{algo:FB-pnp-ana-1} and \eqref{algo:FB-pnp-SD}-\eqref{algo:FB-pnp-synthesis-1} enable using a unique sub-iteration of an unrolled FB-based algorithm within a FB-PnP framework, for solving~\eqref{pb:inv-min-AF} and~\eqref{pb:inv-min-SF}, respectively. 
In particular, we showed in Theorem~\ref{prop:analysis_warm_restart} and Theorem~\ref{thm:cvg-SF-L1} that such an approximation enables solving the problems of interest~\eqref{pb:inv-min-AF} and~\eqref{pb:inv-min-SF}, respectively, as if an infinite number of sub-iterations were computed.
In this section, we focus on a particular case when a smoothed version of the problems of interest is considered. 

In this context, we do not need to consider the AD and SD problems separately. Instead, we focus on a generic problem of the form
\begin{equation}\label{pb:min:gen-bilevel}
    \minimize{x \in \mathbb R^{\widetilde{N}}} \widetilde{f}(x) + \widetilde{g}_\lambda(x),
\end{equation}
where $\widetilde{f} \colon \mathbb R^{\widetilde{N}} \to (-\infty, +\infty]$ is assumed to be convex and $\widetilde{\beta}$-Lipschitz differentiable, for $\widetilde{\beta}>0$, 
and $\widetilde{g}_\lambda = \lambda \widetilde{g}$ is a convex, proper and lower-semicontinuous function with $\widetilde{g} \colon \mathbb R^{\widetilde{N}} \to (-\infty, +\infty]$ and $\lambda>0$. 
Then, the synthesis problem is a particular case of~\eqref{pb:min:gen-bilevel} where
\begin{equation}\label{pb:min:gen-bilevel:SF}
    % \text{(SF)}\quad 
    \begin{cases}
    \begin{array}{l@{}l@{}l}
        \widetilde{f} \colon & \mathbb R^S \to (-\infty, +\infty] \colon & z \mapsto \| A \Dics z - y \|^2 \\
        \widetilde{g}_\lambda \colon & \mathbb R^S \to (-\infty, +\infty] \colon & z \mapsto g_\lambda(z).
    \end{array}
    \end{cases}
\end{equation}
Similarly, the analysis problem is a particular case of~\eqref{pb:min:gen-bilevel} where
\begin{equation}\label{pb:min:gen-bilevel:AF}
    % \text{(AF)}\quad 
    \begin{cases}
    \begin{array}{l@{}l@{}l}
        \widetilde{f} \colon & \mathbb R^N \to (-\infty, +\infty] \colon & x \mapsto \| Ax - y \|^2 \\
        \widetilde{g}_\lambda \colon & \mathbb R^N \to (-\infty, +\infty] \colon & x \mapsto g_\lambda(\Dica x).
    \end{array}
    \end{cases}
\end{equation}
\begin{remark}
    We would emphasize that, according to Theorem~\ref{thm:equivalence_synthesis}, if $\widehat{z}$ is a solution to full synthesis problem~\eqref{pb:min:gen-bilevel}-\eqref{pb:min:gen-bilevel:SF}, then $\Dics \widehat z$ is a solution to analysis problem with SD~\eqref{pb:inv-min-SF}.
\end{remark}

In this section, we focus on a smooth regularized version of problem~\eqref{pb:min:gen-bilevel}, aiming to
\begin{equation}
    \label{pb:min:gen-bilevel-moreau}
    \text{find } x^\ddagger = \argmind{x \in \mathbb{R}^{\widetilde{N}}} %F(x,u) := 
    \widetilde{f}(x) + {^\mu\widetilde{g}_\lambda}(x), % + \frac{1}{2}\|x - u\|_2^2 ,
\end{equation}
where $^\mu\widetilde{g}_\lambda$ denotes de Moreau-Yosida envelope of $\widetilde{g}_\lambda$ with parameter $\mu$, defined as
\begin{equation*}%\label{def:moreau-env}
    (\forall x \in \mathbb R^{\widetilde{N}})\quad
    {^\mu\widetilde{g}_\lambda}(x) = \min_{u \in \mathbb R^{\widetilde{N}}} \widetilde{g}_\lambda(u) + \frac\mu2 \|x-u\|^2.
\end{equation*}
According to~\cite[Prop.~12.30]{bauschke2011convex}, ${^\mu\widetilde{g}_\lambda}$ is $\mu$-Lipschitz-differentiable, with gradient given by
\begin{equation*}
    (\forall x \in \mathbb R^{\widetilde{N}})\quad
    \nabla {^\mu\widetilde{g}_\lambda}(x) = \mu \big( x - \prox_{\widetilde{g}_{\frac\lambda\mu}}(x) \big).
\end{equation*}
Hence, Problem~\eqref{pb:min:gen-bilevel-moreau} can be solved using a gradient descent algorithm, given by
\begin{equation}    \label{algo:GD-moreau}
    \begin{array}{l}
        x_0\in \Rset^N, u_0 \in \Rset^N,\\
       \text{for } k = 0, 1, \ldots \\
       \left\lfloor
       \begin{array}{l}
        x_{k+1} =  x_k - \tau \nabla_x F( x_k, u_k) , \\
        u_{k+1} = \prox_{\widetilde{g}_{\frac\lambda\mu}}(x_{k+1}), 
       \end{array}
       \right.
    \end{array}
\end{equation}
where, for every $(x,u) \in \mathbb{R}^{\widetilde{N}} \times \mathbb{R}^{\widetilde{N}}$, 
\begin{equation}\label{eq:def:F-bi}
    {F(x,u) =} \widetilde{f}(x) + \widetilde g_\lambda(u) + \frac\mu2\|x - u\|^2,
\end{equation}
and
\begin{equation}\label{eq:def:F-bi-grad}
    \nabla_x F(x,u) = \nabla \widetilde{f}(x) + \mu (x-u)
\end{equation}
is Lipschitz-differentiable, with constant $\widetilde{\beta} + \mu$.
The sequence ${(x_k)}_{k \in \mathbb N}$ generated by algorithm~\eqref{algo:GD-moreau} is then ensured to converge to the solution $x^\ddagger$ to~\eqref{pb:min:gen-bilevel-moreau}, if $0 < \tau < 2 {(\widetilde{\beta} + \mu)}^{-1}$.

In practice, as the proximity operator of $\widetilde{g}_{\frac{\lambda}{\mu}}$ often does not have a closed form, the computational complexity of this algorithm is too high.
However, the solution of the proximity operator can be computed using iterative schemes such as the FB or the dual-FB algorithms, as described in Sections~\ref{Sec:AD-PnP} and~{Sec:SD-PnP}.
Using a bi-level framework similar to the one proposed by~\cite{dagreou2022framework}, we show that one can use inexact proximity operator computation powered by iterative proximal algorithms with warm restarts to build an algorithm that converges to a solution to~\eqref{pb:min:gen-bilevel-moreau}.
In this context, we then focus on the following algorithm
\begin{equation}    \label{algo:algo_bi_level_prox}
    \begin{array}{l}
        x_0\in \Rset^{\widetilde{N}}, u_0 \in \Rset^{\widetilde{N}}x,\\
       \text{for } k = 0, 1, \ldots \\
       \left\lfloor
       \begin{array}{l}
        x_{k+1} =  x_k - \tau \nabla_x F( x_k, u_k) , \\
        u_{k+1} = \widetilde{G}_{L, \lambda\mu^{-1}} (x_{k+1}, u_k),
       \end{array}
       \right.
    \end{array}
\end{equation}
where, for every $k\in \mathbb N$, $\widetilde{G}_{L, \lambda\mu^{-1}} \colon \mathbb R^{\widetilde{N}} \times \mathbb R^{\widetilde{N}} \to \mathbb R^{\widetilde{N}}$ is an operator that aims to approximate $\prox_{ \widetilde{g}_{\frac{\lambda}{\mu}}}(x_{k+1})$ by computing $L$ sub-iterations of some iterative algorithm. The first input corresponds to the output $x_{k+1}$ of the gradient descent step, \emph{i.e.} the point at which the proximity step should be computed. The second input corresponds to the output of $\widetilde{G}_{L, \lambda\mu^{-1}}$ from the previous iteration $k-1$, which is aimed to be used for warm-restart, i.e. initializing the sub-iterations to compute the approximation of $\prox_{ \widetilde{g}_{\frac{\lambda}{\mu}}}(x_{k+1})$.
Unlike in previous sections, in this section, this operator is not necessarily based on FB or dual-FB iterations. However, we will assume that it satisfies some sufficient decrease property (see Theorem~\ref{thm:bilevel} Condition~\ref{thm:bilevel:cond:i}).

\begin{remark}
    It can be noticed that the regularized problem~\eqref{pb:min:gen-bilevel-moreau} can equivalently be rewritten as the following bi-level optimization problem
    \begin{equation}\label{pb:min:bilevel}
        \text{find }  x^\ddagger = 
            \argmind{x \in \mathbb{R}^{\widetilde{N}}} 
                % \widetilde{f}(x) + \widetilde g_\lambda(\pgx{x}) + \frac\mu2\|x - \pgx{x}\|^2
                F(x, \pgx{x})
                \quad \text{such that} \quad \pgx{x} = \prox_{\widetilde{g}_{\frac{\lambda}{\mu}}}(x),
    \end{equation}
    where $F$ is defined in \eqref{eq:def:F-bi} and $\mu>0$. 
    Hence, the sequence ${(x_k)}_{k\in \mathbb N}$ generated by~\eqref{algo:GD-moreau} is ensured to converge to a solution to~\eqref{pb:min:bilevel} if $0< \tau < 2{(\widetilde{\beta} + \mu)}^{-1}$.
\end{remark}

In the following result, we analyse the asymptotic behavior of algorithm~\eqref{algo:algo_bi_level_prox}. The two conditions on $\widetilde{G}_{L, \lambda\mu^{-1}}$ and on the step size $\tau$ are commented in Remark~\ref{rmk:bilev}.

\begin{theorem}\label{thm:bilevel}
    Let ${(x_k)}_{k \in \mathbb N}$ and ${(u_k)}_{k \in \mathbb N}$ be sequences generated by~\eqref{algo:algo_bi_level_prox}. 
    Assume that the following conditions hold:
    \begin{enumerate}
        \item\label{thm:bilevel:cond:i}
        there exists $\alpha_L \in ]0, 1/\sqrt{2}[$ such that
        \begin{equation} \label{thm:bilev:ass}
            (\forall k \in \mathbb N)\quad
            \| \widetilde{G}_{L, \lambda\mu^{-1}}(x_{k+1}, u_k) - \pgx{x_{k+1}} \| \le \alpha_L \| u_k - \pgx{x_{k+1}} \|,
        \end{equation}
        where $\pgx{x_{k+1}} = \prox_{\widetilde{g}_{\frac{\lambda}{\mu}}}(x_{k+1})$, 
        \item\label{thm:bilevel:cond:ii}
        the step-size $\tau$ in~\eqref{algo:algo_bi_level_prox} satisfies $0 < \tau < \frac{2\phi^\ddagger(1 - 2\alpha_L^2)}{\mu^2}$, where $\phi^\ddagger>0$ is the positive root of the polynomial 
        \begin{equation}\label{thm:bilevel:cond:pol}
        \phi \in \Rset^* \mapsto p(\phi) = (8\alpha_L^2(1 - 2\alpha_L^2))\phi^2 + 2(\widetilde{\beta} + \mu)(1-2\alpha_L^2)\phi - \mu^2.    
        \end{equation}
    \end{enumerate}
    Then, ${(x_k)}_{k\in \mathbb N}$ converges to a solution to~\eqref{pb:min:bilevel}.
\end{theorem}
%%% --------------------------------
\begin{proof}
Let $k \in \mathbb N$,
and let us define, for every $x \in \mathbb R^{\widetilde{N}}$, $h(x) = F(x, \pgx{x}) $, where $F$ and $\pgx{x} \in \mathbb R^{\widetilde{N}}$ are defined in~\eqref{eq:def:F-bi} and \eqref{pb:min:bilevel}, respectively.
According to~\eqref{eq:def:F-bi-grad}, since $\nabla_x F( \cdot, \pgx{x})$ is $(\widetilde{\beta}+\mu)$-Lipschitz, then $h$ is Lipschitz-differentiable, with constant $\widetilde{\beta}+ \mu>0$. Thus we can apply the descent lemma from~\cite[Prop.~A.24]{berstsekas99} to obtain
\begin{multline}\label{thm:bilevel:pr1}
    h(x_{k+1}) 
    =  h(x_k - \tau \nabla_x F(x_k, u_k))
    \le   h(x_k) - \tau \langle  \nabla_x F(x_k, u_k)), \nabla h(x_k) \rangle  \\
    + \frac{(\widetilde{\beta}+\mu) \tau^2}{2} \| \nabla_x F(x_k, u_k)) \|^2 .
\end{multline}
Further, we have
\begin{multline}\label{thm:bilevel:pr2}
\frac12 \|\nabla_x F(x_k, u_k) - \nabla h(x_k)\|^2
 = \frac12 \|\nabla_x F(x_k, u_k)\|^2 + \frac12 \|\nabla h(x_k)\|^2 \\
 - \langle  \nabla_x F(x_k, u_k)), \nabla h(x_k) \rangle.
\end{multline}
Then, by combining~\eqref{thm:bilevel:pr1} and~\eqref{thm:bilevel:pr2}, we obtain
\begin{align}
    h(x_{k+1}) 
    &\le   h(x_k) 
    - \frac{\tau}{2} \|\nabla_x F(x_k, u_k)\|^2 - \frac{\tau}{2} \|\nabla h(x_k)\|^2  + \frac{\tau}{2} \|\nabla_x F(x_k, u_k) - \nabla h(x_k)\|^2 \nonumber \\
    & \qquad\qquad\qquad
    + \frac{(\widetilde{\beta}+\mu) \tau^2}{2} \| \nabla_x F(x_k, u_k)) \|^2 \nonumber \\
    &=   h(x_k) 
    - \frac{\tau}{2} \|\nabla h(x_k)\|^2 
    - \frac{\tau \big(1 - (\widetilde{\beta}+\mu) \tau \big)}{2} \|\nabla_x F(x_k, u_k)\|^2 \nonumber \\
    & \qquad\qquad\qquad
    + \frac{\tau}{2} \|\nabla_x F(x_k, u_k) - \nabla h(x_k)\|^2 . \label{thm:bilevel:pr3}
\end{align}
Furthermore, by definition of $F$ and $h$, we have
\begin{align*}
    \nabla_x F(x_k, u_k) - \nabla h(x_k)
    & = \nabla_x \Big( F(x_k, u_k) - F(x_k, \pgx{x_k}) \Big) 
    = \mu (x_k - u_k) - \mu (x_k - \pgx{x_k})  \\
    &=  \mu (\pgx{x_k} - u_k).
\end{align*}
Hence, combining this equality with~\eqref{thm:bilevel:pr3}, we obtain
\begin{equation}
h(x_{k+1}) -h(x_k) 
\le - \frac{\tau}{2} \|\nabla h(x_k)\|^2 
    - \frac{\tau \big(1 - (\widetilde{\beta}+\mu)  \tau \big)}{2} \|\nabla_x F(x_k, u_k)\|^2 + \frac{\tau \mu^2}{2} \|\pgx{x_k} - u_k\|^2 .\label{thm:bilevel:pr3b}
\end{equation}

% -----
% \smallskip
According to~\eqref{algo:algo_bi_level_prox} and to~\eqref{thm:bilev:ass}, we have
\begin{align}
    \| u_{k+1} - \pgx{x_{k+1}} \|
    &   \le \alpha_L \| u_k - \pgx{x_{k+1}} \|
        = \alpha_L \| (u_k - \pgx{x_k}) + (\pgx{x_k} - \pgx{x_{k+1}}) \| \nonumber\\
    &   \le \alpha_L \| u_k - \pgx{x_k} \| + \alpha_L \|\pgx{x_k} - \pgx{x_{k+1}} \|. \label{thm:bilevel:pr4}
\end{align}
Since the proximity operator is 1-Lipschitz~\cite{bauschke2011convex}, we have
\begin{equation}
    \|\pgx{x_k} - \pgx{x_{k+1}} \|
    \le \| x_k - x_{k+1} \| = \tau \| \nabla_x F(x_k, u_k) \|, \label{thm:bilevel:pr5}
\end{equation}
where the last equality is obtained by definition of $x_{k+1}$ in~\eqref{algo:algo_bi_level_prox}.
By combining~\eqref{thm:bilevel:pr4} and~\eqref{thm:bilevel:pr5}, we have
\begin{equation*}
    \| u_{k+1} - \pgx{x_{k+1}} \| 
    \le \alpha_L \| u_k - \pgx{x_k} \| + \alpha_L \tau \| \nabla_x F(x_k, u_k) \|.
\end{equation*}
Then, by squaring the above inequality, 
and by applying the Jensen's inequality we obtain
\begin{equation}\label{thm:bilevel:pr6}
    \| u_{k+1} - \pgx{x_{k+1}} \|^2 
    \le 2\alpha_L^2 \| u_k - \pgx{x_k} \|^2 
    + 2\alpha_L^2 \tau^2 \| \nabla_x F(x_k, u_k) \|^2.
\end{equation}

\smallskip
For $\phi > 0$, we introduce the Lyapunov function 
\begin{equation}\label{thm:bilevel:pr-lyap}
    (\forall k \in \mathbb N) \quad \mathcal L_k = h(x_k) + \phi\|u_k - \pgx{x_k}\|^2
\end{equation}
and we will show that ${(\mathcal{L}_k)}_{k \in \mathbb N}$ is a decreasing sequence.
By definition of $\mathcal L_k$, and using~\eqref{thm:bilevel:pr3b} and~\eqref{thm:bilevel:pr6}, we obtain
\begin{align}
\mathcal{L}_{k+1} - \mathcal{L}_{k} 
    &   = h(x_{k+1}) - h(x_k) + \phi \big(\|u_{k+1} - \pgx{x_{k+1}}\|^2 - \|u_k - \pgx{x_k}\|^2\big) \nonumber\\
    &   \le -\frac{\tau}{2} \|\nabla h(x_k)\|^2 
        - \frac{\tau \big( 1 - (\widetilde{\beta}+\mu)  \tau \big)}{2} \|\nabla_x F(x_k, u_k)\|^2 + \frac{\tau \mu^2}{2} \|\pgx{x_k} - u_k\|^2  \nonumber\\
    & \qquad\qquad 
        + \phi \Big((2\alpha_L^2 -1) \| u_k - \pgx{x_k} \|^2 
        + 2\alpha_L^2 \tau^2 \| \nabla_x F(x_k, u_k) \|^2\Big) \nonumber\\
    &   \le  -\frac{\tau}{2} \|\nabla h(x_k)\|^2 - \tau \Big( \frac{1 - (\widetilde{\beta}+\mu)  \tau}{2} - \phi 2\alpha_L^2 \tau \Big)  \|\nabla_x F(x_k, u_k)\|^2 \nonumber\\
    &\qquad\qquad 
        - \Big( \phi(1 - 2\alpha_L^2) - \frac{\tau \mu^2}{2} \Big) \| u_k - \pgx{x_k} \|^2
    \nonumber %\label{thm:bilevel:pr7}
\end{align}
If $\tau>0$, $\frac{1 - (\widetilde{\beta}+\mu)  \tau}{2} - \phi 2\alpha_L^2 \tau >0$, and $\phi(1 - 2\alpha_L^2) -
\frac{\tau \mu^2}{2} >0$, then ${(\mathcal{L}_k)}_{k \in \mathbb N}$ is a decreasing sequence.
Hence $0 < \alpha_L < \frac{1}{\sqrt{2}}$, and $\tau$ must satisfy $0 < \tau < \min \Big\{ \frac{2\phi(1 - 2\alpha_L^2)}{\mu^2}, \frac{1}{\widetilde{\beta} + \mu + 4\phi\alpha_L^2} \Big\}$. 
We thus can choose $\phi>0$ to maximize the upper bound $\min \Big\{ \frac{2\phi(1 - 2\alpha_L^2)}{\mu^2}, \frac{1}{\widetilde{\beta} + \mu + 4\phi\alpha_L^2} \Big\}$, \textit{i.e.}, such that $\frac{2\phi(1-2\alpha_L^2)}{\mu^2} = \frac{1}{\widetilde{\beta} + \mu + 4\phi\alpha_L^2}$.
This is equivalent to find the positive root of polynomial~\eqref{thm:bilevel:cond:pol}.
The determinant of this polynomial is positive with a positive solution, as it is negative in $\phi=0$ and positive in
$\phi \to \infty$. By denoting $\phi^\ddagger$ the positive root of this polynomial, we then obtain that if $0< \tau <
\frac{2\phi^\ddagger(1 - 2\alpha_L^2)}{\mu^2}$, then ${(\mathcal{L}_k)}_{k\in \mathbb N}$ is decreasing, and that
\begin{equation}
     \frac{\tau}{2} \|\nabla h(x_k)\|^2  \le \mathcal{L}_{k} - \mathcal{L}_{k+1} .
     \label{thm:bilevel:pr8}
\end{equation}
Hence, by summing~\eqref{thm:bilevel:pr8} over $k \in \{0, \ldots, K-1\}$, for $K>0$, we obtain $\sum_{k=0}^K \frac{\tau}{2} \|\nabla h(x_k)\|^2  \le \mathcal L_0 - \mathcal L_K < +\infty$.
Thus ${(\|\nabla h(x_k)\|^2)}_{k\in \mathbb N}$ converges to $0$.

Since $h$ is convex, coercive (see \cite[Cor.~11.16 \&~11.17]{bauschke2011convex}), and has a continuous gradient, the sequence $(x_k)_{k\in\mathbb{N}}$ is bounded and satisfies $\nabla h(x_k) \to 0$. Therefore, all its cluster points are minimizers of $h$, and $(x_k)$ converges to a solution to~\eqref{pb:min:bilevel}.
\end{proof}

\begin{remark}\ \label{rmk:bilev}
    \begin{enumerate}
        \item 
        Condition~\ref{thm:bilevel:cond:i} in Theorem~\ref{thm:bilevel} means that the operator $\widetilde{G}_{L, \lambda\mu^{-1}}$ in~\eqref{algo:algo_bi_level_prox} must be chosen to satisfy some sufficient approximation condition in the sense that, for every $k\in \mathbb N$, $u_{k+1} = \widetilde{G}_{L, \lambda\mu^{-1}}(x_{k+1}, u_k)$ should be closer to $\prox_{\widetilde{g}_{\frac{\lambda}{\mu}}}(x_{k+1})$ than $u_k$ (i.e., the initialization of that operator). \\ 
        In practice, if we choose a monotone algorithm such as FB or dual-FB, this condition can be satisfied for an $L$ sufficiently large (due to decreasing properties on the associated objective functions). However, it can be difficult to determine the value of $L$ ensuring $\alpha_L < 1/\sqrt{2}$.
        
        \item 
        For $(x,u) \in (\Rset^N)^2$, we define $\bar{G}_\lambda(x, u) = \frac12\|x - u\|^2 + \widetilde{g}_\lambda(u)$. 
        We consider the first $L$ iterates generated by an algorithm designed to estimate $\pgx{x} = \prox_{\widetilde{g}_{\lambda}}(x) = \argmind{u} \bar{G}_\lambda(x, u)$, that is initialized using a warm-starting strategy, i.e., with $u_\ell$.
        Assume that the $L$-th estimate $u_{\ell + L}$ satisfies
        \begin{equation*}
            \bar{G}_\lambda(x, u_{\ell+L}) - \bar{G}_\lambda(x, \pgx{x}) \le \frac{C}{L}\|u_\ell - \pgx{x}\|^2
        \end{equation*} 
        where $C>0$ is a constant that does not depend on $x$. 
        Then, since $\widetilde{g}_\lambda$ is assumed to be convex, $\bar{G}_{\lambda}(x, \cdot)$ is $1$-strongly convex \cite[Prop.~10.8]{bauschke2011convex}, and we have 
        \begin{equation*}
            \|u_{\ell+L} - \pgx{x}\|_2^2 \le \frac{2 C}{L} \|u_\ell - \pgx{x}\|^2.
        \end{equation*}
        In this case, $L$ should be chosen larger than $4  C$ to satisfy assumption~\eqref{thm:bilev:ass}.
        Note that this is typically the case when $g$ is the SD formulation and we use FB as an inner-solver~\cite{BT09}.
        
        \item 
        When $\alpha_L$ is unknown, the root of the polynomial~\eqref{thm:bilevel:cond:pol} cannot be computed. In this case, one can choose $\phi =1$ in the Lyapunov function~\eqref{thm:bilevel:pr-lyap}, and \linebreak ${0 < \tau < \min \Big\{ \frac{2(1 - 2\alpha_L^2)}{\mu^2}, \frac{1}{\widetilde{\beta} + \mu + 4\alpha_L^2} \Big\}}$.
        Since $0 < \alpha_L^2 < 1/2$, if $\tau$ is chosen to satisfy ${0 < \tau < \min \big\{ 2/\mu^2, 1/(\widetilde \beta + \mu) \big\}}$, then ${(x_k)}_{k\in \mathbb N}$ generated by algorithm~\eqref{algo:algo_bi_level_prox} converges to a solution to~\eqref{pb:min:bilevel}. \\
        In practice, since $\mu$ is the smoothing parameter, it is often chosen very small. In this context, we recover a similar condition as for the convergence of~\eqref{algo:GD-moreau}, that is $0 < \tau < 1/(\mu + \widetilde \beta)$.
    \end{enumerate}
\end{remark}

% --------------------------------------------------------------------
% --------------------------------------------------------------------

\section{Numerical experiments}
\label{Sec:exp}
We investigate the behavior of algorithm~\eqref{algo:FB-PnP}, in the settings described in Section~\ref{Sec:AD-PnP} and Section~\ref{Sec:SD-PnP}, i.e. with $G$ built as sub-iterations for solving either the analysis or the synthesis Gaussian denoising problem. 

In~Section~\ref{subsec:pnp-results} we study the stability of the FB-PnP iterations described in~Section~\ref{Sec:AD-PnP} and Section~\ref{Sec:SD-PnP}. In particular, we aim to illustrate results presented in Theorem~\ref{prop:analysis_warm_restart} and in Theorem~\ref{thm:cvg-SF-L1}, for the analysis and synthesis denoisers (AD and SD, respectively). This study is conducted on a toy compressed sensing example. 

In~Section~\ref{subsec:deblur} we will investigate the behavior of the proposed approaches in a deep dictionary learning (DDL) framework, where the linear operators $\Dica$ and $\Dics$ are learned. We will apply the resulting DDL methods to a deblurring image problem, and for the sake of completeness we will show that they are very competitive with more advanced PnP methods where the denoiser is a full neural network (namely DRUnet).

Codes to reproduce the results presented in this section are available on \href{https://github.com/tomMoral/dictionary-based_denoisers_FBPnP}{GitHub}.

% --------------------------------------------------------------------
% --------------------------------------------------------------------

\subsection{Stability analysis of the FB-PnP iterations}
\label{subsec:pnp-results}

In this section we consider a simple compressive sensing example, and we aim to analyze the stability of the FB-PnP iterations when using the proposed AD or SD, depending on the number of layers $L$. Specifically, we want to illustrate the results presented in Theorem~\ref{prop:analysis_warm_restart} and in Theorem~\ref{thm:cvg-SF-L1}. Although these results have only been shown for $L=1$ or $L\to \infty$ (i.e., the proximity operators are computed accurately), we will also empirically investigate the behavior of the AD and SD when $1 < L < +\infty$.
To this aim, we consider a toy example consisting of an equation system of the form of $y = A \overline{x}$, where
$\overline{x} \in \Rset^N$ is a vector of dimension $N=50$ generated randomly following a uniform distribution,
$A \colon \Rset^N \to \Rset^M$ is a Gaussian measurement operator with $M=20$, where coefficients are generated randomly following a Gaussian distribution with mean $0$ and standard deviation $1$. 
Both the analysis and synthesis dictionaries are generated as Gaussian dictionaries with mean $0$ and standard deviation $1$, of size $\Dica \in \Rset^{100 \times 50}$ and $\Dics \in \Rset^{50 \times 100}$, respectively.

% --------------------------------------------------------------------
% --------------------------------------------------------------------

\subsubsection{Analysis denoiser}
\label{subsec:convergence_L_analysis}

\begin{figure}[b!]
    \centering
    \includegraphics[width=\textwidth]{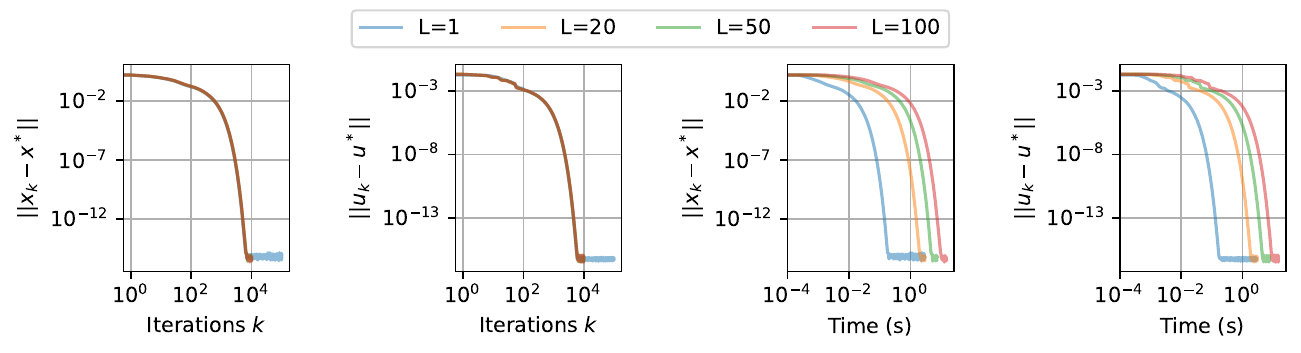}

    \vspace{-0.5cm}
    \caption{Convergence profile of algorithm~\eqref{algo:FB-pnp-AD} with $L\in \{1, 20, 50, 100\}$. Solutions $(x^*, u^*)$ are pre-computed considering algorithm~\eqref{algo:FB-pnp-AD} with $L \to \infty$. 
    %The dictionary $\Dica$ is learned with $L=1$.
    }
    % The number of layers does not change anything to the result of the algorithm when the dictionary is fixed.}
    \label{fig:convergence_analysis}
\end{figure}

In Theorem~\ref{prop:analysis_warm_restart} we showed that taking $L=1$ in algorithm~\eqref{algo:FB-pnp-AD} is equivalent to using a primal-dual algorithm, namely the Loris-Verhoeven algorithm \cite{loris2011generalization}. In particular, the resulting algorithm, given in \eqref{algo:FB-pnp-ana-1}, leads to the same asymptotic solution as when taking $L\to\infty$ in algorithm~\eqref{algo:FB-pnp-AD}. In this section we aim to illustrate this result by comparing the output of algorithm~\eqref{algo:FB-pnp-ana-1} (i.e., $L=1$) with the output of algorithm~\eqref{algo:FB-pnp-AD} for $L=10^4$. 
Furthermore, although the results from Theorem~\ref{prop:analysis_warm_restart} only hold for $L=1$, we show empirically that they remain true for $1 < L < +\infty$ for our problem by considering $L \in \{20, 50, 100\}$.

In Figure~\ref{fig:convergence_analysis} we show 
the behavior of $(\|x_k - x^*\|)_{k \in \mathbb N}$ and $(\|u_k - u^*\|)_{k \in \mathbb N}$ with $(x_k, u_k)_{k\in \mathbb N}$ generated by algorithm~\eqref{algo:FB-pnp-AD} with $L\in \{1, 20, 50, 100\}$, and $(x^*, u^*)$ generated by algorithm~\eqref{algo:FB-pnp-AD} with $L\to\infty$ (obtained after $10^4$ iterations). Note that $x^*$ is solution to problem~\eqref{pb:inv-min-AF}. 
We see that the trajectory per iteration is the same independently of the value of $L$. Hence, according to Theorem~\ref{prop:analysis_warm_restart} we can deduce that $(x^*, \tau^{-1} u^*)$ is solution to the saddle-point problem~\eqref{prop:analysis_warm_restart:SPP}.
We can further observe that the convergence is faster in time for smaller values of $L$.

% --------------------------------------------------------------------
% --------------------------------------------------------------------

\subsubsection{Synthesis denoiser}
\label{subsec:convergence_L_synthesis}

We showed in Theorem~\ref{thm:cvg-SF-L1} that taking $L=1$ in algorithm~\eqref{algo:FB-pnp-SD} enables solving \eqref{pb:inv-min-SF} as well as the direct sparse coding problem~\eqref{pb:min_gen-synth}. 
Since algorithm~\eqref{algo:FB-pnp-SD} for $L\to \infty$ solves \eqref{pb:inv-min-SF} (Theorem~\ref{thm:cvgce-pnp-limit-SD}), we deduce that both algorithm~\eqref{algo:FB-pnp-SD} for $L\to \infty$ and algorithm~\eqref{algo:FB-pnp-synthesis-1} (i.e., $L=1$) solve the same problem~\eqref{pb:inv-min-SF}.
In this section we aim to illustrate this result by comparing the output of algorithm~\eqref{algo:FB-pnp-synthesis-1} with the output of algorithm~\eqref{algo:FB-pnp-SD} for $L=10^4$. 
Furthermore, although the results from Proposition~\ref{prop:synthesis_as_prox_first} only hold for $L=1$, we show empirically that they remain true for $1 < L < +\infty$ for our problem by considering $L \in \{20, 50, 100\}$.

In Figure~\ref{fig:convergence_synthesis} we show 
the behavior of $(\|x_k - x^*\|)_{k \in \mathbb N}$ and $(\|z_k - z^*\|)_{k \in \mathbb N}$ with $(x_k, z_k)_{k\in \mathbb N}$ generated by algorithm~\eqref{algo:FB-pnp-SD} with $L\in \{1, 20, 50, 100\}$, and $(x^*, z^*)$ generated by algorithm~\eqref{algo:FB-pnp-SD} with $L\to\infty$ (obtained after $10^4$ iterations). Note that $x^*$ is solution to problem~\eqref{pb:inv-min-SF}. 
We see that the trajectory per iteration is fairly the same when $L \in \{20, 50, 1000\}$, while $L=1$ is taking a bit more iterations to reach convergence. 
However, for the convergence speed, all strategies seems to behave similarly, with $L\in \{1, 20\}$ being very slightly faster.

\begin{figure}
    \centering
    \includegraphics[width=\textwidth]{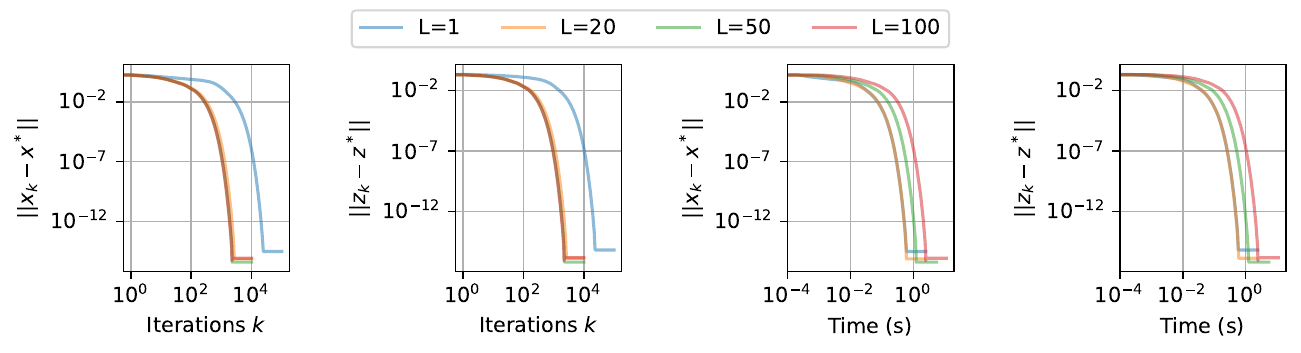}

    \vspace{-0.5cm}
    \caption{Convergence profile of algorithm~\eqref{algo:FB-pnp-SD} with $L\in \{1, 20, 50, 100\}$. Solutions $(x^*, u^*)$ are pre-computed considering algorithm~\eqref{algo:FB-pnp-SD} with $L \to \infty$. 
    %The dictionary $\Dics$ is learned with $L=1$.
    }
    \label{fig:convergence_synthesis}
\end{figure}

\subsection{Application to image deblurring within a DDL framework}
\label{subsec:deblur}

We now consider an image restoration problem to illustrate results presented in Sections~\ref{Sec:AD-PnP} and~\ref{Sec:SD-PnP}.
Although the results presented in this work are generic and can be applied for any linear operators $\Dica$ and $\Dics$,
they are of particular interest in the DDL framework, where these operators are learned. 
We will explore this case in this section, and apply the resulting methods to an image deblurring problem. For the sake of completeness we will compare the resulting algorithms to a FB-PnP algorithm where the denoiser is a DRUnet.

\subsubsection{DDL context and training} 
\label{Ssec:exp:ddl-train}

As mentioned above, in this section, we focus on a deep learning approach, where the dictionaries $\Dica$ and $\Dics$ (see Section~\ref{Sec:AD-PnP} and Section~\ref{Sec:SD-PnP} for details) are learned as Gaussian denoisers, within a supervised context. 
Such an approach using unrolled algorithms for dictionary learning in computational imaging has been investigated in multiple works of the literature, either for the analysis formulation~\cite{chambolle2020learning, LPM20} or the synthesis formulation~\cite{TDB20, SEM19}, leading to the so-called DDL framework. 

In this section, we describe the training process for $\Dica$ and $\Dics$.
The objective is to find an estimate $x^\dagger \in \Rset^N$ of an original unknown signal $\overline{x} \in \Rset^N$ that has been corrupted by
some additive white Gaussian noise. Hence the noisy observed signal is given by
\begin{equation}\label{pb:inv-den}
    v = \overline{x} + {\varepsilon} w,
\end{equation}
where $w \in \mathbb R^N$ is a realization of an i.i.d. standard normal random variable, and ${\varepsilon}>0$ is the noise level. 
Both analysis and synthesis formulations to find the estimate of $\overline{x}$ from $v$ in \eqref{pb:inv-den} 
(see \eqref{pb:AF-intro} and \eqref{pb:SF-inner}, respectively)
have been extensively studied in the literature (see e.g.~\cite{peyre2011learning, YNGD13} for the analysis, and~\cite{Olshausen1997, AEB06, MBPS09} for the synthesis formulation). 
In the context of DDL, these problems can be reformulated as a joint estimation of the image and the dictionary, 
or as bi-level optimization problems~\cite{MBPS09}. Then, the dictionaries are defined as solutions to
\begin{align}
    \Dica^\dagger
    = \argmind{\Dica \in \Rset^{S \times N}}
    \ell^{\Am} \big( x^\dagger_{\Dica} \big)
    \quad &\text{ where } \quad
    x^\dagger_{\Dica} \text{ solution to~\eqref{pb:AF-intro}}
    \label{eqn:opt_analysis}
\end{align}
and
\begin{align}
    \Dics^\dagger
    = \argmind{\Dics \in \Rset^{N \times S}}
    \ell^{\Sm} \big( z^\dagger_{\Dics} \big)
    \quad &\text{ where } \quad
    z^\dagger_{\Dics}  \text{ solution to~\eqref{pb:SF-inner}}
    \label{eqn:opt_synthesis}
\end{align}
respectively, where $\ell^{\Am}$ and $\ell^{\Sm}$ are some loss functions.
In this context, the computation of the signal $x^\dagger_{\Dica}$ or the sparse codes $z^\dagger_{\Dics}$ is often referred to the \textit{inner problem}, while the global minimization for estimating the dictionary $\Dica$ or $\Dics$
is called the \textit{outer problem}. 

If in~\eqref{eqn:opt_analysis} (resp.~\eqref{eqn:opt_synthesis}), the loss function $\ell^{\Am}(\Dica)$ is the same as the one minimized in~\eqref{pb:AF-intro} (resp. if $\ell^{\Sm}(\Dics)$ is the same as the one minimized in~\eqref{pb:SF-intro}),
then the bi-level problem is equivalent to the DDL joint estimation problem. 
This choice is often used in an unsupervised setting.
In a supervised setting, where a ground truth dataset can  be used, a standard choice for $\ell^{\Am}$ and $\ell^{\Sm}$ is the
mean square error (MSE) loss, minimizing the $\ell^2$ norm of the difference between the ground truth and the estimated
signal $x^\dagger_{\Dica}$ or $x^\dagger_{\Dics}$. 
In the remainder, we will use this approach.
Specifically, we aim to learn $\Dica$ and $\Dics$ such that, for some initializations $u_0\in \Rset^S$ and $z_0\in \Rset^S$, we have
\begin{equation}\label{eq:def-unroll-learn}
    \begin{cases}
        x^\dagger_{\Am} = \NN^{\Am}_{L,\lambda, v}({u_0}) \approx \overline{x}, \\
        x^\dagger_{\Sm} = \NN^{\Sm}_{L,\lambda
        , v}({z_0}) \approx \overline{x}
    \end{cases}
\end{equation}
where $\NN^{\Am}$ and $\NN^{\Sm}$ are defined in Model~\ref{mod:AF} and Model~\ref{mod:SF}, respectively, for a fixed number of iterations $L\ge 1$, a regularization parameter $\lambda>0$, and input noisy image $v$ obtained as per~\eqref{pb:inv-den}. We assume that dictionaries $\Dica$ and $\Dics$ are parametrized by learnable parameters $\theta_{\Dica}$ and $\theta_{\Dics}$, respectively (i.e., convolution kernel).
Then, given a dataset of pairs of clean/noisy images ${(\overline{x}_i, v_i)}_{ i \in \mathbb{S}_{\mathcal{T}}}$ obtained following~\eqref{pb:inv-den}, the unrolled denoisers in \eqref{eq:def-unroll-learn} are trained by solving, respectively 
\begin{equation*}
    \minimize{\theta_{\Dica}} \frac{1}{\#(\mathbb{S}_{\mathcal{T}})} \! \sum_{i \in \mathbb{S}_{\mathcal{T}}} \! \| \NN^{\Am}_{\Ltrain, \lambda, v_i}(0_S) -  \overline{x}_i \|,
    \,\text{ and }\,
    \minimize{\theta_{\Dics}} \frac{1}{\#(\mathbb{S}_{\mathcal{T}})} \! \sum_{i \in \mathbb{S}_{\mathcal{T}}} \! \| \NN^{\Sm}_{\Ltrain, \lambda, v_i}(0_S) -  \overline{x}_i \|,
\end{equation*}
where the learnable parameter $\theta_{\Dica}$ (resp. $\theta_{\Dics}$) is associated with dictionary $\Dica$ (resp. $\Dics$) in $\NN^{\Am}_{\Ltrain, \lambda, v_i}$ (resp. $\NN^{\Sm}_{\Ltrain, \lambda, v_i}$), and $0_S \in \Rset^S$ is a constant vector with $0$ on all coefficients. 
For both the AD and the SD we choose the initialization $u_0 = z_0 =0_S$, which is a standard initialization as the dual variables should be sparse. 
We trained our models on patches of size $150 \times 150$ of images from the training dataset from the BSDS500 image bank
and we chose a noise level for the training of ${\varepsilon}=0.05$. 
The learnable parameters $\theta_{\Dica}$ and $\theta_{\Dics}$ correspond to $50$ convolutional filters of dimension $5 \times 5$, and we fix the regularization parameter $\lambda=10^{-2}$ and step sizes $\sigma = \frac{1.8}{||\Gamma||^2}$ in Model~\ref{mod:AF} and $\zeta=\frac{1.8}{||D||^2}$ in Model~\ref{mod:SF}. 
Since our study includes the cases of denoisers with either only $L=1$ iteration, or a large number of iterations, in our simulations we tested two strategies: training the denoisers with (i) only $\Ltrain=1$ iteration, or (ii) with $\Ltrain=20$ iterations sharing the same dictionary.

\smallskip

\textbf{Denoising performances -- }
In Table~\ref{tab:inference-time} we give the inference time for denoising an image, when applying the AD and SD with either $\Lden=1$ or $\Lden=20$ iterations. For comparison, we also provide the inference time when using a DRUnet, where the weights have been taken from the \href{https://deepinv.github.io/deepinv/}{DeepInv} PyTorch library. 
We further provide examples of output images for the different denoisers in Figure~\ref{fig:denoise-example}, evaluated in different configurations. It can be observed that the cases where SD is evaluated with $\Lden=1$ iteration (SD train $\Ltrain=\{1,20\}$ and $\Lden=1$) lead to very poor results. When SD is trained with $\Ltrain=1$ and evaluated with $\Lden=\{20,1000\}$ iterations, results start to improve, although remaining quite poor. %Additional simulations have shown that increasing $\Lden$ in the evaluation continues improving the output. 
All other cases look very similar, with DRUnet giving best denoising results.

\begin{table}
    \centering
    \caption{\label{tab:inference-time}
    Average time and standard deviation when applying different denoisers. 
    Second row shows time to evaluate the denoisers with either $\Lden=1$ or $\Lden=20$, computed over $150$ images from BSDS500 cropped to $150 \times 150$, with noise level ${\varepsilon}=0.05$. 
    Third row shows time to run $1,000$ iterations of the PnP with either $\Lpnp=1$ or $\Lpnp=20$, computed over $4$ examples.} 
    \footnotesize
    \arrayrulecolor{black}
    \begin{tabular}{l||rr|rr|r}
        \toprule
        \multirow{ 2}{*}{Denoiser}
        & \multicolumn{2}{c|}{AD}
        & \multicolumn{2}{c|}{SD}
        & \multirow{ 2}{*}{DRUNet} \\
        & 1 iteration & 20 iterations
        & 1 iteration & 20 iterations &
        \\
        \hline
        Average time per evaluation 
        & \multirow{ 2}{*}{$0.19  (\pm 0.08)$} & \multirow{ 2}{*}{$3.94  (\pm 1.46)$} 
        & \multirow{ 2}{*}{$0.27  (\pm 0.16)$}  & \multirow{ 2}{*}{$3.12  (\pm 0.96)$} 
        & \multirow{ 2}{*}{$5.32  (\pm 0.76)$}
        \\
        \centering $\times 10^{-2}$ sec. ($\pm$std) && && \\
        \hline
        Average time for PnP run
        & \multirow{ 2}{*}{$1.32  (\pm 0.03)$} & \multirow{ 2}{*}{$31.07  (\pm 0.06)$}  
        & \multirow{ 2}{*}{$1.18  (\pm 0.02)$}  & \multirow{ 2}{*}{$18.57  (\pm 0.06)$} 
        & \multirow{ 2}{*}{$49.24  (\pm 0.19)$}
        \\
        \centering ($\pm$std) && && \\
        \bottomrule
    \end{tabular}

\end{table}

\begin{figure}
    \centering\scriptsize\setlength\tabcolsep{0.05cm}
    \begin{tabular}{cccccc}
        & $\phantom{-}$ &  \multicolumn{2}{c}{SD with $\Ltrain=1$} & \multicolumn{2}{c}{SD with $\Ltrain=20$} \\[-0.05cm]
        $\overline{x}$  
        & $\phantom{.}$ 
        &   $\Lden=1$ & $\Lden=20$ 
        & $\Lden=1$ & $\Lden=20$ \\
        % ---------------------------------
        \includegraphics[width=2.3cm, trim={2.6cm 0.5cm 2.6cm 0.5cm}, clip]{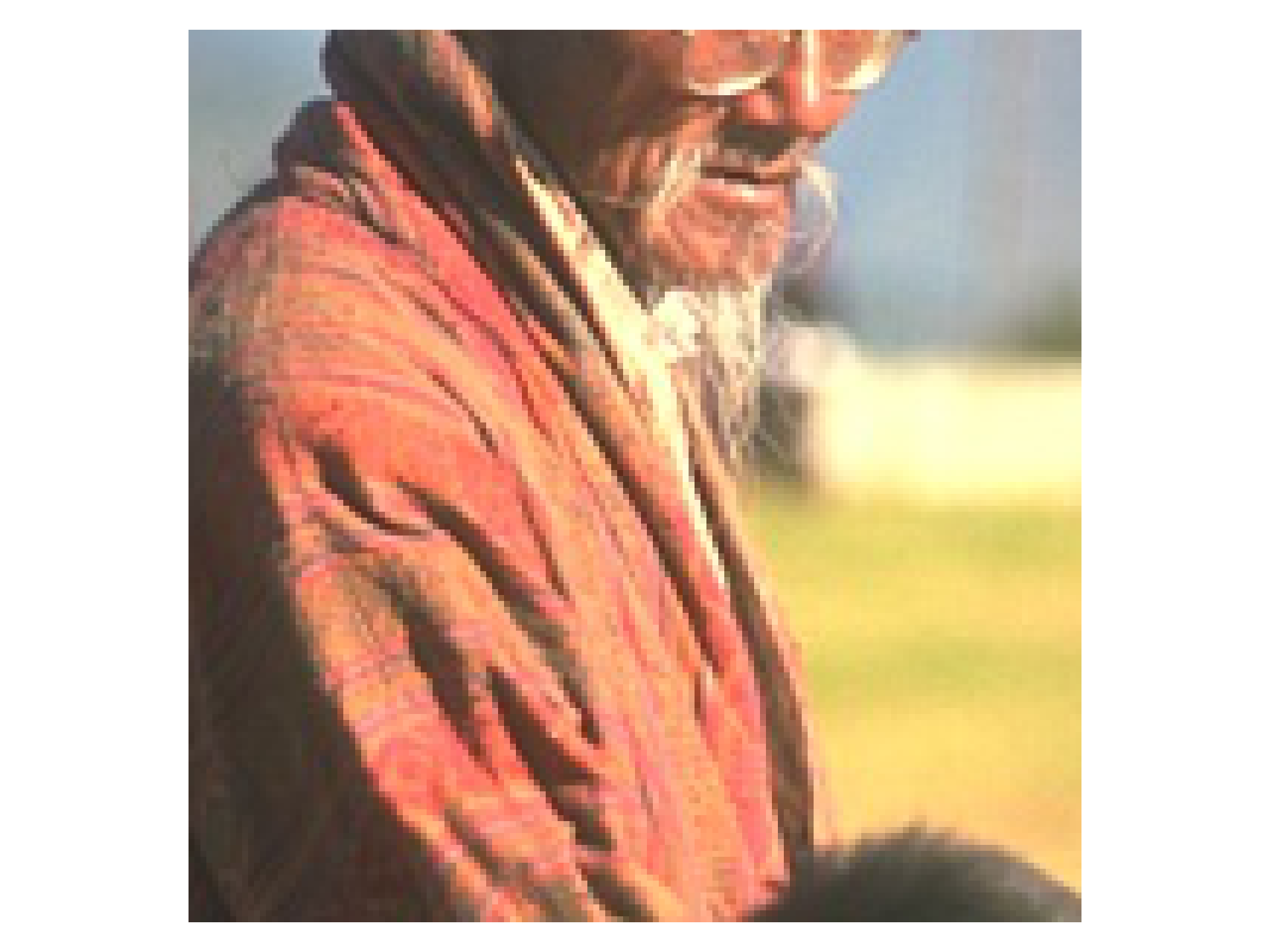}
        & $\phantom{.}$ 
        &   \includegraphics[width=2.3cm, trim={2.6cm 0.5cm 2.6cm 0.5cm}, clip]{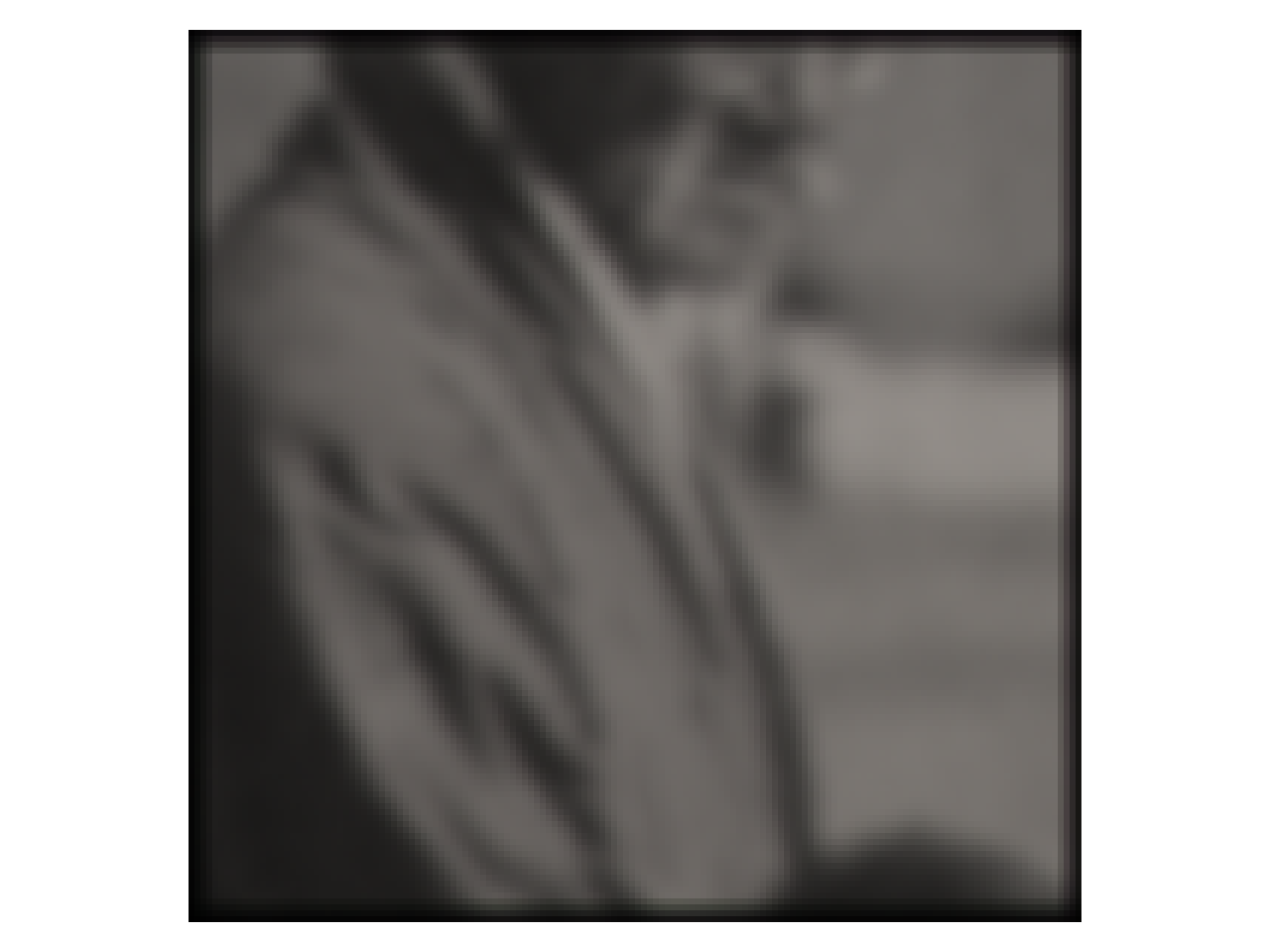}
        &   \includegraphics[width=2.3cm, trim={2.6cm 0.5cm 2.6cm 0.5cm}, clip]{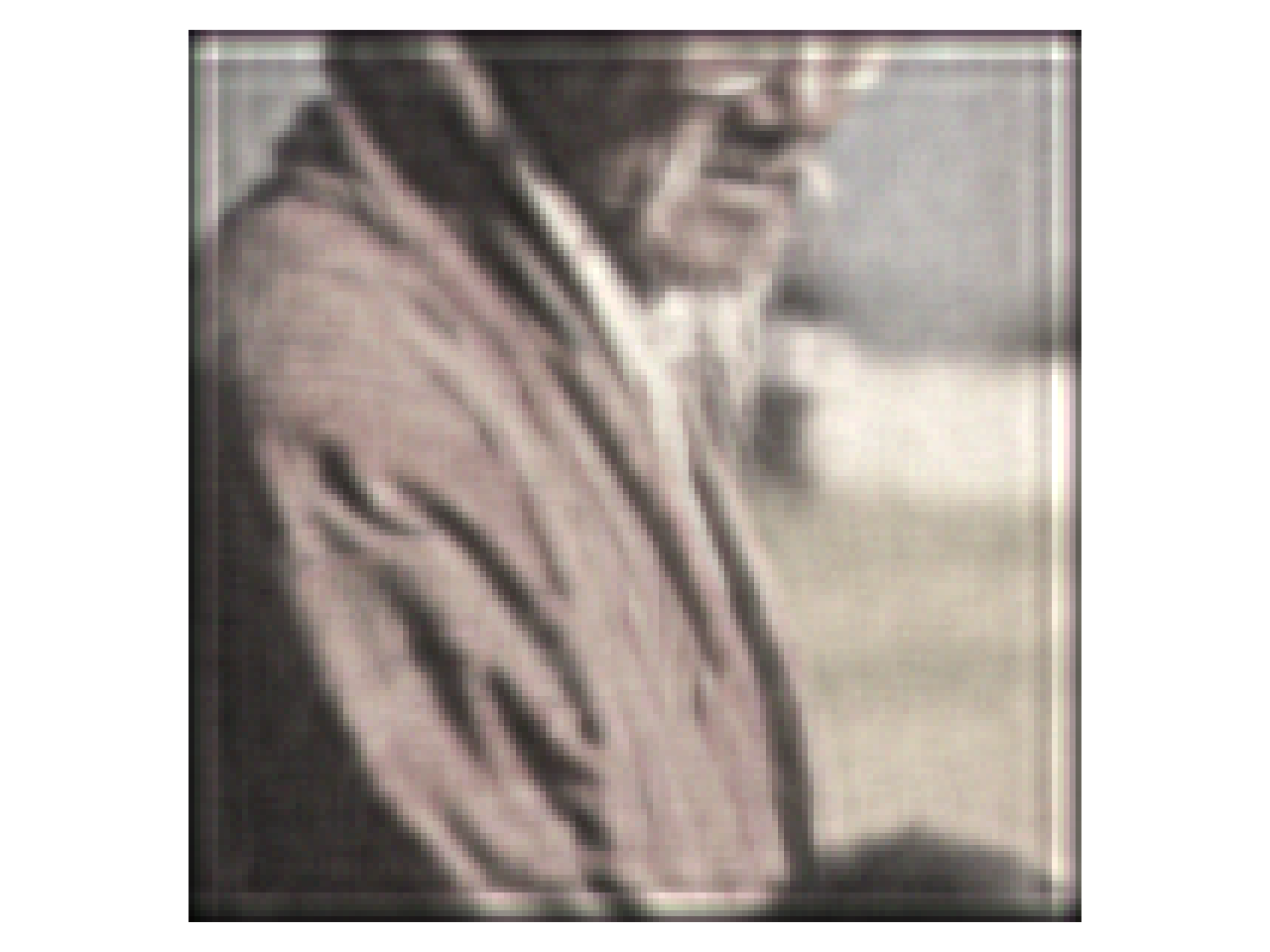}
        &   \includegraphics[width=2.3cm, trim={2.6cm 0.5cm 2.6cm 0.5cm}, clip]{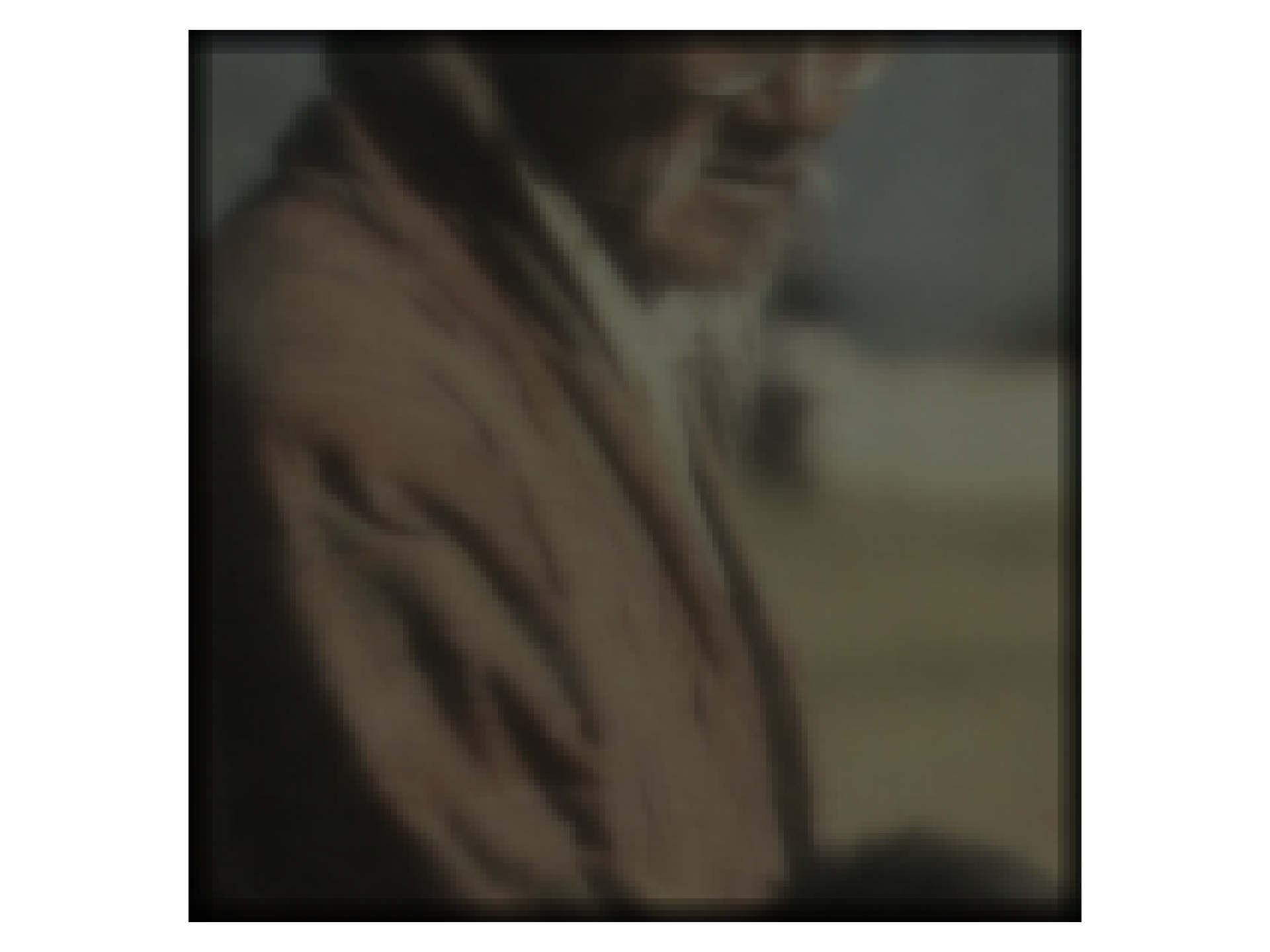}
        &   \includegraphics[width=2.3cm, trim={2.6cm 0.5cm 2.6cm 0.5cm}, clip]{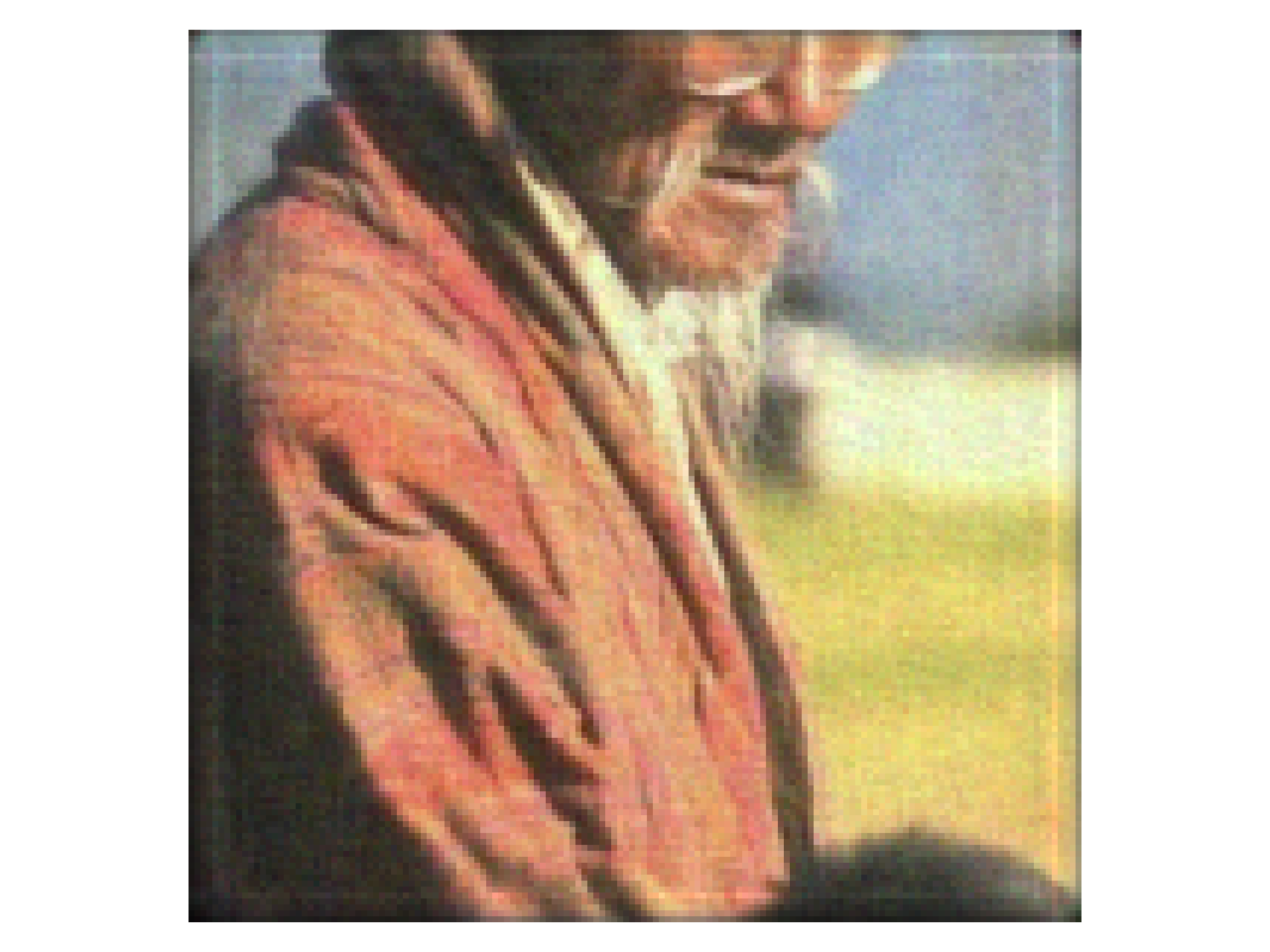}
        \\[0.1cm]
        %
        % ---------------------------------
        % ---------------------------------
        & $\phantom{.}$ &  \multicolumn{2}{c}{AD with $\Ltrain=1$} & \multicolumn{2}{c}{AD with $\Ltrain=20$} \\[-0.05cm]
        $v$ ($\varepsilon=0.05$)
        & $\phantom{.}$ 
        &   $\Lden=1$ & $\Lden=20$ & $\Lden=1$ & $\Lden=20$ \\
        % ---------------------------------
        \includegraphics[width=2.3cm, trim={2.6cm 0.5cm 2.6cm 0.5cm}, clip]{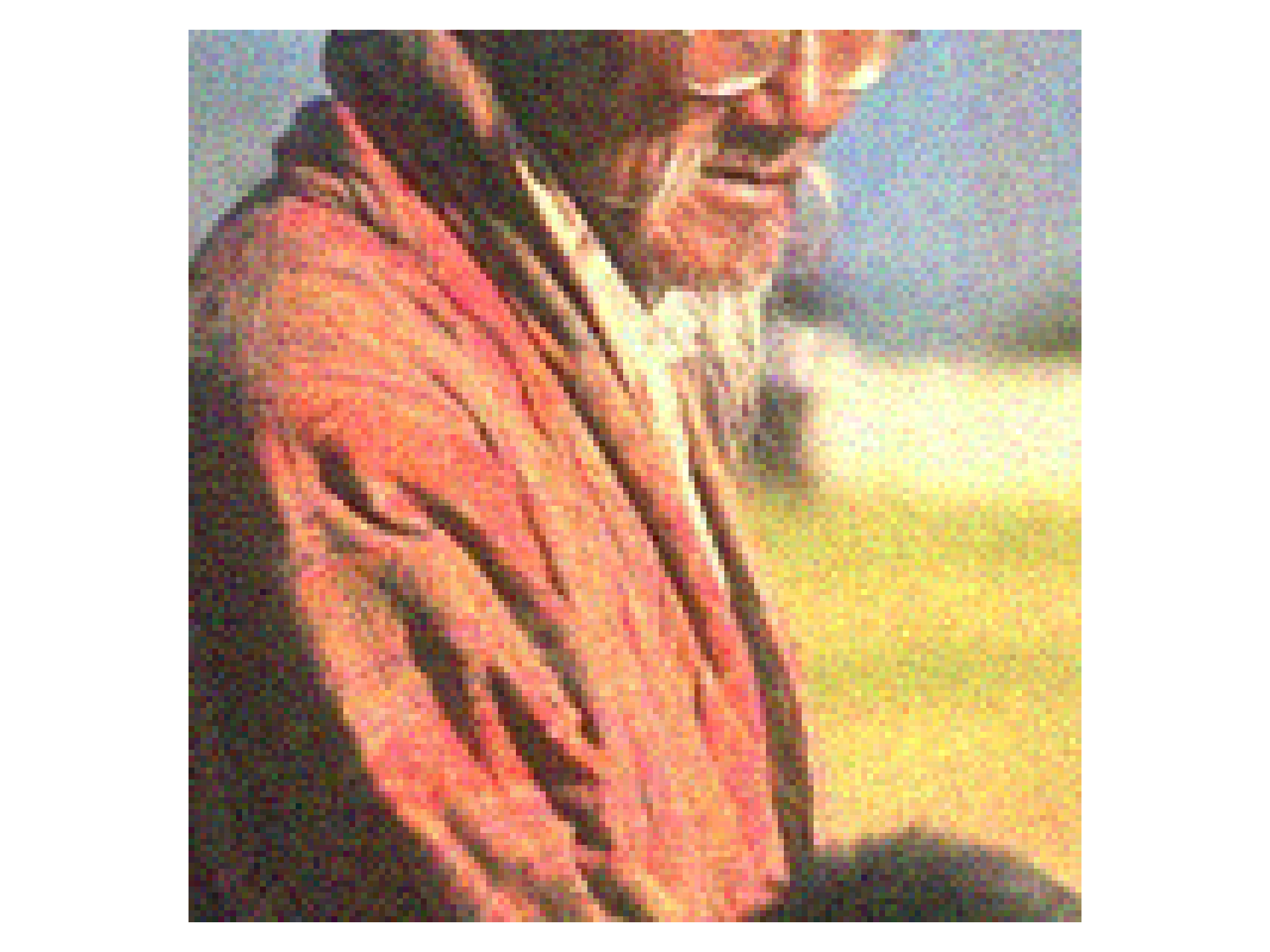}
        & $\phantom{.}$ 
        &   \includegraphics[width=2.3cm, trim={2.6cm 0.5cm 2.6cm 0.5cm}, clip]{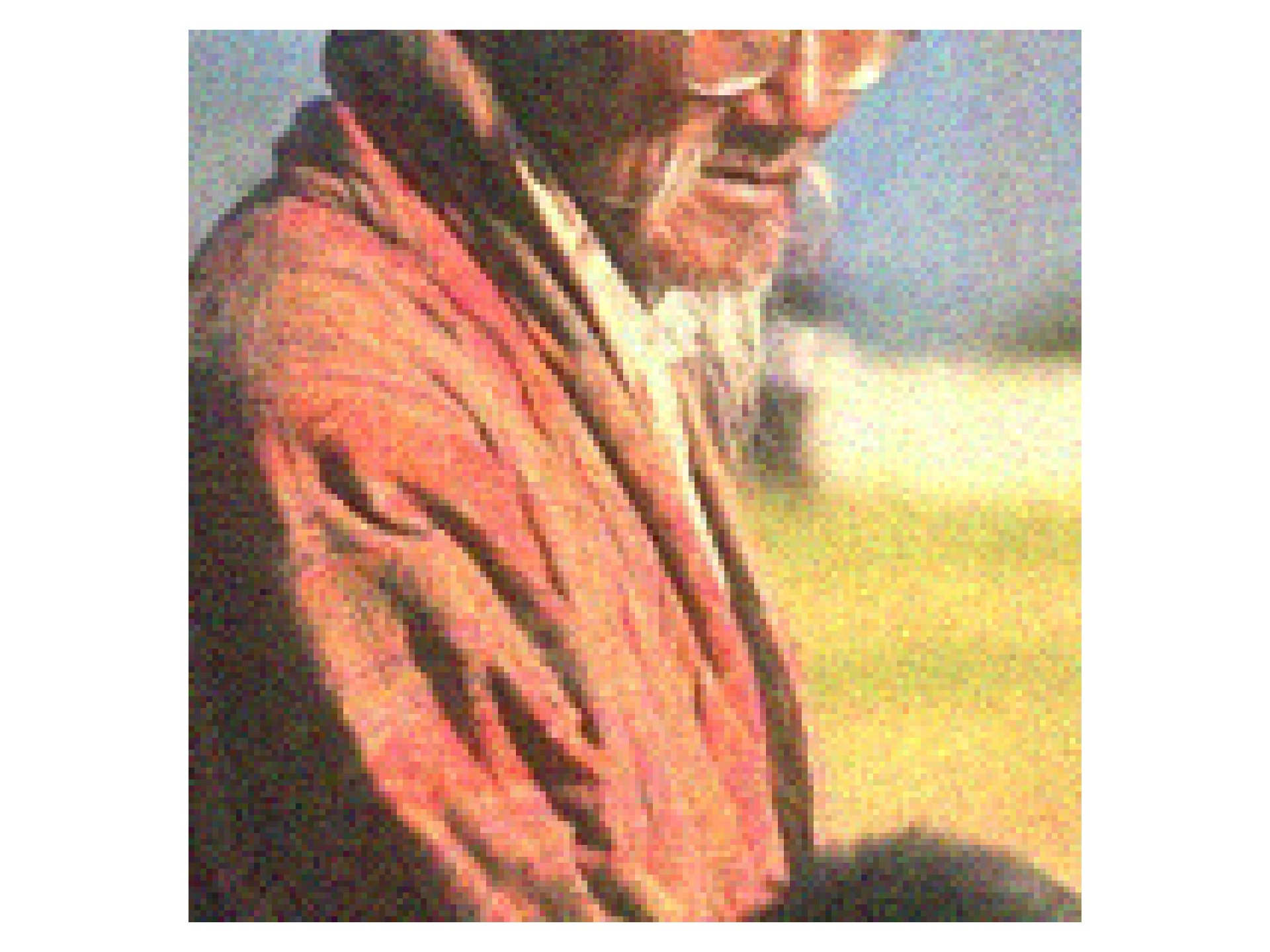}
        &   \includegraphics[width=2.3cm, trim={2.6cm 0.5cm 2.6cm 0.5cm}, clip]{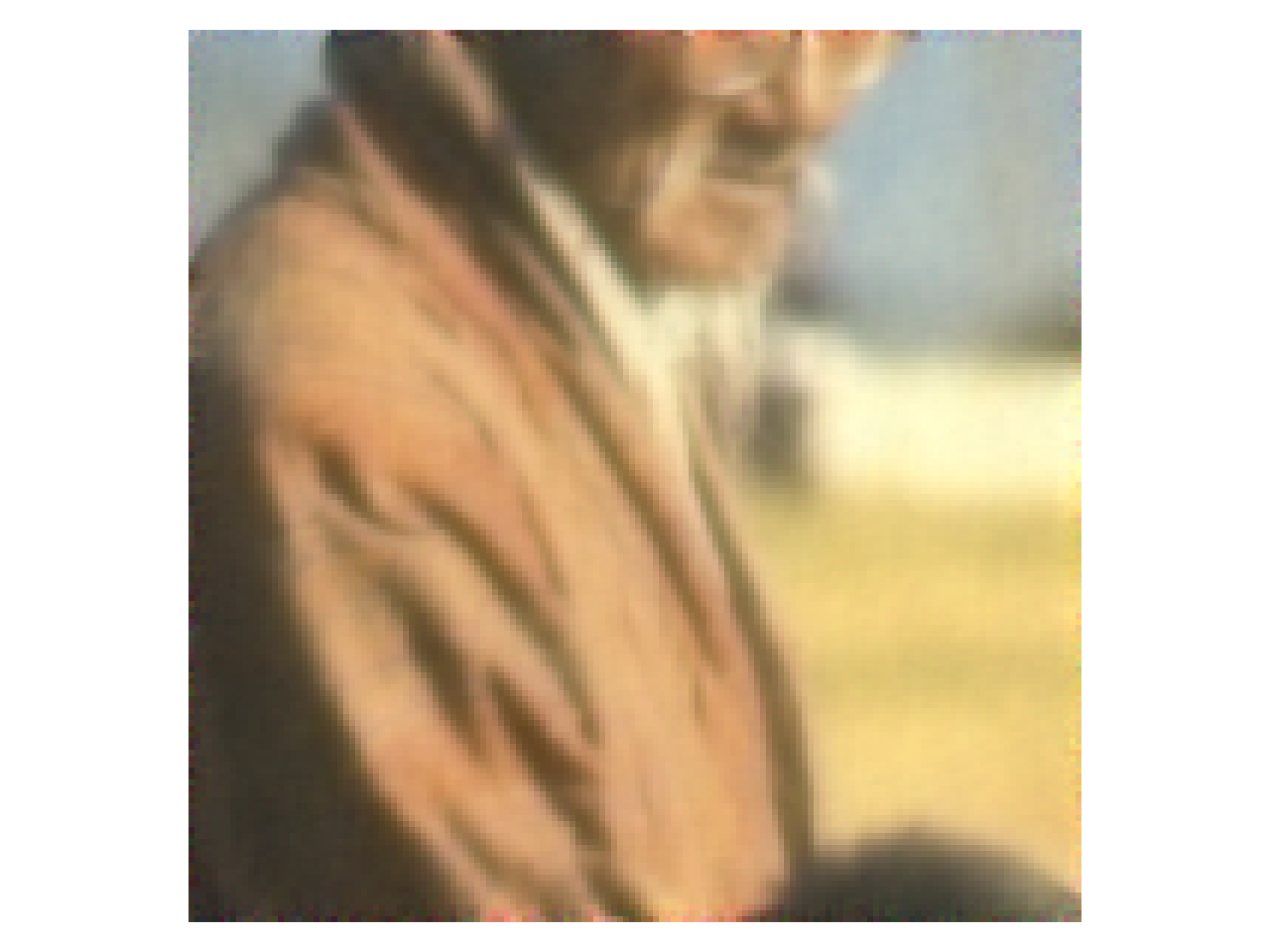}
        &   \includegraphics[width=2.3cm, trim={2.6cm 0.5cm 2.6cm 0.5cm}, clip]{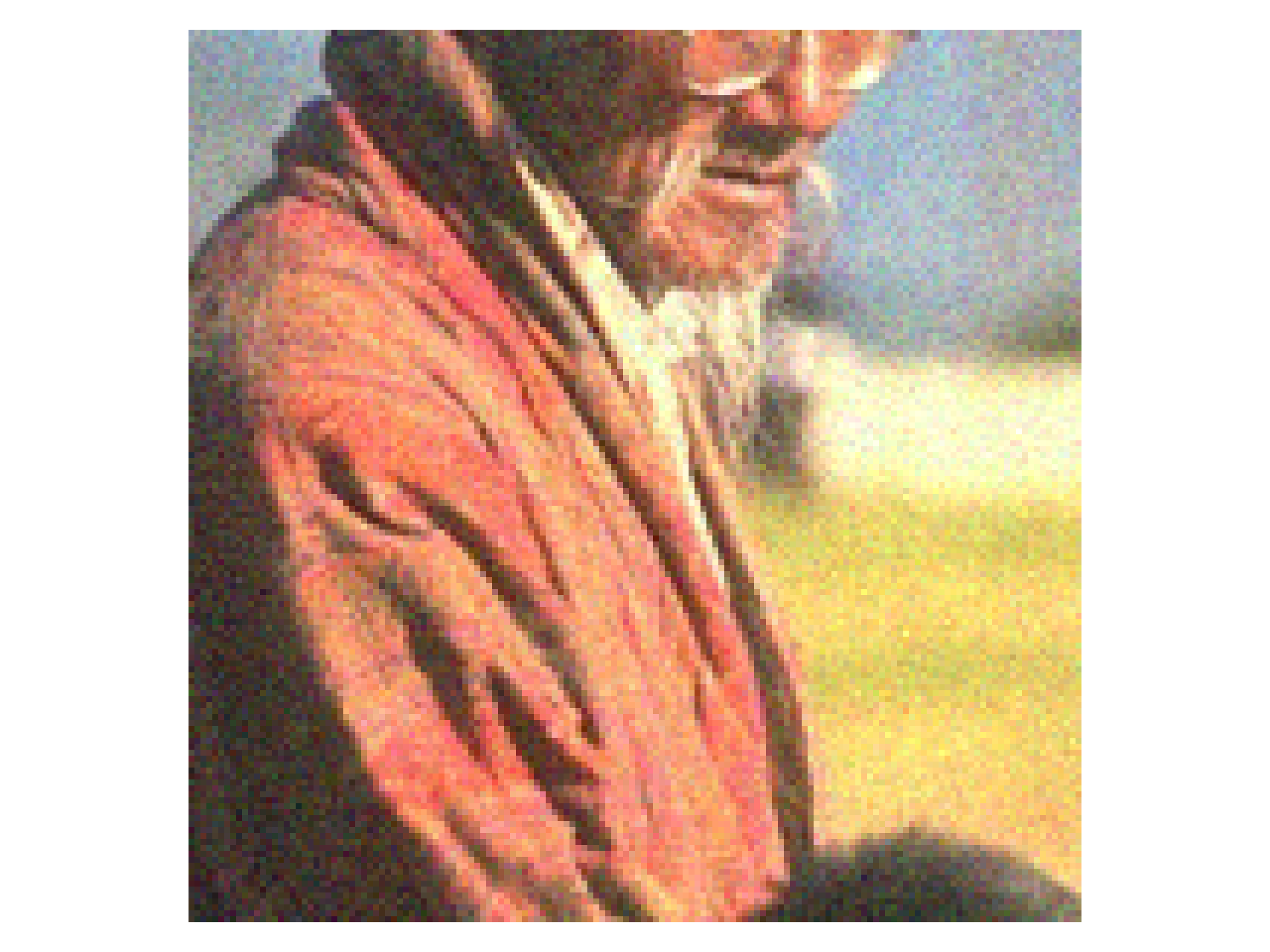}
        &   \includegraphics[width=2.3cm, trim={2.6cm 0.5cm 2.6cm 0.5cm}, clip]{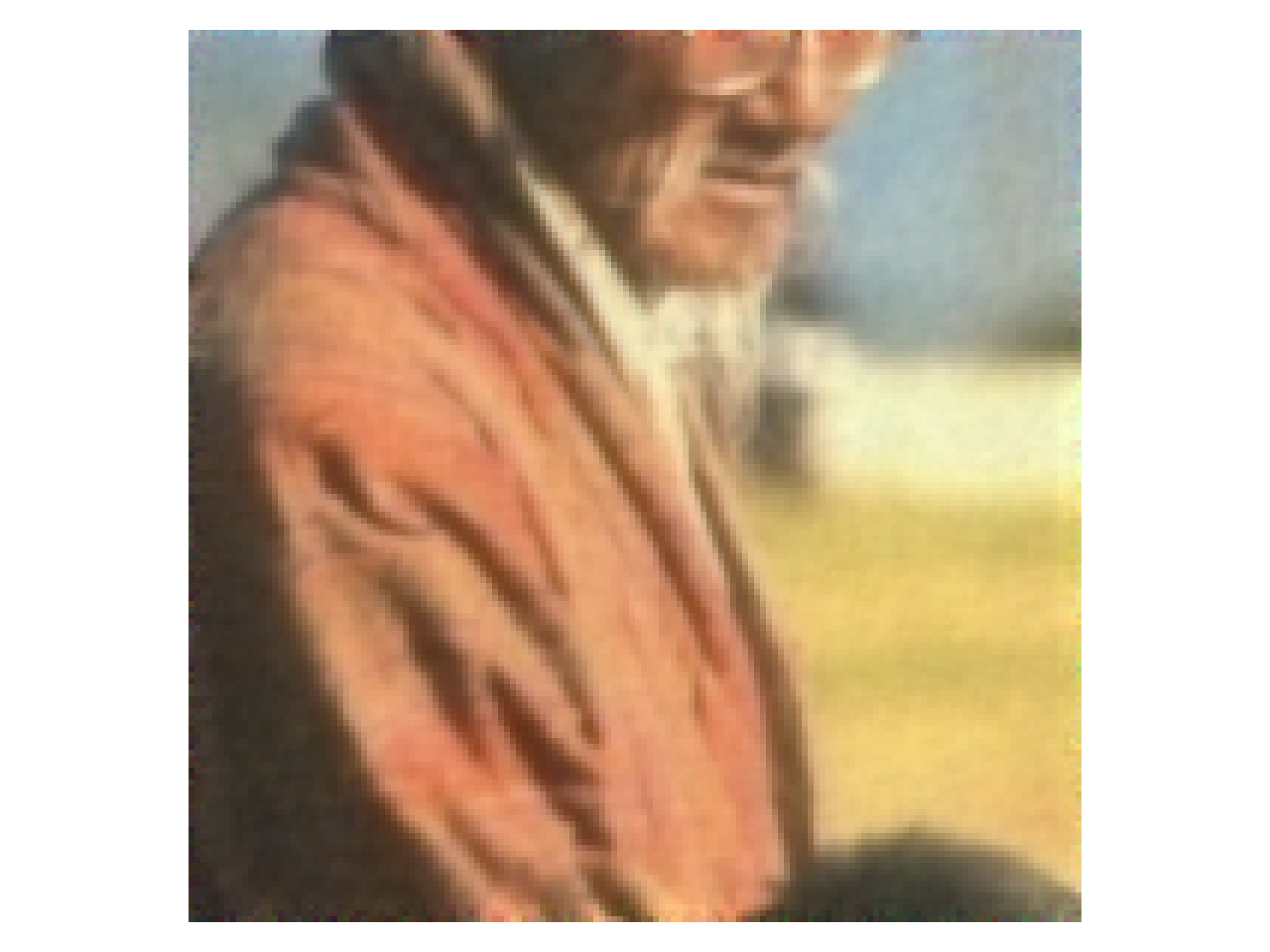}
        \\[0.3cm]
        % ---------------------------------
        % ---------------------------------
        & $\phantom{.}$ & \multicolumn{4}{c}{Denoising for $\Lden= 1,000$} \\
        DRUnet
        & $\phantom{.}$ 
        &   SD ($\Ltrain=1$) &  SD ($\Ltrain=20$) 
        &   AD ($\Ltrain=1$) &  AD ($\Ltrain=20$) \\
        % ---------------------------------
        $35.8$dB 
        && $26.7$dB & $26.6$dB 
        & $29.0$dB & $31.4$dB
        \\
        % ---------------------------------
        \includegraphics[width=2.3cm, trim={2.6cm 0.5cm 2.6cm 0.5cm}, clip]{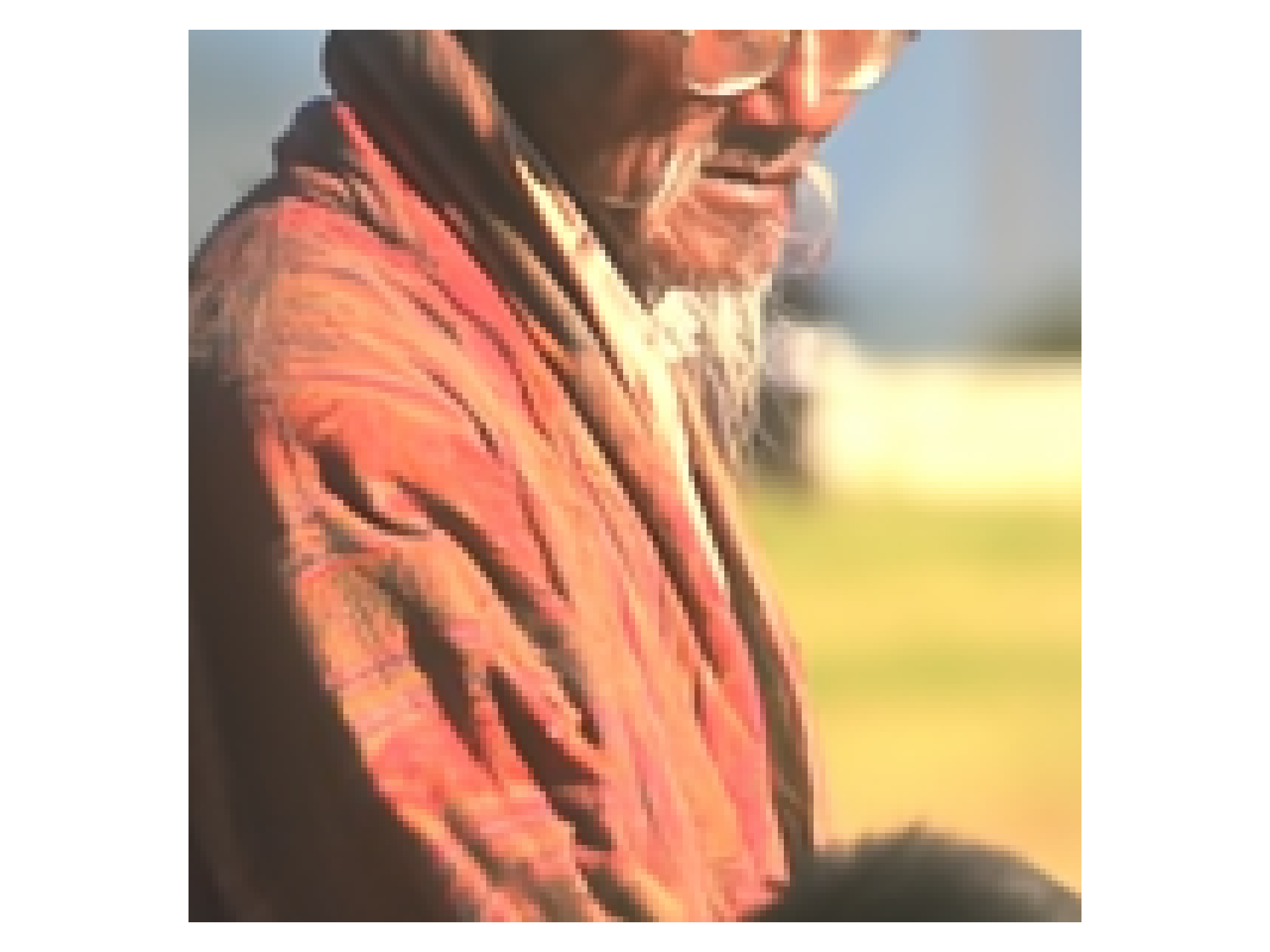}
        & $\phantom{.}$ 
        &  \includegraphics[width=2.3cm, trim={2.6cm 0.5cm 2.6cm 0.5cm}, clip]{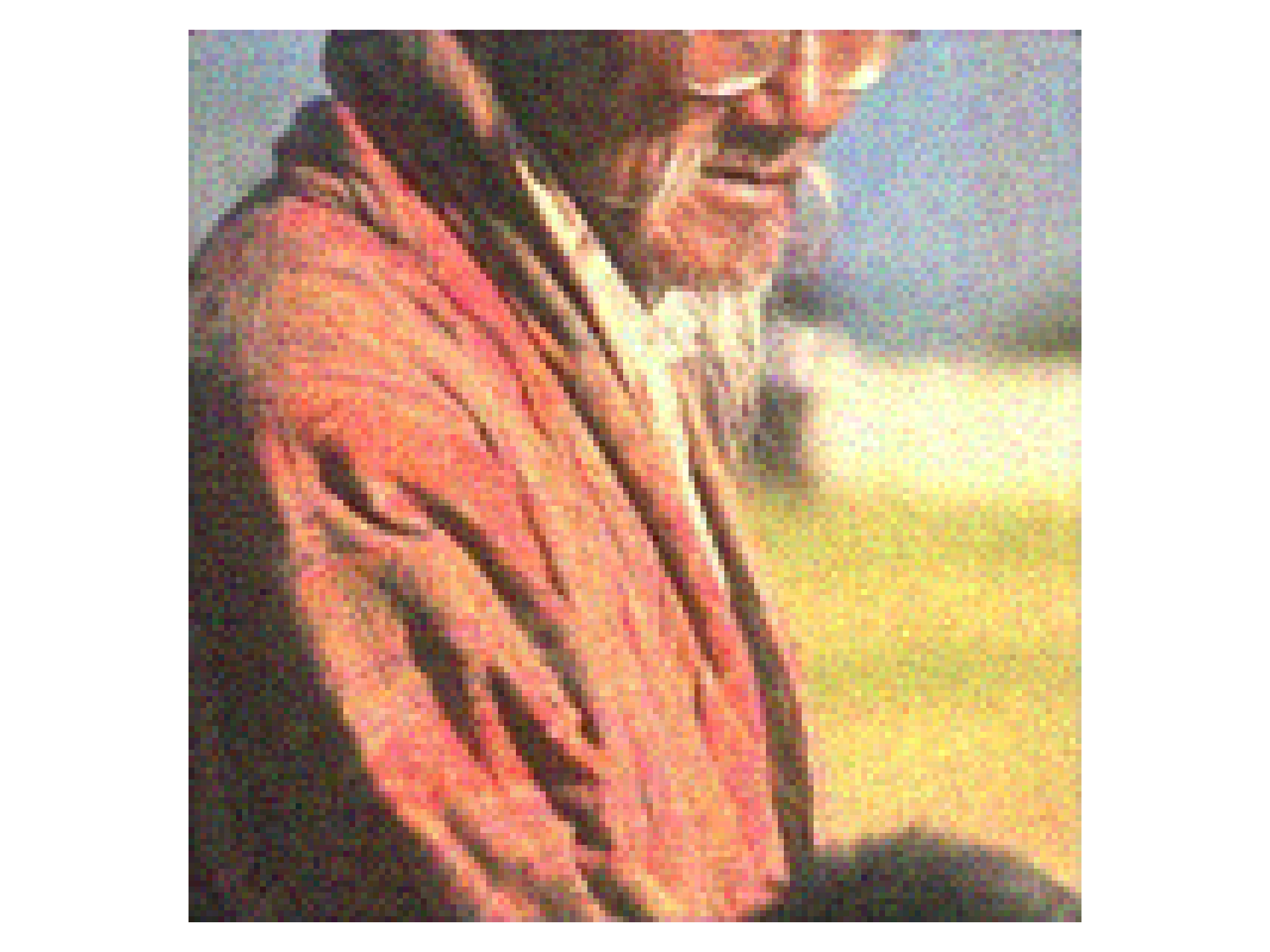}
        &  \includegraphics[width=2.3cm, trim={2.6cm 0.5cm 2.6cm 0.5cm}, clip]{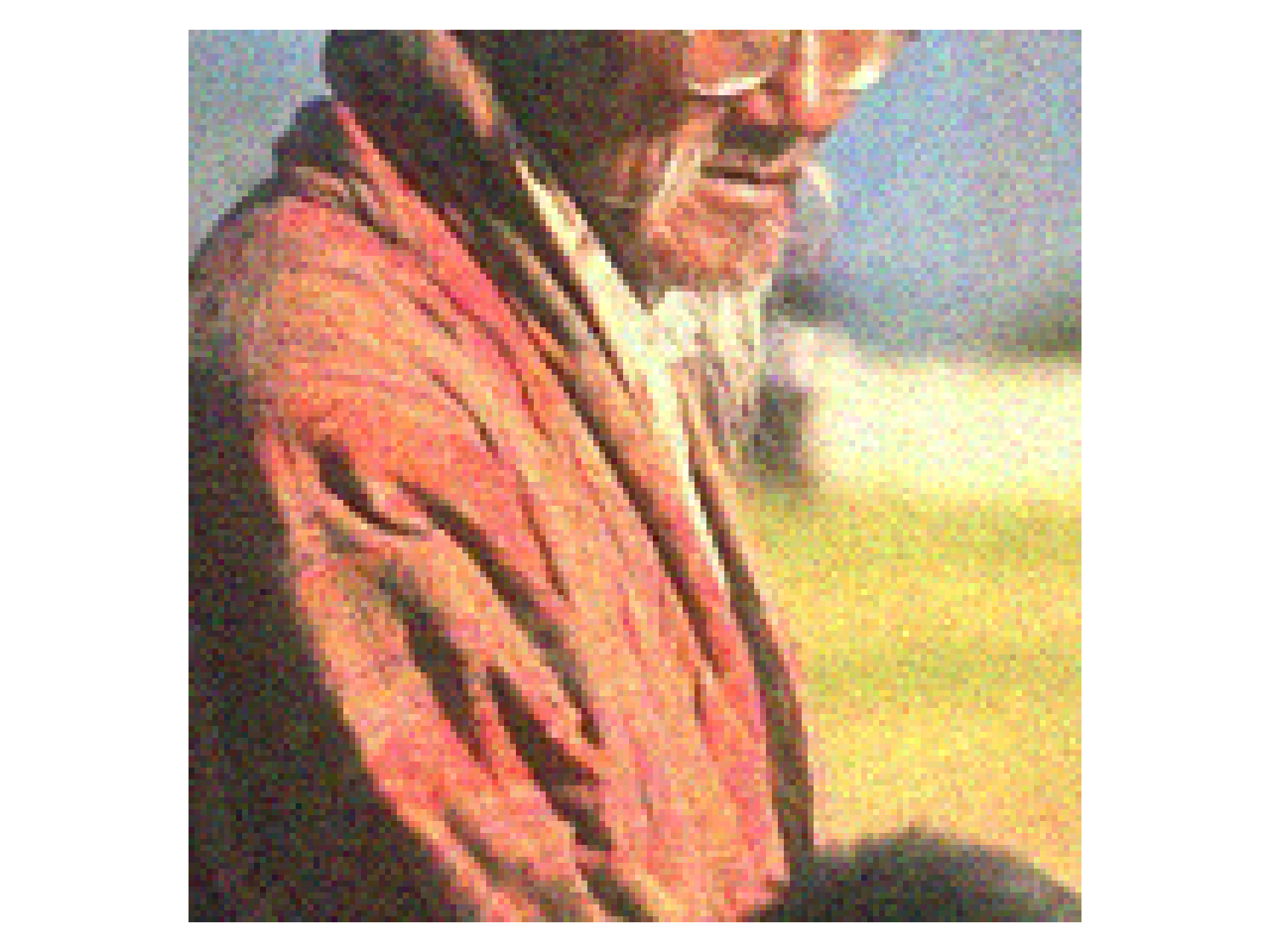}
        &  \includegraphics[width=2.3cm, trim={2.6cm 0.5cm 2.6cm 0.5cm}, clip]{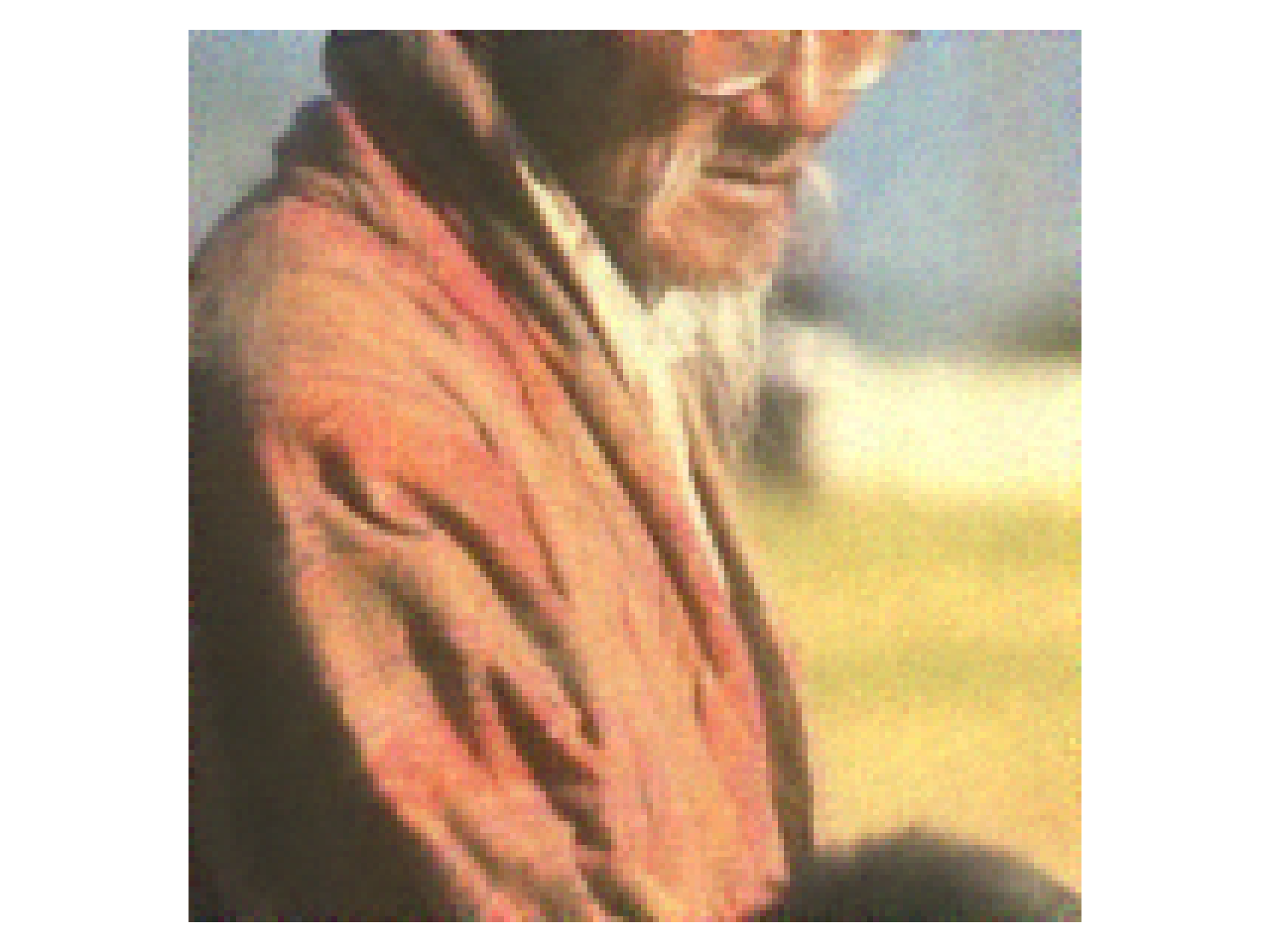}
        &  \includegraphics[width=2.3cm, trim={2.6cm 0.5cm 2.6cm 0.5cm}, clip]{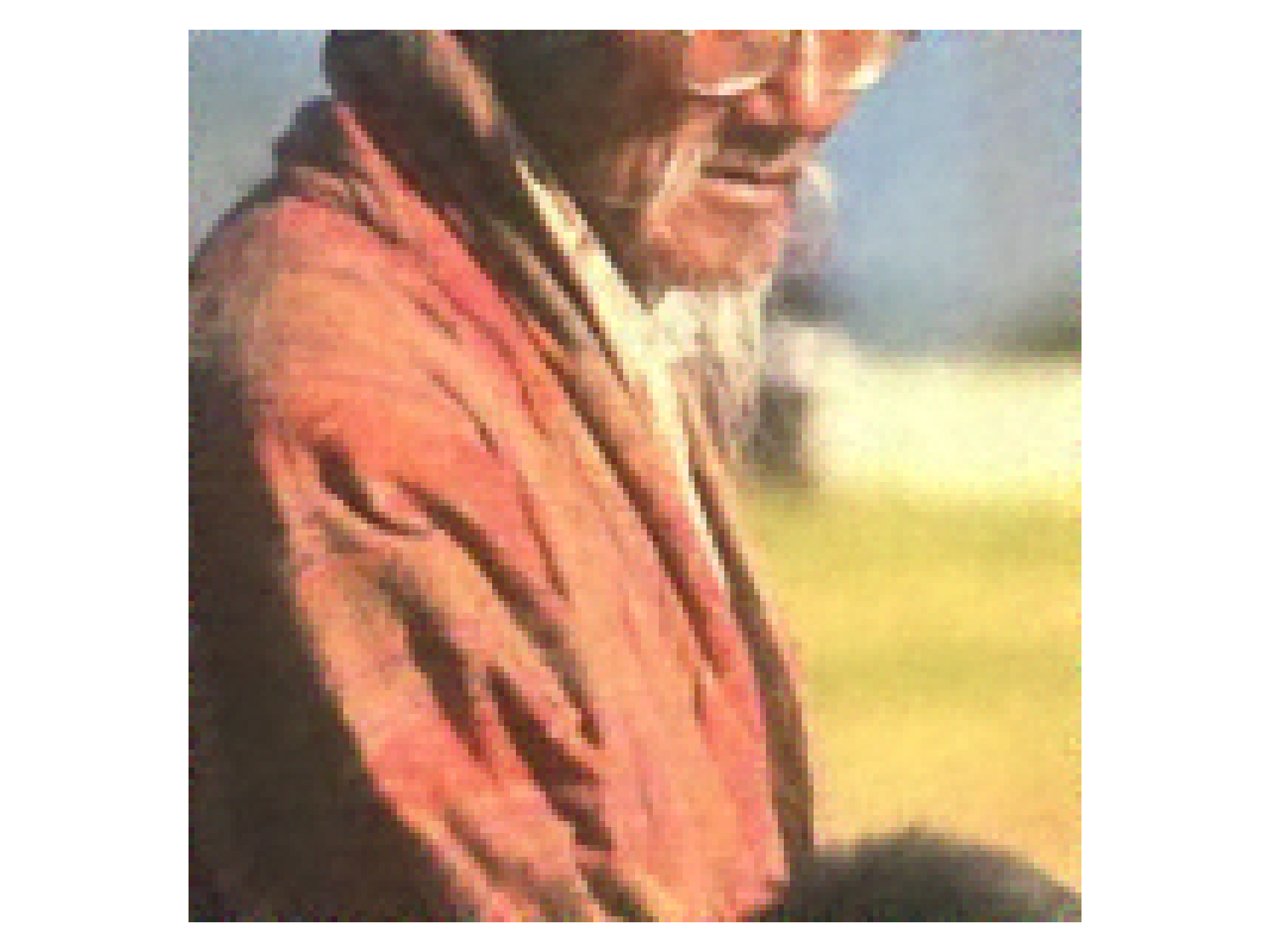}
    \end{tabular}
    
    \vspace{-0.2cm}

    \caption{Example of outputs obtained with the unfolded SD and AD for different configurations. 
    Left column shows original image $\overline{x}$, noisy image $v$ and denoised image obtained with DRUnet for comparison. 
    Other columns are as follows. 
    First (resp. second) row: Outputs of the SD (resp. AD) for different settings $(\Ltrain, \Lden) \in \{1, 20\}^2$.
    Last row: Denoised images obtained for the SD and the AD for the two training settings $\Ltrain \in \{1, 20\}$ and with $\Lden = 1,000$.
    }
    \label{fig:denoise-example}
\end{figure}

\begin{remark}\
\begin{enumerate}
    \item 
    Since unrolled algorithms have a finite number of iterations $L\in \Nset^*$, it can be interpreted as a neural network with $L$ layers, where dictionaries are layer dependent (i.e., $L$ different dictionaries to be learned).
    This approach has been proposed by~\cite{GL10} for the
    synthesis formulation~\eqref{eqn:opt_synthesis} to speed up the computation of $z^\dagger_{\Dics}$, where the authors unrolled iterations of FB~\cite{DDM04}, with $g$ being the $\ell^1$ norm. The
    resulting method is known as LISTA.
    The work of~\cite{GL10} has been further improved in multiple contributions, considering either a supervised setting~\cite{CLWW18, liu2019alista} or an unsupervised setting~\cite{AMMG19, MB17}. 
    \item 
    Unrolled-based networks adapted to the analysis formulation have also been studied~\cite{chambolle2020learning, CSM20, JP2020, le2022fast}. 
    Recently a few works have also proposed to use unfolded proximal analysis denoising networks in PnP algorithms~\cite{le2023pnn, repetti2022dual}. 
    These works have empirically shown that, despite having $10^2$ to $10^3$ less parameters than state-of-the-art denoising networks such as DnCNN~\cite{zhang2017beyond} and DRUnet~\cite{zhang2021plug}, unfolded denoising networks are more robust, especially when plugged in a FB algorithm for solving deblurring imaging problems. However, it is to be noted that in our approach, we are focusing on learning a dictionary, fixed for all the layers of the unrolled algorithm, leading to an even lighter architecture.
\end{enumerate}

\end{remark}

\subsubsection{Simulation results}
\begin{figure}
    \includegraphics[width=\textwidth]{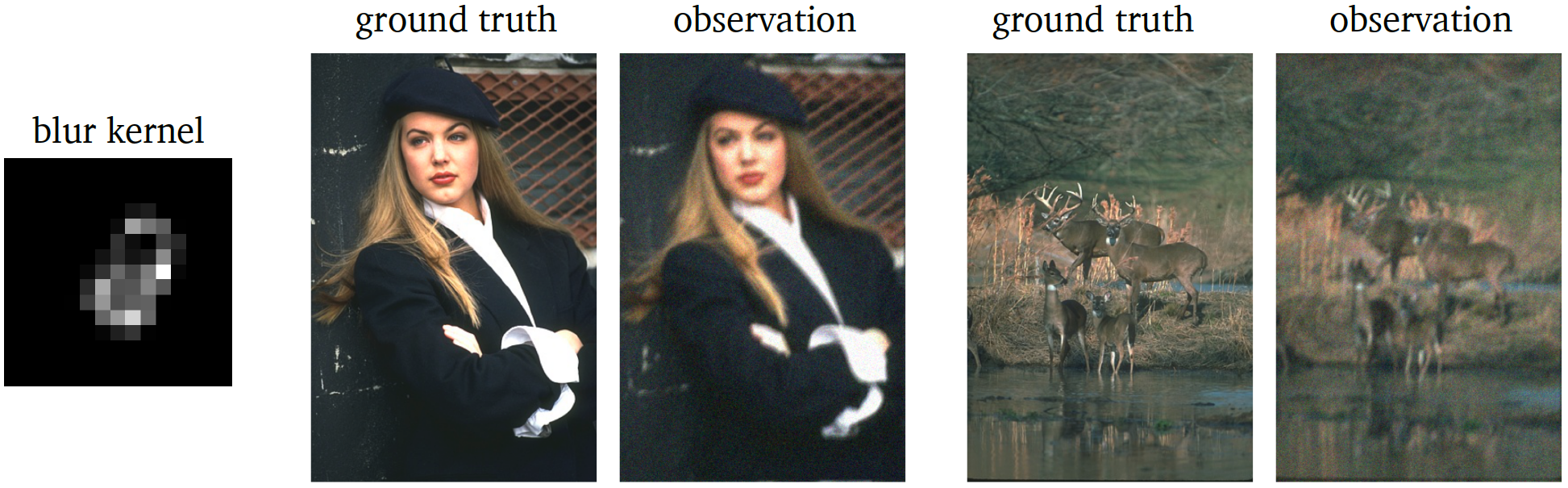}
    % \centering\footnotesize\setlength\tabcolsep{0.0cm}
    % \begin{tabular}{c}
    %     blur kernel \\
    %     \includegraphics[width=2cm]{figures/kernel.pdf}
    % \end{tabular}$\phantom{aaa}$
    % \begin{tabular}{cccc}
    %     ground truth & observation & ground truth & observation \\
    %     \includegraphics[width=2.7cm, trim={5cm 1cm 4.5cm 1.3cm},clip]{figures/img_2_gt.png}
    %     &   \includegraphics[width=2.7cm, trim={5cm 1cm 4.5cm 1.3cm},clip]{figures/img_2_observed.png} 
    %     &   $\phantom{aa}$\includegraphics[width=2.7cm, trim={5cm 1cm 4.5cm 1.3cm},clip]{figures/img_3_gt.png}
    %     &   \includegraphics[width=2.7cm, trim={5cm 1cm 4.5cm 1.3cm},clip]{figures/img_3_observed.png} \\
    % \end{tabular}

    \vspace{-0.3cm}
    
    \caption{\label{fig:kernel_blur}
    Blur kernel used in the experiments to model operator $A$ in~\eqref{pb:invpb}, and example of ground truth and associated observed images.
    }
\end{figure}

We consider an image deblurring problem of the form of~\eqref{pb:invpb}, where $A \colon \Rset^N \to \Rset^M$ models a convolution with kernel given in Figure~\ref{fig:kernel_blur} and $\varepsilon= 0.05$. 
Both the AD and SD are trained following the DDL procedure described in Section~\ref{Ssec:exp:ddl-train}. The resulting AD and SD are then plugged in the FB-PnP iterations and applied to the image deblurring problem. We will further compare the proposed approaches with an FB-PnP algorithm with a DRUnet denoiser \cite{zhang2021plug}, considered as state-of-the-art for PnP (see Section~\ref{ASec:drunet}).

\begin{figure}
    \centering
    \includegraphics[width=1\textwidth]{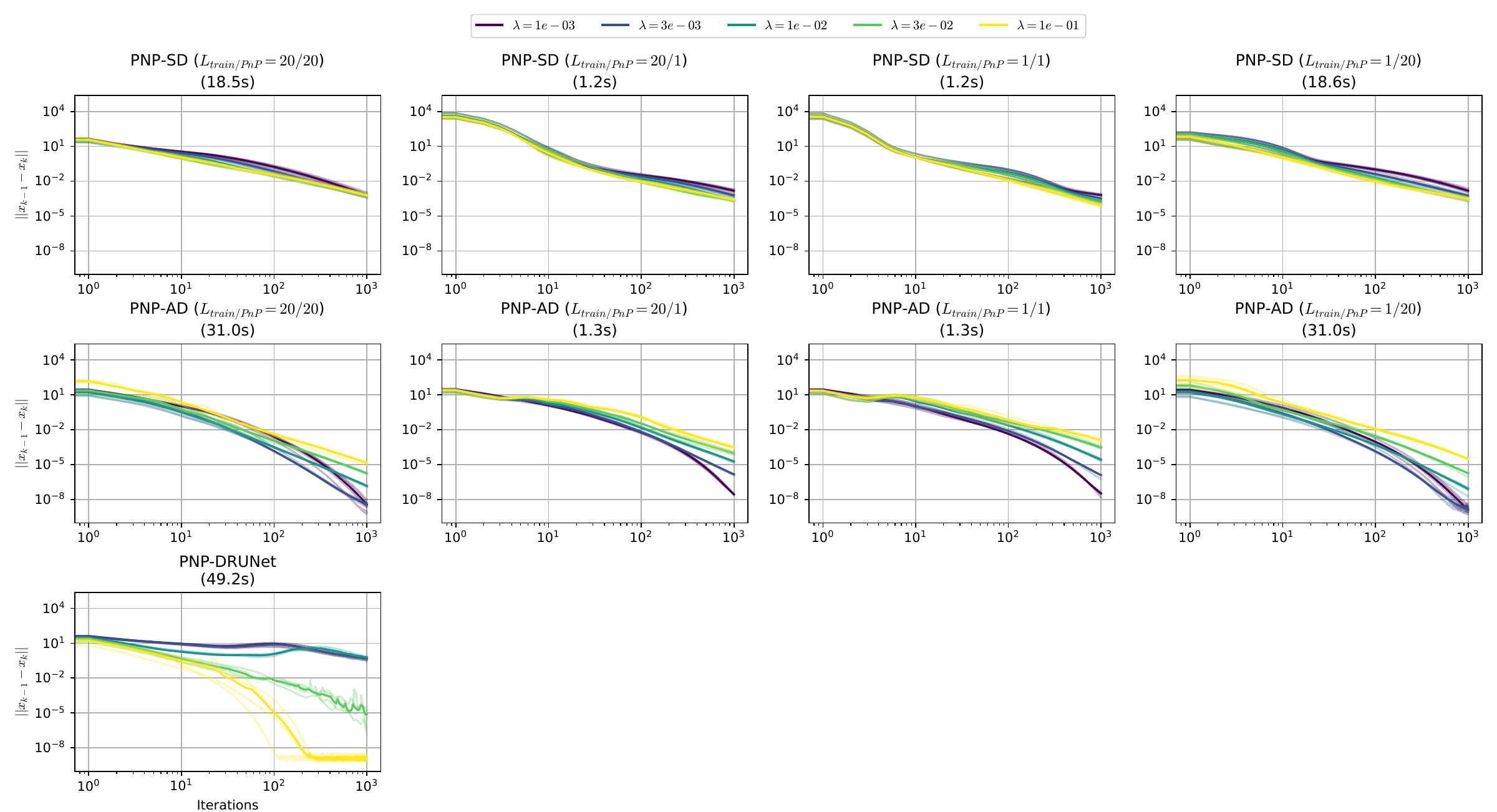}

    \caption{
    Convergence profiles with respect to iterations for the different FB-PnP algorithms. 
    AD and SD are trained with $\Ltrain\in\{1,20\}$, and implemented in the FB-PnP with either $\Lpnp\in\{1,20\}$.}
    \label{fig:cvgce}
\end{figure}
\begin{figure}
    \centering
    \includegraphics[width=\textwidth]{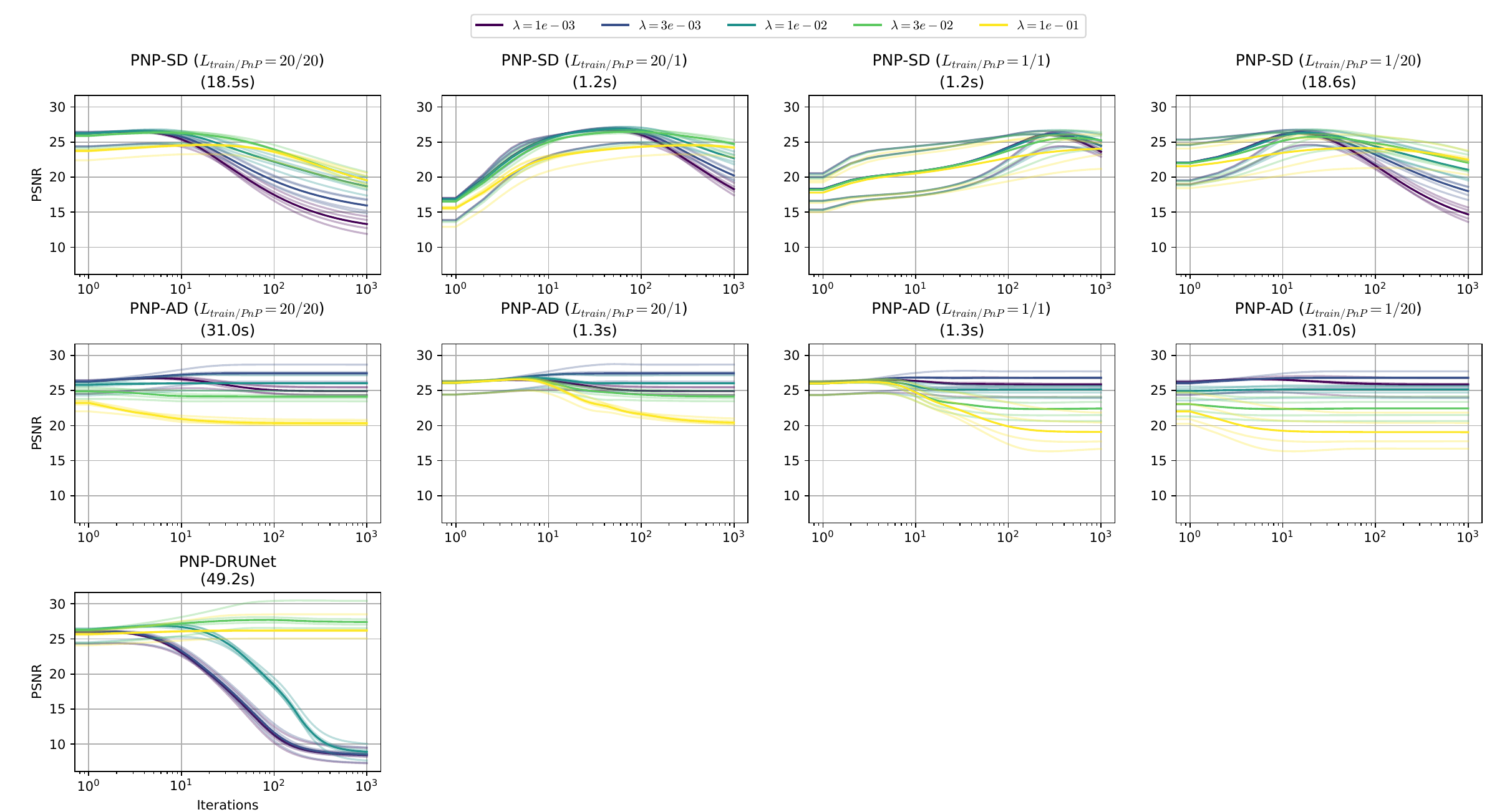}
    
    \caption{
    PSNR profiles with respect to iterations for the different FB-PnP algorithms. AD and SD are trained with $\Ltrain\in\{1,20\}$, and implemented in the FB-PnP with either $\Lpnp\in\{1,20\}$.}
    \label{fig:psnr}
\end{figure}

In Figure~\ref{fig:cvgce}, we compare convergence behaviors of the FB-PnP iterations $(\|x_{k+1} - x_k\|)_{k\in \mathbb N}$ for different denoisers, for different regularization parameters 
$\lambda\in \{10^{-3}, 3\time 10^{-3}, 10^{-2}, 3\time 10^{-2}, 10^{-1}\}$.
Each plot shows convergence curves when running the algorithm on $4$ different images.
The first (resp. second) row shows results obtained with SD (resp. AD), trained on different settings $\Ltrain \in \{1, 20\} $ and evaluated on different settings $\Lpnp \in \{1, 20\}$.
The third row in Figure~\ref{fig:cvgce} shows FB-PnP behaviour with DRUnet. 
FB-PnP algorithms with SD and AD exhibit monotonic convergence trends independently of the choice of the regularization parameter $\lambda$. This is not the case for DRUnet that seems to converge only for $\lambda=10^{-1}$, while other choices are not stable\footnote{For DRUnet $\lambda$ corresponds to the noise level given as an input to the network (see Section~\ref{ASec:drunet} for details). Additional simulations have been performed for DRUnet increasing $\lambda$, however these choices were observed to lead to lower PSNR values.
}. 
We further provide curves of PSNR values with respect to iterations in Figure~\ref{fig:psnr}, for the above mentioned FB-PnP iterations. 
It can be observed that for AD the PSNR profiles are very similar when fixing $\lambda$, and looking at the different strategies $(\Ltrain, \Lpnp) \in \{1, 20\}^2$. This suggests that the AD framework is stable, independently on the number of sub-iterations used for training or for the reconstruction process. 
This is not the case for SD. We can observe that the strategies $(\Ltrain, \Lpnp) = (1, 20)$ and $(\Ltrain, \Lpnp) = (20, 20)$ have similar trends, and strategies $(\Ltrain, \Lpnp) = (1, 1)$  and $(\Ltrain, \Lpnp) = (20, 1)$ have similar trends too. The last two strategies seems also to convergence slower. This was to be expected given theirs denoising performances shown in Figure~\ref{fig:denoise-example} for $(\Ltrain, \Lden) = (1, 1)$ and $(\Ltrain, \Lden) = (20, 1)$. 
DRUnet exhibits very unstable behaviour for $\lambda \in \{10^{-3}, 3\times 10^{-3}, 10^{-2}\}$, as expected given Figure~\ref{fig:cvgce}. More stable results are for $\lambda=10^{-1}$. The choice $\lambda=3\times 10^{-2}$ seems to give slightly higher PSNR values, although the associated convergence profiles in Figure~\ref{fig:cvgce} are not very stable (no monotony as per convergence guarantees).

\begin{figure}
    \includegraphics[width=\textwidth]{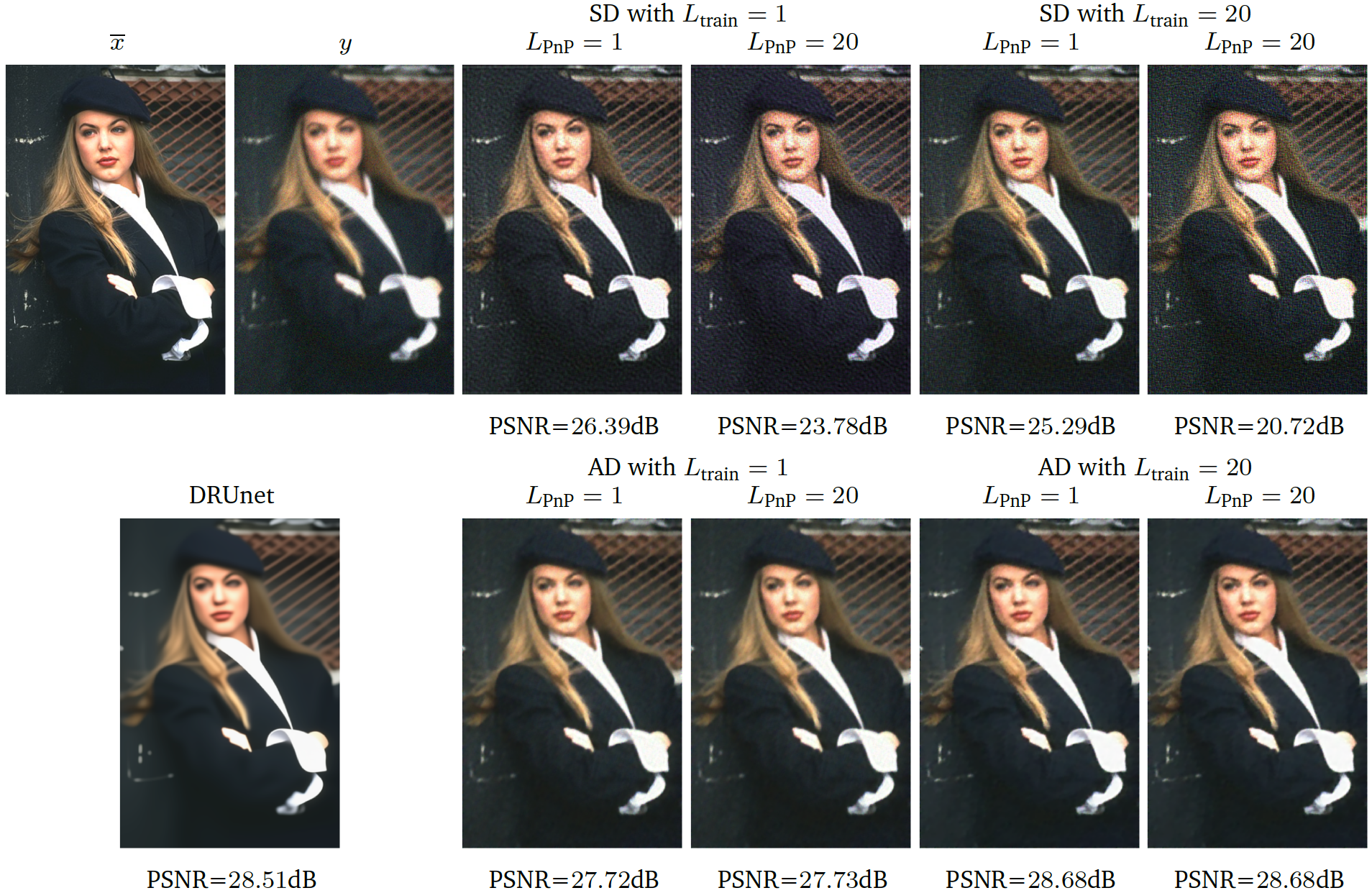}\\[0.2cm]
    \includegraphics[width=\textwidth]{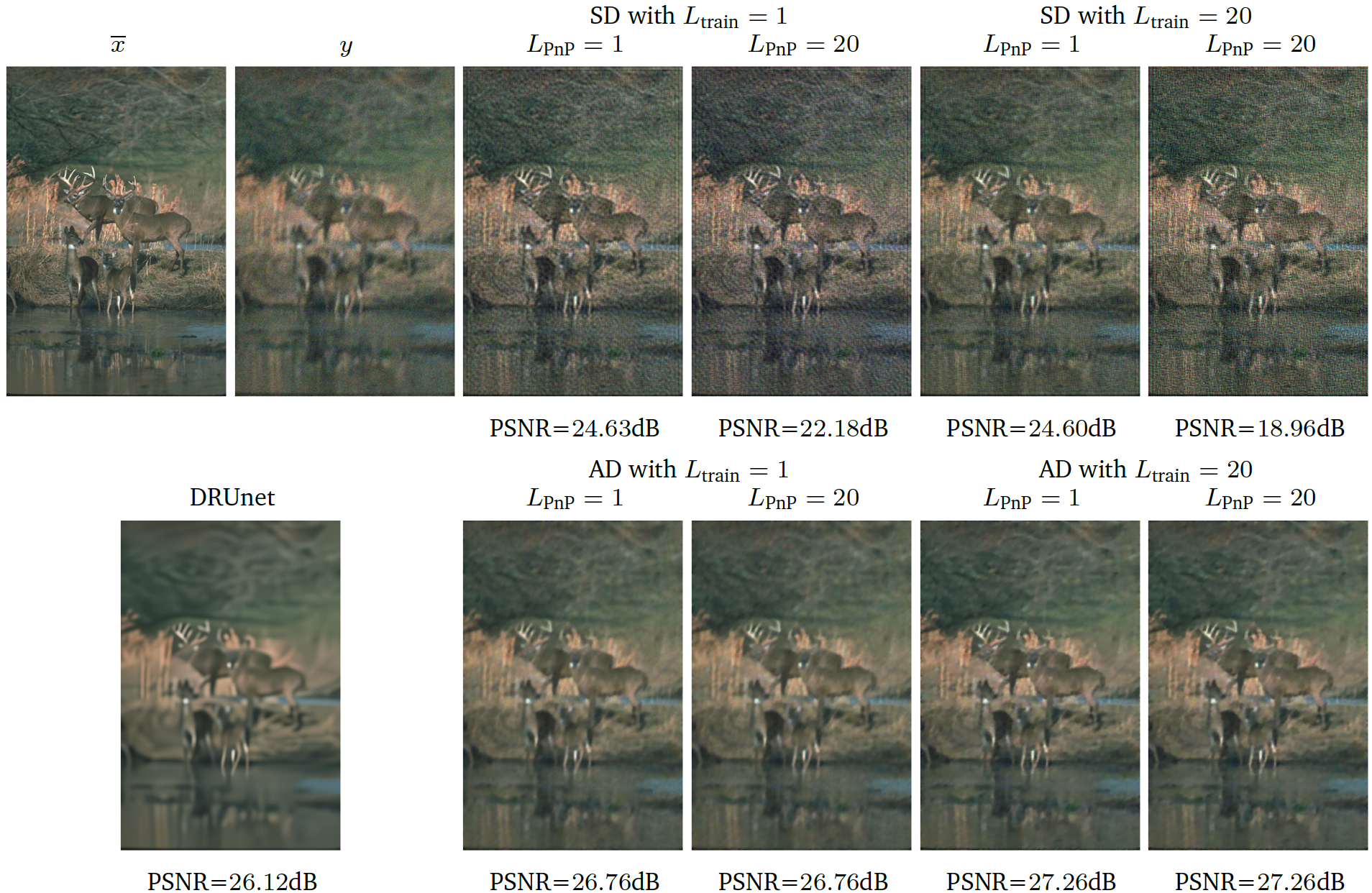}

    \caption{Examples of FB-PnP outputs obtained with the different denoisers tested, for two image examples. Regularization parameters are chosen according to convergence profiles given in Figure~\ref{fig:cvgce} and Figure~\ref{fig:psnr}: 
    $\lambda_{\textsf{SD}}=3\times 10^{-2}$ for $(\Ltrain,\Lpnp) \in \{(1,1), (20,1), (20,20\}$, 
    $\lambda_{\textsf{SD}}=10^{-2}$ for $(\Ltrain,\Lpnp) = (1,20) = 10^{-10}$,
    $\lambda_{\textsf{AD}}=3 \times 10^{-3}$ for all $(\Ltrain,\Lpnp)$ 
    and $\lambda_{\textsf{DRUnet}}=10^{-1}$.}
    \label{fig:reconstruction-examples}
\end{figure}

Additionally, we give in Figure~\ref{fig:reconstruction-examples} reconstruction results for two images. 
The regularization parameters are chosen according to Figure~\ref{fig:cvgce} and Figure~\ref{fig:psnr}. For AD we chose $\lambda=3 \times 10^{-3}$ that achieves best results for all testing configurations. For SD we chose $\lambda=3\times 10^{-2}$ for $(\Ltrain,\Lpnp) \in \{(1,1), (20,1), (20,20\}$, and $\lambda_{\textsf{SD}}=10^{-2}$ for $(\Ltrain,\Lpnp) = (1,20) = 10^{-10}$. 
Finally, for DRUnet we take $\lambda=10^{-1}$ that is the only value for which DRUnet is stable\footnote{Note that $\lambda=3\times 10^{-2}$ gives better PSNR values for DRUnet, however to the price of important artifact in the reconstruction due to instability of the PnP iterations (see also Figure~\ref{fig:reconstructed-lambda}).}. 
Overall it can be observed that AD leads to the best reconstructions, with any configuration $(\Ltrain, \Lden) \in  \{1,20\}^2$. Further, for fixed dictionary, doing $\Lpnp=1$ or $\Lpnp=20$ sub-iterations leads to the same reconstructions in our examples. This observation is not true for SD due to early stopping of the algorithms. In particular, SD produces artifacts at the edges of the images for both examples, even when high PSNR values are achieved.

\begin{figure}
    \includegraphics[width=\textwidth]{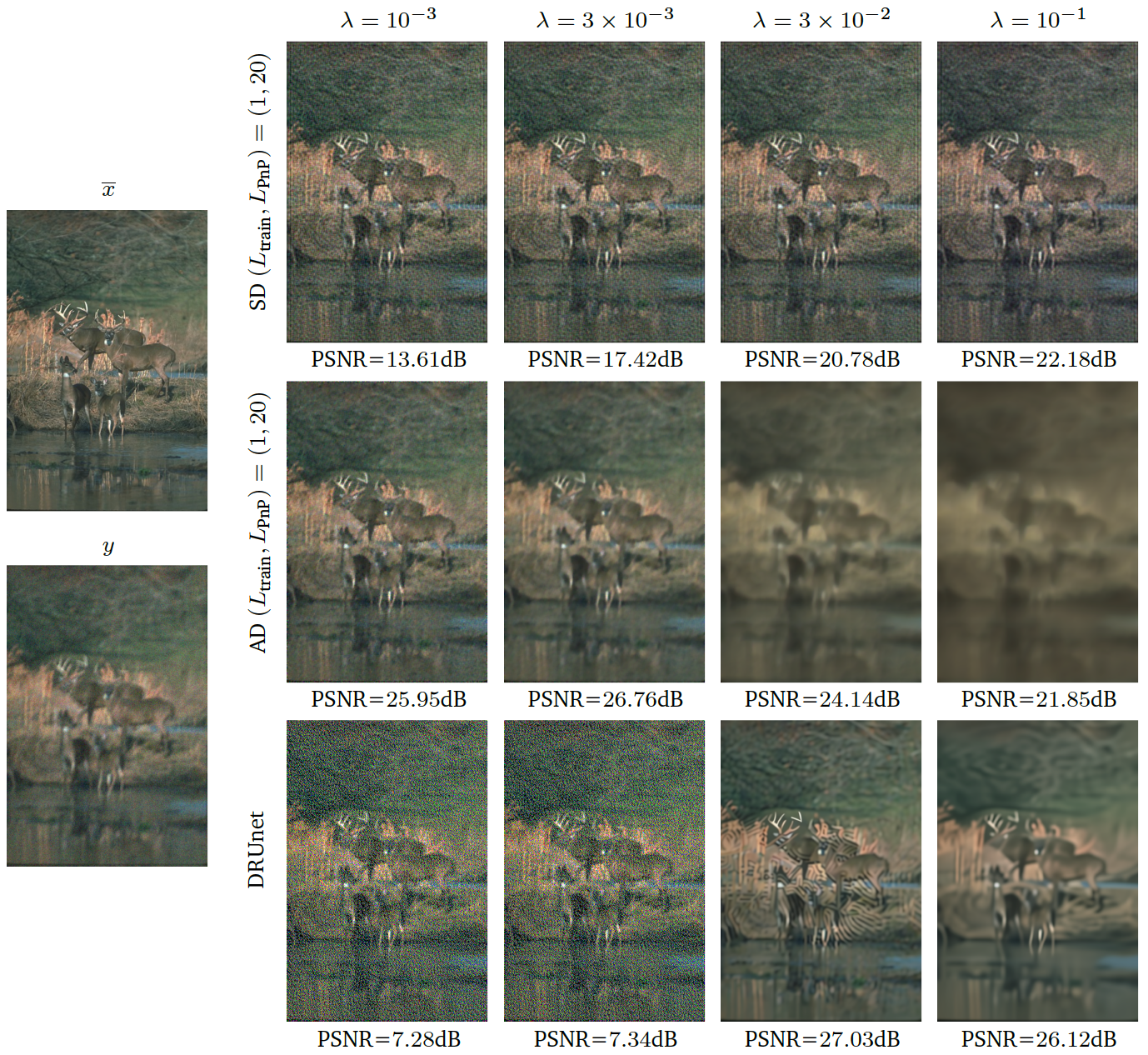}
    \caption{Example of FB-PnP outputs obtained with SD for $(\Ltrain,\Lpnp)=(1,20)$ (top row), AD for $(\Ltrain,\Lpnp)=(1,20)$ (middle row) and DRUnet (bottom row) when varying the regularization parameter $\lambda\in \{10^{-3}, 3 \times 10^{-3}, 3 \times 10^{-2}, 10^{-1}\}$.}
    \label{fig:reconstructed-lambda}
\end{figure}

Finally, for the sake of completeness, we provide in Figure~\ref{fig:reconstructed-lambda} reconstructed images obtained with the FB-PnP when varying $\lambda\in \{10^{-3}, 3\times 10^{-3}, 3\times 10^{-2}, 10^{-1}\}$. We show images obtained with SD and AD for $(\Ltrain,\Lpnp)=(1,20)$ (top and middle rows), and with DRUnet (bottom row). Interestingly, with visual inspection, it seems that SD is less sensitive to the choice of $\lambda$ than AD and DRUnet, while AD is the more stable with respect to PSNR (despite very blurry reconstructions for large values of $\lambda$).

% --------------------------------------------------------------------
% --------------------------------------------------------------------
\section{Conclusion}
In this work we study convergence behavior of the FB-PnP algorithm where the proximity operator is replaced by a specific unfolded denoiser. In particular, we study two strategies: when the denoising problem is solved with an analysis formulation or when it is solved with a synthesis formulation. The two resulting analysis and the synthesis denoising minimization problems can then be solved with sub-iterations based on the dual FB algorithm and the FB algorithm, respectively.
For both strategies, we show that computing only one sub-iteration within the FB-PnP combined with a warm-restart procedure on the sparse coefficients is equivalent to the FB-PnP algorithm when the full analysis or the synthesis denoising minimization problems are solved at each iteration. We further show that when smoothing the denoising problem using a Moreau envelope, under theoretical conditions, this equivalence extends to an arbitrary number of sub-iterations. 
Finally, we investigated how these different strategies compare on a toy compressive sensing problem as well as on a deblurring imaging problem in the context of DDL. 
Interestingly, we observed that training the AD with only $1$ iteration ($\Ltrain=1$) or a larger number of iterations ($\Ltrain=20$) does not seem to change much the denoising performances (see Figure~\ref{fig:denoise-example}). Further, for the FB-PnP, training with only $\Ltrain=1$ iterations for both the SD and the AD leads to the best performances (see Figure~\ref{fig:reconstruction-examples}). This suggests that choosing $\Ltrain=1$ enables a fast and efficient denoiser training. In particular, AD with $\Ltrain=1$ appeared to lead to the best results overall in our experiments.

\bibliographystyle{siamplain}
\bibliography{references}

% --------------------------------------------------------------------
% --------------------------------------------------------------------

\appendix

\section{Loris-Verhoeven primal-dual algorithm}
\label{ASec:lvpda}

With our notation, the scaled Loris-Verhoeven primal-dual algorithm proposed in~\cite[Section~6, eq.~(40)]{loris2011generalization} is designed to 
$\text{find } \widehat{x} \in \Argmind{x \in \Rset^N} \frac12 \| Ax - y \|^2 + g_\lambda(\Dica x)$,
% \begin{equation*}
%     \text{find } \widehat{x} \in \Argmind{x \in \Rset^N} \frac12 \| Ax - y \|^2 + g_\lambda(\Dica x),
% \end{equation*}
which is the same problem as stated in~\eqref{pb:inv-min-AF}.
The iterations to solve this problem are given by
\begin{equation*}
\begin{array}{l}
    \text{for } k = 0, 1, \ldots \\
    \left\lfloor
    \begin{array}{l}
        \widetilde{x}_{k+1} = x_k - \tau A^*(Ax_k - y) - \tau \Dica^* \widetilde{u}_k \\
        \widetilde{u}_{k+1} = \prox_{\frac{\sigma}{\tau} g_\lambda^*}( \widetilde{u}_k + \frac{\sigma}{\tau} \Dica \widetilde{x}_{k+1} ) \\
        x_{k+1} = x_k - \tau A^*(Ax_k - y) - \tau \Dica^* \widetilde{u}_{k+1},
    \end{array}
    \right.
\end{array}
\end{equation*}
with $\tau<2/\|A^*A\|$ and $\sigma<1/\|\Dica^*\Dica\|$.

\section{FB-PnP with DRUnet}
\label{ASec:drunet}

DRUnet is a \textit{noise-aware} denoising network that has been proposed in~\cite{zhang2021plug}. It has a U-Net like structure, with convolutional blocks in the encoder and decoder parts, and it takes into account the noise level of the input image as an extra input channel. 
It is considered as a state-of-the-art denoising network to use in PnP frameworks, and trained DRUnet can be found on PyTorch-based libraries such as \href{https://deepinv.github.io/deepinv/user_guide/reconstruction/denoisers.html}{DeepInv}.
Then, if $G \colon \Rset^N \times \Rset^N \to \Rset^N$ denotes the DRunet operator, the corresponding FB-PnP iterations are given as
\begin{equation}    \label{algo:FB-PnP-drunet}
    \begin{array}{l}
        % x_0 \in \mathbb R^N \\
       \text{for } k = 0, 1, \ldots \\
       \left\lfloor
       \begin{array}{l}
        x_{k+1} = G \big( x_k - \tau A^*(A x_k - y), \, \lambda \tau \operatorname{Id}_N \big),
       \end{array}
       \right.
    \end{array}
\end{equation}
where $\lambda>0$ denotes the input noise level that acts as a regularization parameter on the network, and $\tau>0$ is the step-size chosen to ensure convergence of $(x_k)_{k\in \mathbb N}$.

\end{document}